\documentclass[10pt,reqno]{amsart}

\usepackage{mathtools}

\usepackage[T1]{fontenc}

\usepackage{amsmath,amsfonts,amsbsy,amsgen,amscd,mathrsfs,amssymb,amsthm}
\usepackage{enumerate}
\usepackage{bm}

\usepackage{stmaryrd}
\SetSymbolFont{stmry}{bold}{U}{stmry}{m}{n} % suppress bm warning

\usepackage[usenames,dvipsnames]{xcolor}
\usepackage[colorlinks=true,citecolor=blue,linkcolor=blue]{hyperref}

\usepackage{tikz}

\newtheorem{thm}{Theorem}[section]
\newtheorem{lem}[thm]{Lemma}

\newtheorem{prop}[thm]{Proposition}
\newtheorem{cor}[thm]{Corollary}

\theoremstyle{definition}

\newtheorem{defn}[thm]{Definition}

\theoremstyle{remark}

\newtheorem{example}[thm]{Example}

\newtheorem{rem}[thm]{Remark}
\newtheorem*{rem*}{Remark}

\numberwithin{equation}{section} 
\numberwithin{figure}{section}
\numberwithin{table}{section}

\makeatletter

\let\oldtocsection=\tocsection
\let\oldtocsubsection=\tocsubsection
\let\oldtocsubsubsection=\tocsubsubsection
\renewcommand{\tocsection}[2]{\hspace{-1.2em}\oldtocsection{#1}{#2}}
\renewcommand{\tocsubsection}[2]{\hspace{-.2em}\oldtocsubsection{#1}{#2}}
\renewcommand{\tocsubsubsection}[2]{\hspace{0.8em}\oldtocsubsubsection{#1}{#2}}

\DeclareRobustCommand{\gobblefive}[5]{}
\newcommand*{\SkipTocEntry}{\addtocontents{toc}{\gobblefive}}

\makeatother

\newcommand*\centermathcell[1]{\omit\hfil$\displaystyle#1$\hfil\ignorespaces}

\newcommand{\ntr}{\mathop{\mathrm{tr}}}
\newcommand{\tr}{\mathop{\mathrm{Tr}}}

\newcommand{\M}{\mathrm{M}}
\newcommand{\EE}{\mathbf{E}}
\newcommand{\PP}{\mathbf{P}}
\newcommand{\spc}{\mathrm{sp}}
\newcommand{\id}{\mathbf{1}}
\newcommand{\svl}{\mathrm{s}}

\begin{document}

\title[Matrix concentration and free probability]{Matrix concentration 
inequalities\\and free probability}

\author{Afonso S.\ Bandeira}
\address{Department of Mathematics, ETH Z\"urich, Switzerland}
\email{bandeira@math.ethz.ch}

\author{March T.\ Boedihardjo}
\address{Department of Mathematics, ETH Z\"urich, Switzerland}
\email{march.boedihardjo@ifor.math.ethz.ch}

\author{Ramon van Handel}
\address{Fine Hall 207, Princeton University, Princeton, NJ 08544, USA}
\email{rvan@math.princeton.edu}

\begin{abstract}
A central tool in the study of nonhomogeneous random matrices, the 
noncommutative Khintchine inequality, yields a nonasymptotic bound on the 
spectral norm of general Gaussian random matrices $X=\sum_i g_i A_i$ where 
$g_i$ are independent standard Gaussian variables and $A_i$ are matrix 
coefficients. This bound exhibits a logarithmic dependence on dimension 
that is sharp when the matrices $A_i$ commute, but often proves to be 
suboptimal in the presence of noncommutativity.
In this paper, we develop nonasymptotic bounds on the spectrum of 
arbitrary Gaussian random matrices that can capture noncommutativity. 
These bounds quantify the degree to which the spectrum of $X$ is captured 
by that of a noncommutative model $X_{\rm free}$ that arises from free 
probability theory. This ``intrinsic freeness'' 
phenomenon provides a powerful tool for the study of various questions 
that are outside the reach of classical methods of random matrix theory. 
Our nonasymptotic bounds are easily applicable in concrete situations, and 
yield sharp results in examples where the noncommutative Khintchine 
inequality is suboptimal. When combined with a linearization argument, our 
bounds imply strong asymptotic freeness for a remarkably general class of 
Gaussian random matrix models that may be very sparse, have dependent 
entries, and lack any special symmetries.
When combined with a universality principle, our bounds extend beyond the 
Gaussian setting to general sums of independent random matrices.
\end{abstract}

 % RMT: algebraic aspects


\maketitle

\thispagestyle{empty}
\iffalse
\setcounter{tocdepth}{2}
\tableofcontents
\fi

\section{Introduction}

The study of the spectrum of random matrices arises as a central problem 
in many areas of mathematics. Motivated by topics ranging from 
mathematical physics to operator algebras, much of classical random matrix 
theory is concerned with the study of highly homogeneous matrix ensembles, 
such as those with i.i.d.\ entries or that are invariant under symmetry 
groups. Deep results obtained over the past six decades by numerous 
mathematicians have resulted in a very detailed understanding of the 
asymptotic properties of such models \cite{AGZ10,Tao12}.

In contrast, many problems in areas such as functional analysis 
\cite{DS01,Pis03} and in applied mathematics \cite{Tro15,Ban15} fall 
outside the scope of classical random matrix theory. The random matrix 
models that arise in such problems possess two common features. On the one 
hand, such models are often highly nonhomogeneous and lack any natural 
symmetries. On the other hand, the type of questions that arise in these 
areas are generally nonasymptotic in nature, as the study of 
nonhomogeneous models often does not lend itself naturally to an 
asymptotic formulation.

The above considerations motivate the need for nonasymptotic methods that 
can capture the spectral properties of essentially arbitrarily structured 
nonhomogeneous random matrices. It may appear hopeless at first sight that 
anything at all can be said at this level of generality. Nonetheless, as 
we will recall below, there exists a set of tools, known colloquially as 
``matrix concentration inequalities'', that makes it possible to compute 
certain spectral statistics of very general nonhomogeneous random matrices 
up to logarithmic factors in the dimension. The results of this paper 
provide a powerful refinement of this theory that makes it possible to 
achieve sharp results in many situations that are outside the reach of 
classical methods.

\subsection{Matrix concentration inequalities}

As a guiding motivation for this paper, consider the problem of 
estimating the spectral norm (i.e., largest singular value) of an 
arbitrary $d\times d$ self-adjoint random matrix with centered 
jointly Gaussian entries. Any such matrix $X$ can be represented as
\begin{equation}
\label{eq:simplest}
	X = \sum_{i=1}^n g_i A_i,
\end{equation}
where $A_i\in \M_d(\mathbb{C})_{\rm sa}$ are deterministic 
self-adjoint $d\times d$ 
matrices
and $g_i$ are i.i.d.\ standard real Gaussian variables. As was 
noted in \cite{Rud99}, the noncommutative Khintchine inequality of 
Lust-Piquard and Pisier \cite[\S 9.8]{Pis03} implies that\footnote{%
	We write $x\lesssim y$ if $x\le Cy$ for
	a universal constant $C$, and 
	$x\asymp y$ if $x\lesssim y$ and $y\lesssim x$.}
\begin{equation}
\label{eq:nck}
	\sigma(X)
	\lesssim \mathbf{E}\|X\| 
	\lesssim 
	\sigma(X)\sqrt{\log d},
\end{equation}
where we define
\begin{equation}
\label{eq:sigma}
	\sigma(X)^2 = \|\mathbf{E}X^2\|=
	\Bigg\| \sum_{i=1}^n A_i^2\Bigg\|.
\end{equation}
Thus the expected spectral norm of any Gaussian random matrix can be 
explicitly computed up to a logarithmic factor in the dimension. 

It should be emphasized that \eqref{eq:simplest} is an extremely general 
model: no assumption is made on the covariance of the entries of $X$, so 
that the model can capture arbitrary variance profiles and dependencies 
between the entries. Analogues of \eqref{eq:nck} extend even further to 
the model $X=\sum_i Z_i$ where $Z_i$ are arbitrary independent random 
matrices. Due to their generality and ease of use, these ``matrix 
concentration inequalities'' \cite{Tro15} have had a major impact on 
numerous applications. On the other hand, the utility of \eqref{eq:nck} is 
limited by the gap between the upper and lower bounds, which becomes 
increasingly severe in high dimension.

To understand the origin of this gap, it is instructive to recall the 
basic principle behind the proofs of almost all known matrix concentration 
inequalities: the norm of a random matrix is largest when the 
coefficients $A_i$ commute. This idea arises clearly in proofs of 
these inequalities \cite{Tro15,Tro16,vH17}: the key step is  
application of trace inequalities that permute the order of the matrices 
$A_i$, which become equalities when all $A_i$ commute. In the latter case, 
the \emph{upper} bound of \eqref{eq:nck} is typically of the correct 
order. 
Indeed, by simultaneously diagonalizing $A_i$, we may assume $X$ is a 
diagonal matrix. Then $\sigma(X)^2=\|\mathbf{E}X^2\|=\max_i 
\mathrm{Var}(X_{ii})$, while 
$$
	\mathbf{E}\|X\|=\mathbf{E}\max_i|X_{ii}|\asymp 
	\sigma(X)\sqrt{\log d}
$$
under mild assumptions (as the maximum of $d$ Gaussian variables is 
typically of order $\sqrt{\log d}$,
see, e.g., \cite[\S 3.3]{LT91}). On the other 
hand, when the coefficients $A_i$ do not commute, it is observed in many 
examples that it is the \emph{lower} bound of \eqref{eq:nck} that is of 
the correct order. This is already the case for the most basic model of 
random matrix theory: when $X$ has i.i.d.\ standard Gaussian entries 
$X_{ij}$ for $i\ge j$, it is classical that $\mathbf{E}\|X\| \asymp 
\sqrt{d}=\sigma(X)$ \cite[\S 2.3]{Tao12}.

Such examples raise the tantalizing question whether there exists a 
refinement of \eqref{eq:nck} that can capture the correct behavior of 
nonhomogeneous random matrices beyond the commutative case. To date, a 
satisfactory answer to this question has been obtained only in the
special case that $X$ has \emph{independent} entries $X_{ij}$ for $i\ge 
j$ with arbitrary variances $\mathrm{Var}(X_{ij})=b_{ij}^2$. In this 
case, \cite{BvH16} showed that
\begin{equation}
\label{eq:bvh}
	\mathbf{E}\|X\|\lesssim \sigma(X)
	+ \max_{ij}|b_{ij}|\,\sqrt{\log d},
	\quad\qquad
	\sigma(X)^2 = \max_i \sum_j b_{ij}^2,
\end{equation}
which can be reversed under mild assumptions. The key feature of
\eqref{eq:bvh} is that the dimensional factor enters here
through a smaller parameter $\sigma_*(X)=\max_{ij}|b_{ij}|$ that 
controls which extreme case of \eqref{eq:nck} dominates: 
diagonal matrices satisfy $\sigma_*(X)=\sigma(X)$, in which case we 
recover the 
upper bound 
of \eqref{eq:nck}; but as soon as $\sigma_*(X) \lesssim (\log 
d)^{-\frac{1}{2}}\sigma(X)$, the lower bound 
of \eqref{eq:nck} is of the correct order.

The existence of the bound \eqref{eq:bvh} hints at the possibility that an 
analogous refinement of \eqref{eq:nck} might hold even in the setting of 
general Gaussian random matrices \eqref{eq:simplest}. In particular, one 
may conjecture the existence of a general bound
\begin{equation}
\label{eq:secondorder}
	\mathbf{E}\|X\| \,\stackrel{?}{\lesssim}\,
	\sigma(X) + \sigma_{**}(X)(\log d)^\beta
\end{equation}
for some $\beta>0$, where $\sigma(X)$ is as in \eqref{eq:sigma} and
$\sigma_{**}(X)$ is a parameter that is small 
when the coefficients $A_i$ are far from being commutative. This question 
was first considered by Tropp \cite{Tro18}, who introduced a number of 
important ideas that form the basis for the present paper. Using these 
ideas, Tropp was able to prove a bound of the form \eqref{eq:secondorder} 
for a special class of models that satisfy strong symmetry assumptions 
(and for general models with a dimensional factor $(\log d)^{\frac{1}{4}}$ 
in the leading term). To date, however, a general bound of the form 
\eqref{eq:secondorder} has remained elusive.

\subsection{Free probability}
\label{sec:introfree}

The challenge in proving an inequality of the form \eqref{eq:secondorder} 
is to capture the intrinsic noncommutativity of the matrices $A_i$. There 
is however an entirely different way to introduce noncommutativity into 
\eqref{eq:simplest} that arises from Voiculescu's theory of free 
probability \cite{Voi91,NS06}: one may modify the model by replacing the 
scalar Gaussian coefficients $g_i$ by noncommuting random matrices or 
operators. When noncommutativity is externally introduced into 
\eqref{eq:simplest} in this manner, the dimensional factor in 
\eqref{eq:nck} is unnecessary regardless of the properties of the matrices 
$A_i$ (see \eqref{eq:freenck} below). However, on its face, this appears 
to shed little light on the behavior of the original model 
\eqref{eq:simplest}.

Remarkably, this intuition proves to be incorrect.
The central theme that 
will be developed in this paper is described informally by the following 
principle:
\vskip\abovedisplayskip
\begin{quote}
\emph{When the coefficient matrices $A_i$ are sufficiently noncommutative, 
the spectral statistics of the random matrix model 
$X=\sum_i g_iA_i$
are already accurately captured by free probability.}
\end{quote}
\vskip\abovedisplayskip
This ``intrinsic freeness'' phenomenon will prove to have far-reaching 
implications: it will enable us to prove nonasymptotic bounds of the form 
\eqref{eq:secondorder} in complete generality (both for Gaussian random 
matrices and for general sums of independent random matrices), and to 
develop new asymptotic results in free probability in far more general 
situations than are accessible by previous methods.

Before we can formulate precise results along these lines, we must briefly 
recall some relevant notions of free probability. We will use 
the following terminology.

\begin{defn}
\label{defn:wigner}
A \emph{standard Wigner matrix} of dimension $N$ is an $N\times N$ 
self-adjoint random matrix $G^N$ whose entries on and above the diagonal 
are independent real Gaussian variables with mean zero and variance 
$\frac{1}{N}$.
\end{defn}

Free probability provides an asymptotic description of the behavior 
of Wigner matrices as $N\to\infty$. Let $G_1^N,\ldots,G_n^N$ be 
independent standard Wigner matrices; the associated 
limiting objects are certain infinite-dimensional self-adjoint operators 
$s_1,\ldots,s_n$ that form a \emph{free semicircular family}, together 
with a trace $\tau$ acting on the algebra generated by these operators. We 
postpone the precise definitions of these objects to Section 
\ref{sec:free}; for our purposes, they may be viewed as an 
algebraic tool that allows us to compute spectral properties of large 
random matrices. In particular, a celebrated result of Voiculescu 
\cite{Voi91} states that
\begin{equation}
\label{eq:weakfree}
        \lim_{N\to\infty}
        \mathbf{E}[\ntr p(G_1^N,\ldots,G_n^N)]
        = \tau(p(s_1,\ldots,s_n))
\end{equation}
for any noncommutative polynomial $p$, where $\ntr(M):=\frac{1}{N}\tr(M)$ 
denotes the normalized trace of a matrix $M\in\M_N(\mathbb{C})$.
In an important paper, Haagerup and
Thorbj{\o}rnsen \cite{HT05} showed that the \emph{weak asymptotic 
freeness} property \eqref{eq:weakfree} may be considerably strengthened to 
obtain convergence in norm
\begin{equation}
\label{eq:strongfree}
	\lim_{N\to\infty}
	\mathbf{E}[\|p(G_1^N,\ldots,G_n^N)\|]
	=
	\|p(s_1,\ldots,s_n)\|
\end{equation}
for any noncommutative polynomial $p$. This \emph{strong asymptotic 
freeness} property has important applications in the 
theory of operator algebras \cite{HT05,HST06,HT99}.

A noncommutative analogue of the random matrix model
\eqref{eq:simplest} is obtained by replacing the scalar Gaussian 
coefficients $g_i$ by standard Wigner matrices:
\begin{equation}
\label{eq:xtens}
	X^N = \sum_{i=1}^n A_i\otimes G_i^N.
\end{equation}
When $N=1$, this model coincides with \eqref{eq:simplest}; however, as 
$N$ increases, the matrices $G_i^N$ become increasingly noncommutative.
The weak and strong asymptotic freeness properties \eqref{eq:weakfree} 
and \eqref{eq:strongfree} imply that the behavior of the spectrum of $X^N$ 
as $N\to\infty$ is captured by the infinite-dimensional operator
\begin{equation}
\label{eq:xfree}
	X_{\rm free} = \sum_{i=1}^n A_i \otimes s_i
\end{equation}
in that $\lim_{N\to\infty}\mathbf{E}\ntr[(X^N)^p]=
({\ntr}\otimes\tau)(X_{\rm free}^p)$
and $\lim_{N\to\infty}\mathbf{E}\|X^N\|=\|X_{\rm free}\|$. The study of 
such models plays a fundamental role in \cite{HT05}.

While $X_{\rm free}$ may be viewed abstractly as the limiting object 
associated to $X^N$, its considerable utility (from the perspective of 
this paper) is that it enables explicit computation of many spectral 
statistics of the random matrices $X^N$. For example, as we will recall in 
Section \ref{sec:mainspec}, the norm $\|X_{\rm free}\|$ admits an explicit  
formula in terms of the matrices $A_i$ \cite{Leh99} and admits 
the simple estimates \cite[p.\ 208]{Pis03}
\begin{equation}
\label{eq:freenck}
	\sigma(X)\le\|X_{\rm free}\|\le 2\sigma(X).
\end{equation}
Similarly, the limiting spectral distribution of $X^N$ may be computed by 
means of a (matrix-valued) Dyson equation as in classical random matrix 
theory \cite{HT05,AEK20}.

\subsection{Overview of main results}
\label{sec:intromain}

We now give a brief overview of the main results of this paper. A detailed 
presentation of our results will be given in Section \ref{sec:main}, while 
various examples that illustrate our results will be discussed in Section 
\ref{sec:ex}.

\subsubsection{Gaussian random matrices}
\label{sec:introg}

To illustrate the general principle described in Section 
\ref{sec:introfree}, let us begin by stating a special case of 
one of our main results.
For any centered $d\times d$ random matrix $X$ as in 
\eqref{eq:simplest}, we denote
by $\mathrm{Cov}(X)\in M_{d^2}(\mathbb{C})_{\rm sa}$ the covariance matrix 
of its $d^2$ scalar entries, that is,
$$
	\mathrm{Cov}(X)_{ij,kl} = \mathbf{E}[X_{ij} \overline{X_{kl}}]
	= \sum_{s=1}^n (A_s)_{ij} \overline{(A_s)_{kl}}
$$
which we view as a $d^2\times d^2$ positive semidefinite matrix. We now
define
$$
	v(X)^2 = \|\mathrm{Cov}(X)\| =
	\sup_{{\tr}|M|^2\le 1} \sum_{s=1}^n |{\tr}[A_s M]|^2.
$$
It should be far from apparent at this point that the parameter $v(X)$ 
captures noncommutativity of the matrices $A_i$; this will be explained in 
Section \ref{sec:overviewpf}. Note, for example, that $v(X)\asymp 
\max_{ij}|b_{ij}|$ in the setting of \eqref{eq:bvh} (cf.\ section 
\ref{sec:exind}).

\begin{thm}
\label{thm:intro}
For the model \eqref{eq:simplest} we have
$$
	\mathbf{E}\|X\| \le \|X_{\rm free}\| +
	C\,v(X)^{\frac{1}{2}}\sigma(X)^{\frac{1}{2}} 
	(\log d)^{\frac{3}{4}},
$$
where $X_{\rm free}$ is defined in \eqref{eq:xfree}
and $C$ is a universal constant.
\end{thm}

Using \eqref{eq:freenck} and Young's inequality, Theorem 
\ref{thm:intro} immediately implies a completely general bound of the form 
\eqref{eq:secondorder}:
\begin{equation}
\label{eq:introgood}
	\mathbf{E}\|X\| \lesssim \sigma(X) + v(X) (\log d)^{\frac{3}{2}}.
\end{equation}
However, Theorem~\ref{thm:intro} is much sharper in 
that its leading term captures the exact quantity predicted by free 
probability. In many cases, our results will make it possible to prove that 
$\mathbf{E}\|X\| = (1+o(1))\|X_{\rm free}\|$, that is, to compute the norm
exactly to leading order, as soon as 
$v(X)/\sigma(X)=o((\log d)^{-\frac{3}{2}})$.

Our main results for Gaussian random matrices (see
Sections \ref{sec:mainspec} and \ref{sec:mainstat}) are considerably more 
general than Theorem \ref{thm:intro}. In particular:
\begin{enumerate}[$\bullet$]
\itemsep\abovedisplayskip
\item Our main results are formulated for arbitrary Gaussian random 
matrices, which may have nonzero mean and may be non-self-adjoint.
\item We bound the support of the full spectrum 
$\spc(X)\subseteq \spc(X_{\rm free})+[-\varepsilon,\varepsilon]$ with
high probability, where
$\varepsilon\asymp v(X)^{\frac{1}{2}}\sigma(X)^{\frac{1}{2}}
(\log d)^{\frac{3}{4}}$.
\item We obtain nonasymptotic upper and lower bounds on the moments, 
resolvent, and other spectral statistics of $X$ in terms of $X_{\rm 
free}$.
\end{enumerate}
The ``intrinsic freeness'' phenomenon that is captured by these results 
has strong implications both for matrix concentation inequalities and in 
free probability.

\subsubsection{Asymptotic freeness}
\label{sec:introasympfree}

While our main results are nonasymptotic in nature, they give rise to 
remarkable new asymptotic results in free probability: when combined with 
the linearization trick of \cite{HT05}, our results establish strong 
asymptotic freeness \eqref{eq:strongfree} for a very large class of random 
matrix models. For example, we will prove the following result, as well as 
an analogous strong law (which yields a.s.\ convergence) that will be 
formulated in Section \ref{sec:mainsaf}.

\begin{thm}
\label{thm:introfree}
Let $s_1,\ldots,s_m$ be a free semicircular family.
For each $N\ge 1$, let $H_1^N,\ldots,H_m^N$ be independent self-adjoint 
random matrices of dimension $d=d(N)\ge N$ such that each $H_k^N$ has 
jointly Gaussian entries, $\mathbf{E}[H_k^N]=0$, and 
$\mathbf{E}[(H_k^N)^2]=\id$.
\begin{enumerate}[a.]
\itemsep\abovedisplayskip
\item
If $v(H_k^N)=o(1)$ as $N\to\infty$ for all $k$, then 
for any polynomial $p$
$$
	\lim_{N\to\infty} \mathbf{E}[\ntr p(H_1^N,\ldots,H_m^N)] =
	\tau(p(s_1,\ldots,s_m)).
$$
\item If $v(H_k^N)=o((\log d)^{-\frac{3}{2}})$ as $N\to\infty$ for all $k$, 
then for any polynomial $p$
$$
	\lim_{N\to\infty} \mathbf{E}[\|p(H_1^N,\ldots,H_m^N)\|] =
	\|p(s_1,\ldots,s_m)\|.
$$
\end{enumerate}
\end{thm}

A striking consequence of Theorem \ref{thm:introfree} is the unexpected 
ubiquity of the strong asymptotic freeness property. To date, strong 
asymptotic freeness has been proved only for Wigner matrices and for 
certain highly symmetric ensembles; for a detailed overview of prior 
results, see \cite{CGP19,BC20} and the references cited therein. In 
contrast, neither symmetry nor independent entries plays any role in 
Theorem \ref{thm:introfree}, which enables us to establish strong 
asymptotic freeness in models that appear to lie far outside the reach of 
previous methods (for example, for sparse Wigner matrices of dimension $d$ 
with only $O(d\log^4d)$ nonzero entries, see Example \ref{ex:sparse}).
For many such models, even weak asymptotic freeness \eqref{eq:weakfree} 
was not previously known.

\subsubsection{Sums of independent random matrices}
\label{sec:introsums}

When viewed as matrix concentration inequalities, bounds such as 
\eqref{eq:introgood} are easily applicable in concrete situations and 
yield results of optimal order in many examples where classical matrix 
concentration inequalities are suboptimal. To illustrate this, we will 
discuss in Section \ref{sec:ex} a variety of explicit examples that 
appear, at this level of generality, to be outside the reach of 
classical methods of random matrix theory.

Nonetheless, the main results of this paper are obtained for Gaussian 
random matrices, which may be restrictive in applications. One important 
reason for the broad utility of classical matrix concentration 
inequalities \cite{Tro15} is that they extend to arbitrary sums of 
independent random matrices, a setting that captures many non-Gaussian 
models that arise in practice. It turns out, however, that non-Gaussian 
versions of our results already follow as a consequence of the Gaussian 
inequalities, so that the focus of this paper on Gaussian inequalities is 
not a significant restriction. Indeed, in the follow-up work \cite{BV22}, 
it is shown that the spectrum of any sum of independent random matrices 
behaves, under mild conditions, like that of the Gaussian random matrix 
whose entries have the same mean and covariance. When the results of the 
present paper are applied to the resulting Gaussian matrices, one 
immediately obtains non-Gaussian extensions of our main results. For sake 
of illustration we state a non-Gaussian analogue of Theorem 
\ref{thm:intro} here, as well as a tail bound that may be 
compared with the matrix Bernstein inequality \cite{Tro15}.

\begin{thm}
\label{thm:mconcbvh}
Let $Z_1,\ldots,Z_n$ be arbitrary independent $d\times d$ self-adjoint
centered random matrices, and let $X=\sum_{i=1}^n Z_i$. Then
$$
	\EE\|X\| \le
	\|X_{\rm free}\| +
	C\{
	v^{\frac{1}{2}}\sigma^{\frac{1}{2}}(\log d)^{\frac{3}{4}} +
	R^{\frac{1}{3}}\sigma^{\frac{2}{3}}(\log d)^{\frac{2}{3}} +
	R\log d\}
$$
and
$$
	\mathbf{P}\big[|X\| \ge \|X_{\rm free}\| 
	+ C\{
	v^{\frac{1}{2}}\sigma^{\frac{1}{2}}(\log d)^{\frac{3}{4}}
	+ 
	\sigma_*t^{\frac{1}{2}} + 
	R^{\frac{1}{3}}\sigma^{\frac{2}{3}}t^{\frac{2}{3}}
	+ Rt
	\}
	\big] \le de^{-t}
$$
for all $t\ge 0$, where $C$ is a universal constant,
$\sigma=\|\EE X^2\|^{\frac{1}{2}}$, 
$v=\|\mathrm{Cov}(X)\|^{\frac{1}{2}}$,
$\sigma_*=\sup_{\|v\|=\|w\|=1}\EE[|\langle v,Xw\rangle|^2]^{\frac{1}{2}}
\le v$, $R = 
\|\max_i\|Z_i\|\|_\infty$, and $X_{\rm free}$ is the free 
model associated to the centered Gaussian random matrix whose entries have 
the same covariance as those of $X$ \emph{(}in particular, $\|X_{\rm 
free}\|\le 2\sigma$\emph{)}.
\end{thm}

We refer to \cite{BV22} for analogous extensions of all the main results 
of this paper.
(Further discussion of non-Gaussian extensions may be found in 
Section \ref{sec:universality}.)

\subsection{Overview of the proofs}
\label{sec:overviewpf}

\subsubsection{Crossings}

Before we describe the main technique used in our proofs, let us briefly 
outline the origin of the key parameter $v(X)$ that quantifies 
noncommutativity in our results, and its relation to free probability.

The simplest way to understand the difference between the random matrix
$X$ and its free counterpart $X_{\rm free}$ is in terms of their moments. 
Let us recall that these moments may be expressed combinatorially as
\cite[pp.\ 128--129]{NS06}
$$
	\EE[\ntr X^{2p}] =
	\sum_{\pi\in \mathrm{P}_2([2p])}
	\sum_{(i_1,\ldots,i_{2p})\sim\pi}
	\ntr[A_{i_1}\ldots A_{i_{2p}}]
$$
and
$$
	({\ntr}\otimes\tau)(X_{\rm free}^{2p}) =
	\sum_{\pi\in \mathrm{NC}_2([2p])}
	\sum_{(i_1,\ldots,i_{2p})\sim\pi}
	\ntr[A_{i_1}\ldots A_{i_{2p}}],
$$
where $\mathrm{P}_2([2p])$ and $\mathrm{NC}_2([2p])$ denote the families 
of all pair partitions and noncrossing pair partitions of $[2p]$, 
respectively, and $(i_1,\ldots,i_{2p})\sim\pi$ signifies that $i_k=i_l$ 
whenever $\{k,l\}\in\pi$. In other words, what distinguishes free 
probability from classical probability is the absence of crossings, that 
is, of terms of the form $\sum_{ij} \cdots A_i \cdots A_j\cdots A_i\cdots 
A_j\cdots$ in the moment formulae.

In free probability, the vanishing of crossings arises from the 
noncommutativity of the semicircular family $s_i$. Even in 
\eqref{eq:simplest}, however, crossings may be intrinsically suppressed 
due to the noncommutativity of the coefficients $A_i$. It is a beautiful 
idea of Tropp \cite{Tro18} to quantify the latter effect by the 
parameter
%\begin{equation}
%\label{eq:wx}
$$
	w(X) = 
	\sup_{U,V,W} \|\mathbf{E}[X_1 U X_2 V X_1 W X_2]\|^{\frac{1}{4}}
	=
	\sup_{U,V,W}\Bigg\|
	\sum_{i,j=1}^n A_i U A_j V A_i W A_j\Bigg\|^{\frac{1}{4}},
$$
%\end{equation}
where $X_1,X_2$ are i.i.d.\ copies of $X$ and the supremum is taken over 
all (nonrandom) unitary matrices $U,V,W$ of the same dimension as 
$X$. Note that when all $A_i$ commute, $w(X)\ge 
\|\sum_{ij}A_iA_jA_iA_j\|^{\frac{1}{4}}= 
\|(\sum_{i}A_i^2)^2\|^{\frac{1}{4}}=\sigma(X)$; but if
$w(X)\ll\sigma(X)$, the contribution of crossings will be suppressed.

Unfortunately, the quantity $w(X)$ is very unwieldy and is difficult to 
compute in practice. Moreover, as will be explained below, the quantity 
that will arise in our proofs is not $w(X)$, but rather $w(\tilde X)$ for 
an auxiliary matrix $\tilde X$ of much higher dimension. To control this 
parameter, we will show in Section \ref{sec:wxvx} that 
\begin{equation}
\label{eq:wxvx}
	w(X) \le v(X)^{\frac{1}{2}} \sigma(X)^{\frac{1}{2}},
\end{equation}
which enables us to formulate our results in terms of the much simpler 
quantity $v(X)$ that is readily computable in concrete situations. In 
particular, it follows that $v(X)$ does indeed capture noncommutativity, 
as it controls $w(X)$.

The notion that smallness of $w(X)$ should lead to free behavior is 
implicit in the work of Tropp \cite{Tro18}. However, the attempt in 
\cite{Tro18} to exploit this idea by means of moment recursions appears to 
be insufficiently powerful to capture this phenomenon without imposing 
strong symmetry assumptions on the coefficients $A_i$. A key new idea of 
this paper enables us to capture this phenomenon in its full strength.

\subsubsection{Interpolation}

The central idea behind our proofs is the following 
construction. Let $G_1^N,\ldots,G_n^N$ be independent standard Wigner 
matrices as in Section \ref{sec:introfree}, and let $D_1^N,\ldots,D_n^N$ 
be independent $N\times N$ diagonal matrices with i.i.d.\ standard 
Gaussians on the diagonal. Define for $q\in[0,1]$ the random matrix
$$
	X^N_q = \sum_{i=1}^n A_i\otimes (\sqrt{q}\,D_i^N +
	\sqrt{1-q}\,G_i^N).
$$
The point of this construction is that the family 
$(X^N_q)_{q\in[0,1],N\in\mathbb{N}}$ enables us to interpolate between $X$ 
and $X_{\rm free}$. Indeed, $X^N_0=X^N$ is the model \eqref{eq:xtens}, 
whose moments converge as $N\to\infty$ to those of $X_{\rm free}$ by 
\eqref{eq:weakfree} (this is the only property that will be used in 
our proofs; strong asymptotic freeness will not be assumed). On the other 
hand, it is readily verified that $X^N_1$ has the same moments as $X$ in 
the sense $\mathbf{E}[\ntr X^p]=\mathbf{E}\ntr[(X^N_1)^p]$ for every 
$p,N\in\mathbb{N}$.

In order to bound the moments of $X$ by those of $X_{\rm free}$, it 
suffices to bound the rate at which the moments change along the above 
interpolation. Given that the moments of $X$ and $X_{\rm free}$ differ 
only by terms involving crossings, it is natural to expect that the rate 
of change along the interpolation will be controlled by the contributions 
of the crossings. It will turn out that the construction of the matrices 
$X_q^N$ has precisely the right form in order to capture this phenomenon 
in terms of the parameters described in the previous section. More 
precisely, the explicit expression for the derivative 
$\frac{d}{dq}\mathbf{E}\ntr[(X^N_q)^p]$, which can be computed using a 
standard Gaussian interpolation lemma \cite[\S 1.3]{Tal11}, can be 
controlled in terms of the quantity
$$
        \tilde w(X) = \sup_{N} w(X^N_1).
$$
The resulting differential inequality may be integrated to 
bound the moments of $X$ in terms of the moments of $X_{\rm free}$ and the 
parameter $\tilde w(X)$. As the latter is nearly impossible to compute, 
we finally obtain a practical bound $\tilde w(X)\le 
v(X)^{\frac{1}{2}}\sigma(X)^{\frac{1}{2}}$ using \eqref{eq:wxvx} and 
$v(X^N_1)=v(X)$, $\sigma(X^N_1)=\sigma(X)$.

The above interpolation method proves to be a powerful tool for capturing 
``intrinsic freeness''. The same method can be used to control not just 
the moments, but also various other spectral statistics. In particular, we 
will control the full spectrum of $X$ by that of $X_{\rm free}$ by 
applying the interpolation method to large moments of the resolvent 
$\mathbf{E}[\ntr |z\id-X|^{-2p}]$. Such control of the full spectrum is 
crucial for the applications of our results to free probability described 
in Section \ref{sec:introasympfree}.

\begin{rem}
\label{rem:collins}
After the results of this paper were completed, we learned that a 
different interpolation method was recently used by Collins, Guionnet, and 
Parraud \cite{CGP19} to obtain a quantitative form of the strong asymptotic 
freeness of Wigner matrices due to Haagerup-Thorbj{\o}rnsen. Rather than 
interpolating between scalar Gaussians and Wigner matrices, \cite{CGP19} 
interpolate in the opposite direction, between Wigner matrices and a 
semicircular family, using the free Ornstein-Uhlenbeck semigroup.
The latter approach can only capture the noncommutativity of the 
Wigner matrices themselves, in contrast to the results of this paper that 
capture the noncommutativity of the coefficient matrices $A_i$ 
(``intrinsic freeness'') and therefore open the door to the study of 
general Gaussian random matrices. On the other hand, by exploiting the 
special structure of Wigner matrices, the methods of 
\cite{CGP19} can be adapted to obtain higher order expansions \cite{Par22} 
which play a key role in the recent work on the Peterson-Thom
conjecture \cite{BC22}. Taken together, all these results illustrate 
the power of interpolation methods for the study of quantitative phenomena
in free probability theory and random matrix theory.
\end{rem}

\subsection{Organization of this paper}

The rest of this paper is organized as follows. In Section \ref{sec:main}, 
the main results of this paper will be presented in full detail. The 
utility of our main results will then be illustrated in a number of 
concrete examples in Section \ref{sec:ex}. Section \ref{sec:prelim} 
briefly reviews some basic notions of free probability, and introduces 
various tools that are used throughout the rest of the paper. The proofs 
of our main results are given in Sections 
\ref{sec:spec}--\ref{sec:strongfree}.

The final Section \ref{sec:disc} is devoted to a discussion of various 
broader questions arising from our main results. In particular, we will 
show that there cannot exist a canonical choice of the parameter 
$\sigma_{**}(X)$ in the inequality \eqref{eq:secondorder}, as any such 
parameter must violate some natural property of the spectral norm. This 
disproves a conjecture, formulated in \cite{Tro18,vH17,Ban15}, which 
suggests that the parameter $v(X)$ in our main results can be replaced by 
a certain smaller parameter $\sigma_*(X)$ that will be defined below. 
We conclude by discussing a number of further questions.

\subsection{Notation}

The following notations will be frequently used throughout this paper. We 
write $[n]:=\{1,\ldots,n\}$ for $n\in\mathbb{N}$. For a bounded operator 
$X$ on a Hilbert space, we denote by $\|X\|$ its operator (i.e., spectral) 
norm and by $|X|:=(X^*X)^{\frac{1}{2}}$. The spectrum of $X$ is denoted as 
$\spc(X)$. If $X$ is self-adjoint and $h:\mathbb{R}\to\mathbb{C}$ is 
measurable, then the operator $h(X)$ is defined by the usual functional 
calculus (in particular, if $X$ is a self-adjoint matrix, $h$ is applied 
to the eigenvalues while keeping the eigenvectors fixed). The algebra of 
$d\times d$ matrices with values in a *-algebra $\mathcal{A}$ is denoted 
as $\M_d(\mathcal{A})$, and its subspace of self-adjoint matrices is 
denoted as $\M_d(\mathcal{A})_{\rm sa}$. For complex matrices 
$M\in\M_d(\mathbb{C})$, we always denote by $\tr M:=\sum_{i=1}^d M_{ii}$ 
the unnormalized trace and by $\ntr M := \frac{1}{d}\tr M$ the normalized 
trace. We use the convention that when an expectation is followed by
square brackets, the expectation is applied before any external 
operations (in particular, $\mathbf{E}[X]^\alpha := 
(\mathbf{E} X)^\alpha$).

\section{Main results}
\label{sec:main}

\subsection{Concentration of the spectrum}
\label{sec:mainspec}

The strongest results of this paper apply to arbitrary random matrices
with jointly Gaussian entries (this model is more general than the one
that was assumed for sake of illustration in the introduction). 
To define this model, fix $d\ge 2$ and $n\in\mathbb{N}$, let
$A_0,\ldots,A_n\in\M_d(\mathbb{C})$, let $g_1,\ldots,g_n$ be i.i.d.\ real 
Gaussian variables with zero mean and unit variance, and let 
$s_1,\ldots,s_n$ be a free semicircular family 
(cf.\ Section \ref{sec:free}). We now define
\begin{equation}
\label{eq:model}
	X := A_0 + \sum_{i=1}^n g_iA_i,\quad\qquad
	X_{\rm free} := A_0\otimes \id + \sum_{i=1}^n A_i\otimes s_i.
\end{equation}
In formulating our results, it will sometimes be convenient to assume in 
addition that the model is self-adjoint, that is, that $A_0,\ldots,A_n\in 
\M_d(\mathbb{C})_{\rm sa}$. In such cases this assumption will be made 
merely for notational convenience and is not a restriction, as will be 
explained in Remark \ref{rem:sa} below.

The following parameters will play a fundamental role in the sequel:
\begin{alignat*}{3}
	&\sigma(X)^2 &{}:={}& 
	\centermathcell{
		\Bigg\|\sum_{i=1}^n A_i^*A_i\Bigg\|\vee
		\Bigg\|\sum_{i=1}^n A_iA_i^*\Bigg\|
	}
	&&=
	\|\EE \hat X^*\hat X\|\vee \|\EE \hat X\hat X^*\|,
	\\
	&\sigma_*(X)^2 &{}:={}&
	\centermathcell{
		\sup_{\|v\|=\|w\|=1} \sum_{i=1}^n |\langle v,A_iw\rangle|^2
	}
	&&=
	\sup_{\|v\|=\|w\|=1}
	\EE[|\langle v,\hat Xw\rangle|^2],\\
	&v(X)^2 &{}:={}&
	\centermathcell{ 
		\sup_{{\tr}|M|^2\le 1} \sum_{i=1}^n |{\tr}[A_i M]|^2
	}
	&&= \|\mathrm{Cov}(X)\|,	
\end{alignat*}
where $\hat X:=X-\EE X$.
It follows readily from the definitions that $\sigma_*(X)\le v(X)$ and
$\sigma_*(X)\le\sigma(X)$. 
As the following combination will appear frequently, we let
\begin{equation*}
	\tilde v(X)^2 := v(X)\sigma(X).
\end{equation*}
Note that the definitions of these parameters 
do not involve $A_0$.

We can now formulate our main result on concentration of the spectrum of 
$X$. Here $\spc(M)$ denotes the spectrum of a self-adjoint operator $M$.

\begin{thm}
\label{thm:mainsp}
For the model \eqref{eq:model} with 
$A_0,\ldots,A_n\in\M_d(\mathbb{C})_{\rm sa}$, we have 
$$
	\PP\big[\spc(X) \subseteq \spc(X_{\rm free}) +
	C\{\tilde v(X)(\log d)^{\frac{3}{4}}+\sigma_*(X)t\}[-1,1]
	\big]
	\ge 1-e^{-t^2}
$$
for all $t\ge 0$, where $C$ is a universal constant.
\end{thm}

The spectrum of $X_{\rm free}$ always consists of a finite union of 
bounded intervals \cite{AEK20}. Theorem~\ref{thm:mainsp} implies that when 
$v(X)\ll (\log d)^{-\frac{3}{2}}\sigma(X)$, \emph{all} eigenvalues of 
$X$ are close to the spectrum of $X_{\rm free}$. In particular, not only 
must the extreme eigenvalues of $X$ lie close to the edge of the spectrum 
of $X_{\rm free}$, but also the interior eigenvalues cannot lie far inside 
the gaps in the spectrum of $X_{\rm free}$.

When specialized to the extreme eigenvalues, Theorem \ref{thm:mainsp} 
yields a bound on the spectral norm of $X$. We formulate it here directly 
for non-self-adjoint matrices.

\begin{cor}
\label{cor:norm}
For the model \eqref{eq:model} with
$A_0,\ldots,A_n\in\M_d(\mathbb{C})$, we have
$$
	\PP\big[\|X\| > \|X_{\rm free}\| +
	C\tilde v(X)(\log d)^{\frac{3}{4}}+C\sigma_*(X)t
	\big]
	\le e^{-t^2}
$$
for all $t\ge 0$, where $C$ is a universal constant. Moreover,
$$
	\EE\|X\| \le \|X_{\rm free}\| +
	C\tilde v(X)(\log d)^{\frac{3}{4}}.
$$
\end{cor}

Theorem \ref{thm:mainsp} and Corollary \ref{cor:norm} will be proved
in Section \ref{sec:norm}.

\begin{rem}
In order to apply Corollary \ref{cor:norm} in concrete situations, we must 
be able to compute or estimate $\|X_{\rm free}\|$. For ease of reference, 
we presently recall two useful facts; further discussion and references 
may be found in Section \ref{sec:free}. In the following, $\lambda_{\rm 
max}(M)$ denotes the maximal eigenvalue of a self-adjoint matrix $M$.
%%%%%%%%%%%%%%%%%%%%%%%%%%%%%%%%%%
\end{rem}

\begin{lem}[Lehner]
\label{lem:lehner}
For the model \eqref{eq:model} with
$A_0,\ldots,A_n\in\M_d(\mathbb{C})_{\rm sa}$, we have
$$
	\|X_{\rm free}\| = \max_{\varepsilon=\pm 1}
	\inf_{Z>0}\lambda_{\rm max}
	\Bigg(Z^{-1}+\varepsilon A_0 + \sum_{i=1}^n A_i Z A_i
	\Bigg),
$$
where the infimum is over positive definite $Z\in\M_d(\mathbb{C})_{\rm sa}$.
The infimum may be further restricted to $Z$ for which the matrix in 
$\lambda_{\rm max}(\cdots)$ is a multiple of the identity.
\end{lem}

\begin{lem}[Pisier]
\label{lem:freenck}
For the model \eqref{eq:model} with 
$A_0,\ldots,A_n\in\M_d(\mathbb{C})$, we have
$$
	\|A_0\|\vee\sigma(X) \le \|X_{\rm free}\| \le
	\|A_0\|+
	\Bigg\|\sum_{i=1}^n A_i^*A_i\Bigg\|^{\frac{1}{2}} +
	\Bigg\|\sum_{i=1}^n A_iA_i^*\Bigg\|^{\frac{1}{2}}.
$$
\end{lem}

Note that the combination of Corollary \ref{cor:norm} and Lemma 
\ref{lem:freenck} immediately yields a Gaussian matrix concentration 
inequality of the form \eqref{eq:secondorder}.
%%%%%%%%%%%%%%%%%%%%%%%%%%%%%%%%%%\end{rem}

\begin{rem}
\label{rem:sa}
For simplicity, we formulated results such as Theorem 
\ref{thm:mainsp} and Lemma \ref{lem:lehner} for self-adjoint matrices.
The following standard device makes it possible to reduce the 
general case to the self-adjoint case. Given
$A_0,\ldots,A_n\in\M_d(\mathbb{C})$, define the matrices
$\breve A_0,\ldots,\breve A_n\in M_{2d}(\mathbb{C})_{\rm sa}$, $\breve X$, 
and $\breve X_{\rm free}$ as
$$
	\breve A_i = \begin{bmatrix}
	0 & A_i \\
	A_i^* & 0
	\end{bmatrix},\qquad
	\breve X = \begin{bmatrix}
	0 & X \\
	X^* & 0
	\end{bmatrix},\qquad
	\breve X_{\rm free} = \begin{bmatrix}
	0 & X_{\rm free} \\
	X_{\rm free}^* & 0
	\end{bmatrix}.
$$
Then it is not difficult to show (see Section \ref{sec:sa}) that
$$
	\spc(\breve X)\cup\{0\} = \spc(|X|) \cup {-\spc(|X|)}\cup\{0\},
$$
and analogously for $X_{\rm free}$; moreover, we have
$$ 
	\sigma(\breve X)=\sigma(X),\qquad
	\sigma_*(\breve X)=\sigma_*(X),\qquad
	v(\breve X) \le \sqrt{2}\,v(X).
$$
Applying Theorem \ref{thm:mainsp} to $\breve X$ therefore shows that 
in the non-self-adjoint case, the \emph{singular values} of $X$ 
concentrate around those of  $X_{\rm free}$. Similarly, we can apply
Lemma \ref{lem:lehner} to $\breve X_{\rm free}$ to obtain an explicit
formula for $\|X_{\rm free}\|$.

The above construction does not require the matrices $A_i$ to be square. 
However, if $A_i$ are $d_1\times d_2$ matrices with $d_1<d_2$, 
the singular values of $X$ are unchanged if we add $d_2-d_1$ zero rows to 
the matrix. Thus there is no loss of generality in restricting attention to 
square matrices, as we do for simplicity throughout this paper. 
\end{rem}

\subsection{Spectral statistics}
\label{sec:mainstat}

The results of the previous section quantify concentration of the 
eigenvalues of $X$ near the spectrum of $X_{\rm free}$. We now formulate 
several complementary results that quantify the closeness of the spectral 
distributions of $X$ and $X_{\rm free}$. We begin by stating a bound on 
the moments.

\begin{thm}
\label{thm:moments}
For the model \eqref{eq:model} with
$A_0,\ldots,A_n\in\M_d(\mathbb{C})$, we have
$$
	|\EE[\ntr |X|^{2p}]^{\frac{1}{2p}}-
	({\ntr}\otimes\tau)(|X_{\rm free}|^{2p})^{\frac{1}{2p}}|
	\le 2p^{\frac{3}{4}}\tilde v(X)
$$
for all $p\in\mathbb{N}$.
\end{thm}

Let us emphasize that unlike the results of Section \ref{sec:mainspec}, 
Theorem \ref{thm:moments} gives a two-sided bound on $X$ in terms of 
$X_{\rm free}$. This opens the door to obtaining sharp asymptotics from 
our nonasymptotic bounds.

The same method of proof is readily applied to other spectral statistics. 
To illustrate this, we will bound the matrix-valued Stieltjes transform, 
which plays an important role in operator-valued free probability 
\cite[Chapters 9--10]{MS17}. A bound of this kind is most naturally 
formulated for self-adjoint matrices.

\begin{thm}
\label{thm:stieltjes}
For the model \eqref{eq:model} with 
$A_0,\ldots,A_n\in\M_d(\mathbb{C})_{\rm sa}$, define the matrix-valued 
Stieltjes transforms $G(Z),G_{\rm free}(Z)\in \M_d(\mathbb{C})$ as
$$
	G(Z) := \EE[(Z - X)^{-1}],
	\qquad\quad
	G_{\rm free}(Z) := (\mathrm{id}\otimes\tau)[(
	Z\otimes\id - X_{\rm free})^{-1}].
$$
Then we have
$$
	\|G(Z)-G_{\rm free}(Z)\| \le \tilde v(X)^4
	\|(\mathrm{Im}\,Z)^{-5}\|
$$
for all $Z\in\M_d(\mathbb{C})$ with
$\mathrm{Im}\,Z:=\frac{1}{2i}(Z-Z^*)>0$.
\end{thm}

Following \cite[\S 6]{HT05}, Theorem \ref{thm:stieltjes} implies a 
bound on smooth spectral statistics. 

\begin{cor}
\label{cor:c6}
For the model \eqref{eq:model} with 
$A_0,\ldots,A_n\in\M_d(\mathbb{C})_{\rm sa}$, we have
$$
	|\EE[\ntr f(X)] - ({\ntr}\otimes\tau)[f(X_{\rm free})]| \lesssim
	\tilde v(X)^4\|f\|_{W^{6,1}(\mathbb{R})}
$$
for every $f\in W^{6,1}(\mathbb{R})$.
\end{cor}

Theorems \ref{thm:moments}--\ref{thm:stieltjes} and
Corollary \ref{cor:c6} will be proved in Section \ref{sec:spec}.

\subsection{Strong asymptotic freeness}
\label{sec:mainsaf}

By combining the bounds of Sections \ref{sec:mainspec}--\ref{sec:mainstat} 
with the linearization trick of \cite{HT05}, we will be able to 
establish strong asymptotic freeness for a remarkably general class of 
random matrices. We presently give a complete formulation of our main 
result in this direction.

\begin{thm}
\label{thm:free}
Let $s_1,\ldots,s_m$ be a free semicircular family.
For each $N\ge 1$, let $H_1^N,\ldots,H_m^N$ be independent self-adjoint 
random matrices of dimension $d=d(N)\ge N$ such that each $H_k^N$ has 
jointly Gaussian entries,
$$
	\lim_{N\to\infty}\|\mathbf{E}[H_k^N]\|=0,\qquad\quad
	\lim_{N\to\infty}\|\mathbf{E}[(H_k^N)^2]-\id\| = 0
$$
for all $k$. Then the following hold.
\begin{enumerate}[a.]
\itemsep\abovedisplayskip
\item
If $v(H_k^N)=o(1)$ as $N\to\infty$ for all $k$, then 
$$
	\lim_{N\to\infty} \mathbf{E}[\ntr p(H_1^N,\ldots,H_m^N)] =
	\tau(p(s_1,\ldots,s_m))
$$
for every noncommutative polynomial $p$.
\item If $v(H_k^N)=o((\log d)^{-\frac{3}{2}})$ as $N\to\infty$ for all $k$, 
then
\begin{align*}
	&\lim_{N\to\infty} \mathbf{E}[\|p(H_1^N,\ldots,H_m^N)\|] =
	\|p(s_1,\ldots,s_m)\|,
	\\
	&\lim_{N\to\infty} \|p(H_1^N,\ldots,H_m^N)\| =
	\|p(s_1,\ldots,s_m)\|\quad\mbox{a.s.}, \\
	&\lim_{N\to\infty} \ntr p(H_1^N,\ldots,H_m^N) =
	\tau(p(s_1,\ldots,s_m))\quad\mbox{a.s.}
\end{align*}
for every noncommutative polynomial $p$.
\end{enumerate}
\end{thm}

Let us recall that the type of convergence in part \emph{b} of Theorem 
\ref{thm:free}, called \emph{strong convergence in distribution}, has even 
stronger implications: it implies that both the spectral distribution and 
support of the spectrum of any polynomial $p(H_1^N,\ldots,H_m^N)$ 
converges to that of $p(s_1,\ldots,s_m)$ as $N\to\infty$ in the sense of 
weak convergence and Hausdorff convergence, respectively; see 
\cite[Proposition 2.1]{CM14}.

Surprisingly, the conclusion of Theorem \ref{thm:free} appears to be new 
at this level of generality already for a single random matrix $m=1$. In
this case, we obtain the following result in the spirit of classical 
random matrix theory.

\begin{cor}
\label{cor:oldie}
Let $H^N$ be a self-adjoint random matrix of dimension $d=d(N)$ with
jointly Gaussian entries, and assume that
$$
	\|\mathbf{E}[H^N]\|=o(1),\qquad
	\|\mathbf{E}[(H^N)^2]-\id\|=o(1),\qquad
	v(H^N)=o((\log d)^{-\frac{3}{2}})
$$
as $N\to\infty$.
Then the empirical distribution
$$
	\mu_{H^N} := \frac{1}{d}\sum_{i=1}^d \delta_{\lambda_i(H^N)}
$$
of the eigenvalues $\lambda_i(H^N)$ of $H^N$ converges weakly a.s.\ to the
semicircle law
$$
	\mu_{H^N} \stackrel{\mathrm{w}}{\longrightarrow} \mu_{\rm sc}
	\quad\mbox{a.s.},\quad\qquad
	\mu_{\rm sc}(dx) = \frac{1}{2\pi}\sqrt{4-x^2}\,1_{|x|\le 2}\,dx,
$$
and we have convergence of the norm $\|H^N\|\to 2$ a.s.\ as $N\to\infty$.
\end{cor}

Let us emphasize that Corollary \ref{cor:oldie} (and Theorem 
\ref{thm:free}) makes no structural assumptions on the variance or 
dependence pattern of $H^N$ beyond the minimal isotropy conditions 
$\EE[H^N]\approx 0$ and $\EE[(H^N)^2]\approx\id$. Previous results on 
Gaussian random matrices with dependent entries require restrictive 
structural assumptions to obtain even the semicircle law, cf.\ 
\cite{FKK21} and the references therein.

Theorem \ref{thm:free} and
Corollary \ref{cor:oldie} will be proved in Section \ref{sec:strongfree}.

\section{Examples}
\label{sec:ex}

The aim of this section is to illustrate our main results in 
concrete examples. In Section \ref{sec:exind} we consider Gaussian random 
matrices with independent entries, while Section \ref{sec:exdep} discusses 
some simple examples of random matrix models with dependent entries. 
Section \ref{sec:samplecov} is concerned with Gaussian sample covariance 
matrices, whose samples may be neither independent nor identically 
distributed. Section \ref{sec:sigmamin} is concerned with bounds on the 
smallest singular value of random matrices.

\subsection{Independent entries}
\label{sec:exind}

In this section, we consider the case of real symmetric Gaussian random 
matrices with independent entries (nonsymmetric or complex matrices may be 
considered analogously, but we restrict attention to the real symmetric 
case for simplicity). More precisely, let $X$ be the $d\times d$ symmetric 
random matrix with entries $X_{ij}=b_{ij}g_{ij}$, where $\{g_{ij}:i\ge 
j\}$ are i.i.d.\ standard real Gaussian random variables and 
$\{b_{ij}:i\ge j\}$ are given nonnegative scalars. We let $b_{ji}:=b_{ij}$ 
and $g_{ji}:=g_{ij}$. This model may be expressed in the form 
\eqref{eq:model} as 
\begin{equation}
\label{eq:indep}
	X = \sum_{i\ge j} g_{ij} b_{ij} E_{ij},
\end{equation}
where $E_{ii}:=e_ie_i^*$ and $E_{ij}:=(e_ie_j^*+e_je_i^*)$ for $i>j$. Here 
and in the sequel, $e_1,\ldots,e_d$ denotes the coordinate basis of
$\mathbb{R}^d$.

The independent entry setting is the only general model of nonhomogeneous 
random matrices for which satisfactory norm bounds were obtained prior to 
this work \cite{BvH16,vH17b,LvHY18}. In particular, it was proved in 
\cite[Theorem 1.1]{BvH16} that
\begin{equation}
\label{eq:bvhsharp}
	\EE\|X\| \le (2+\varepsilon)\max_{i}\sqrt{\sum_j b_{ij}^2}
	+ \frac{C}{\sqrt{\varepsilon}}\max_{ij}b_{ij} \sqrt{\log d}
\end{equation}
for any $0<\varepsilon<1$, where $C$ is a universal constant. The constant 
$2$ in the leading term is optimal, as 
$\EE\|X\|=2+o(1)$ as $d\to\infty$ when $X$ is a standard 
Wigner matrix, that is, when $b_{ij}=\frac{1}{\sqrt{d}}$ for all $i,j$.
Moreover, \eqref{eq:bvhsharp} is nearly sharp in the 
sense that the inequality can be reversed up to a universal constant under 
mild assumptions \cite[\S 3.5]{BvH16} (a completely sharp 
dimension-free bound, but without the optimal constant in the leading 
term, was proved in \cite{LvHY18}).

Nonetheless, even in the special case of independent entries, the general 
results of this paper can yield a significant improvement over 
\eqref{eq:bvhsharp}. 

\begin{lem}
For the model \eqref{eq:indep}, we have
$$
	\sigma(X) = \max_{i}\sqrt{\sum_j b_{ij}^2},
	\qquad\quad
	\max_{ij}b_{ij} \le
	\sigma_*(X) \le
	v(X) \le \sqrt{2}\max_{ij}b_{ij}.
$$
In particular,
\begin{equation}
\label{eq:normindep}
	\EE\|X\| \le (1+\varepsilon)\|X_{\rm free}\| + 
	\frac{C}{\varepsilon}\max_{ij}b_{ij}\,(\log d)^{\frac{3}{2}}
\end{equation}
for any $\varepsilon>0$, where $C$ is a universal constant.
\end{lem}

\begin{proof}
The expression for $\sigma(X)^2=\|\EE X^2\|$ follows readily as
\begin{equation}
\label{eq:indepsig}
	\EE X^2 = \sum_i e_ie_i^* \sum_j b_{ij}^2
\end{equation}
is a diagonal matrix. Moreover, that $v(X)^2\ge \sigma_*(X)^2\ge 
\max_{ij}\EE[|X_{ij}|^2]=\max_{ij}b_{ij}^2$ follows immediately from the 
definitions in Section \ref{sec:mainspec}.

On the other hand, as the pairs of entries $(X_{ij},X_{ji})$ are 
independent for distinct indices $i\ge j$, we have 
$\mathrm{Cov}(X)=\bigoplus_{i\ge j}C_{ij}$ where $C_{ij}$ is the 
covariance matrix of $(X_{ij},X_{ji})$. Thus $v(X)^2=\|\mathrm{Cov}(X)\|= 
\max_{i\ge j} \|C_{ij}\| \le 2\max_{ij}b_{ij}^2$.

To conclude, it remains to invoke Corollary \ref{cor:norm} and to note
that $c\tilde v(X)(\log d)^{\frac{3}{4}} \le
\varepsilon \|X_{\rm free}\| + \frac{c^2}{4\varepsilon}v(X)(\log 
d)^{\frac{3}{2}}$
for any $c,\varepsilon>0$ by Young's inequality and Lemma 
\ref{lem:freenck}.
\end{proof}

While the second term of \eqref{eq:normindep} has a slightly suboptimal
power on the logarithm as compared to \eqref{eq:bvhsharp}, this 
term is already negligible when
\begin{equation}
\label{eq:indepf}
	\max_{ij}b_{ij}^2 = o\bigg((\log d)^{-3}\max_i\sum_j b_{ij}^2
	\bigg).
\end{equation}
As soon as this is the case, the bound \eqref{eq:normindep} improves on
\eqref{eq:bvhsharp} in that the leading term 
$2\sigma(X)$ is replaced by the sharp free probability quantity $\|X_{\rm 
free}\|$. We always have $\|X_{\rm free}\|\le 2\sigma(X)$ by Lemma 
\ref{lem:freenck}, but this inequality often turns out to be strict in 
nonhomogeneous situations. To understand this phenomenon better, it is 
instructive to compute $\|X_{\rm free}\|$ in the present setting.

\begin{lem}
\label{lem:indepxfree}
For the model \eqref{eq:indep}, we have
$$
	\|X_{\rm free}\| = 
	\inf_{x\in\mathbb{R}^d_{++}} \max_i
	\bigg\{\frac{1}{x_i}+\sum_j b_{ij}^2x_j\bigg\}
	=
	2\sup_{w\in\Delta^{d-1}}\sum_i \sqrt{w_i\sum_j b_{ij}^2w_j},
$$
where we denote by $\mathbb{R}^d_{++}:=\{x\in\mathbb{R}^d:x>0\}$ the 
positive orthant and by $\Delta^{d-1}:=
\{x\in\mathbb{R}^d:x\ge 0,~\sum_i x_i=1\}$ 
is the standard simplex in $\mathbb{R}^d$. We always have $\|X_{\rm 
free}\|\le 2\sigma(X)$.
If $B=(b_{ij}^2)$ is an irreducible nonnegative matrix, then equality
$\|X_{\rm free}\|=2\sigma(X)$ holds if and only if
$\max_i\sum_j b_{ij}^2 = \min_i\sum_j b_{ij}^2$.
\end{lem}

\begin{rem}
The irreducibility assumption entails no loss of generality. In the 
general case, we may write $B=\bigoplus_i B_i$ in terms of its irreducible 
components $B_i$, and $X_{\rm free}=\bigoplus_i X_{{\rm free},i}$ 
decomposes accordingly.
As $\|X_{\rm free}\|=\max_i \|X_{{\rm free},i}\|$, the 
characterization of when $\|X_{\rm free}\|=2\sigma(X)$ reduces 
to the irreducible case.
\end{rem}

\begin{proof}[Proof of Lemma \ref{lem:indepxfree}]
Define 
$$
	f(Z) := Z^{-1} + \sum_{i\ge j} b_{ij}^2 E_{ij}ZE_{ij}.
$$
Fix any $Z>0$ so that $f(Z)$ is a multiple of the 
identity. Then $f(Z)=\mathrm{diag}(f(Z))$, where 
$\mathrm{diag}(M)_{ij}:=M_{ii}\delta_{ij}$. Using that $(Z^{-1})_{ii}\ge
(Z_{ii})^{-1}$ (as $\|Z^{\frac{1}{2}}e_i\|\|Z^{-\frac{1}{2}}e_i\|\ge 1$),
it follows readily that $f(Z)\ge f(\mathrm{diag}(Z))$. Thus Lemma 
\ref{lem:lehner} 
implies 
$$
	\|X_{\rm free}\| = \inf_{Z>0}\lambda_{\rm max}(f(\mathrm{diag}(Z))) =
	\inf_{x\in\mathbb{R}^d_{++}}\max_i
	\bigg\{
	\frac{1}{x_i}+\sum_j b_{ij}^2x_j
	\bigg\}.
$$
We can further compute
$$
	\|X_{\rm free}\|
	=
	\inf_{x\in\mathbb{R}^d_{++}}\sup_{w\in\Delta^{d-1}}
	\sum_i w_i
	\bigg\{
	\frac{1}{x_i}+\sum_j b_{ij}^2x_j
	\bigg\} =
	2\sup_{w\in\Delta^{d-1}}\sum_i \sqrt{w_i\sum_j b_{ij}^2w_j},
$$
where we used the Sion minimax theorem to exchange the infimum and 
supremum.

If we apply Cauchy-Schwarz to the rightmost expression for $\|X_{\rm 
free}\|$, we obtain $\|X_{\rm free}\|\le 2\max_i[\sum_j 
b_{ij}^2]^{\frac{1}{2}}=2\sigma(X)$ directly. Therefore, when 
$\|X_{\rm free}\|=2\sigma(X)$, the maximizing vector $w\in\Delta^{d-1}$ 
must yield equality in Cauchy-Schwarz. The latter implies there exists 
$\rho\ge 0$ such that $Bw=\rho w$ and 
$\|X_{\rm free}\|=2\sqrt{\rho}$. In particular, if $B$ is irreducible,
then $\rho=\rho(B)$ is the 
largest eigenvalue of $B$ by the Perron-Frobenius theorem 
\cite[p.\ 53]{Gan59}. It remains to 
recall that the inequality 
$\rho(B)\le\max_i\sum_j b_{ij}^2$ is strict unless $\max_i\sum_j 
b_{ij}^2=\min_i\sum_j b_{ij}^2$, cf.\ \cite[p.\ 63]{Gan59}. 
\end{proof}

In other words, under the mild assumption \eqref{eq:indepf}, the constant 
$2$ in \eqref{eq:bvhsharp} is suboptimal and the results of the present 
paper yield strictly better bounds on $\mathbf{E}\|X\|$ as soon as $\sum_j 
b_{ij}^2 \ne \sum_j b_{kj}^2$ for some $i,k$ (and $X$ does not decompose 
as a block-diagonal matrix). In such cases, Lemma \ref{lem:indepxfree} can 
be used to explicitly compute or estimate $\|X_{\rm free}\|$. The latter 
quantity has also been studied by completely different methods in 
\cite{EM19}, to which we refer for complementary results.

Even when $\max_i\sum_j b_{ij}^2=\min_i\sum_j b_{ij}^2$, however, our 
main results yield far stronger conclusions than just a bound on the 
spectral norm. Indeed, by \eqref{eq:indepsig}, this corresponds
precisely to the case where $\mathbf{E}[X^2]=\sigma(X)^2\id$; thus 
any independent family of such matrices is strongly asymptotically free by 
Theorem \ref{thm:free}.

\begin{cor}
\label{cor:indepsfree}
Let $s_1,\ldots,s_m$ be a free semicircular family.
For each $N\ge 1$, let $H_1^N,\ldots,H_m^N$ be independent random 
matrices of dimension $d=d(N)\ge N$ of the form \eqref{eq:indep},
such that the variance pattern $(b_{ij}^2)$ of $H_k^N$ satisfies
$$
	\max_i\sum_j b_{ij}^2 =
	\min_i\sum_j b_{ij}^2 
	= 1,\qquad\quad
	\max_{ij} b_{ij}^2 = o((\log d)^{-3})
$$
for every $k,N$. Then
\begin{align*}
	&\lim_{N\to\infty} \|p(H_1^N,\ldots,H_m^N)\| =
	\|p(s_1,\ldots,s_m)\|\quad\mbox{a.s.}, \\
	&\lim_{N\to\infty} \ntr p(H_1^N,\ldots,H_m^N) =
	\tau(p(s_1,\ldots,s_m))\quad\mbox{a.s.}
\end{align*}
for every noncommutative polynomial $p$.
\end{cor}

Corollary \ref{cor:indepsfree} provides a large class of new examples of 
strongly asymptotically free random matrices. Let us highlight a 
particularly interesting case.

\begin{example}[Sparse Wigner matrices]
\label{ex:sparse}
Let $\mathrm{G}=([d],E)$ be a $k$-regular graph with $d$ vertices. A
\emph{$\mathrm{G}$-sparse Wigner matrix} is a $d\times d$ real 
symmetric random matrix $X$ such that 
$X_{ij}=k^{-\frac{1}{2}}g_{ij}1_{\{i,j\}\in E}$ for $i\ge j$, where
$\{g_{ij}:i\ge j\}$ are i.i.d.\ standard Gaussians. Note that
$X$ has only $kd$ nonzero entries.

Now consider any sequence of $k_N$-regular graphs $\mathrm{G}_N$ with 
$d_N$ vertices, and let $H_1^N,\ldots,H_m^N$ be independent 
$\mathrm{G}_N$-sparse Wigner matrices. Then Corollary \ref{cor:indepsfree} 
shows that $H_1^N,\ldots,H_m^N$ are strongly asymptotically free as soon 
as $k_N\gg (\log d_N)^3$.

This example is striking for at least two reasons. First, all but a 
vanishing fraction of the entries of the matrices $H_i^N$ are zero (for 
example, $d\log^4 d$ nonzero entries already suffice), so that strong 
asymptotic freeness is achieved here with far less randomness than is 
present in standard Wigner matrices. Second, no assumption whatsoever made 
on the graphs $\mathrm{G}_N$ except their regularity; in particular, the 
distributions of $H_i^N$ need not possess any special symmetries. Let us 
note that even weak asymptotic freeness was previously known in the 
present setting only under very strong restrictions on the variance 
pattern, cf.\ \cite{Shl96,Au18}.
\end{example}

Beyond norm bounds and asymptotic freeness, applying Theorems 
\ref{thm:mainsp} or \ref{thm:stieltjes} to the independent entry model 
\eqref{eq:indep} provides detailed information on the spectrum of $X$ for 
arbitrary variance patterns $b_{ij}^2$ satisfying the mild assumption 
\eqref{eq:indepf}. In the interest of brevity we do not spell out these 
conclusions further.

\subsection{Dependent entries}
\label{sec:exdep}

The aim of this section is to discuss some simple examples of random 
matrices with dependent entries. Unlike the independent entry model of the 
previous section, the only general nonasymptotic bound that was previously 
available in the dependent setting is the noncommutative Khintchine 
inequality \eqref{eq:nck} and analogous matrix concentration inequalities.

The following examples illustrate that, in many cases, our results are 
able to remove the dimensional factor in \eqref{eq:nck} under mild 
assumptions. To this end, note that for any random matrix $X$ with 
centered jointly Gaussian entries, we have $\EE\|X\|\gtrsim\sigma(X)$ by 
\eqref{eq:nck} and Remark \ref{rem:sa}. On the other hand, Corollary 
\ref{cor:norm} and Lemma \ref{lem:freenck} imply that 
$\EE\|X\|\lesssim\sigma(X)$ as soon as $v(X)(\log 
d)^{\frac{3}{2}}\lesssim\sigma(X)$. We aim to understand when the latter 
condition holds in concrete examples.

\subsubsection{Patterned random matrices}

Our first example is a model where independent Gaussians are 
placed in a matrix according to a given pattern. More precisely, let 
$g_1,\ldots,g_n$ be i.i.d.\ standard real Gaussian variables and let 
$S_1,\ldots,S_n$ be a partition of $[d]\times [d]$. We define 
$X$ such that $X_{jk}=d^{-\frac{1}{2}}g_i$ for $(j,k)\in S_i$; thus
\begin{equation}
\label{eq:pattern}
	X=\sum_{i=1}^n g_i A_i,\qquad\quad
	(A_i)_{jk} = \frac{1_{(j,k)\in S_i}}{\sqrt{d}}.
\end{equation}
Many classical patterned random matrix models, such as random Toeplitz 
or Hankel matrices, are special cases of this model; cf.\ \cite{Bos18}. 

\begin{lem}
\label{lem:pattern}
For the model \eqref{eq:pattern}, we have
$\EE\|X\| \asymp \sigma(X)$
when $\max_i|S_i|\lesssim \frac{d}{(\log d)^3}$.
\end{lem}

\begin{proof}
As $S_1,\ldots,S_n$ partition $[d]\times [d]$, we have
$$
	\sigma(X)^2
	\ge
	\ntr\bigg(\sum_i A_i^*A_i\bigg) =
	\frac{1}{d^2}\sum_i |S_i|=1.
$$
On the other hand, as $(X_{kl})_{(k,l)\in S_i}$ are 
independent for distinct $i$, we have $\mathrm{Cov}(X)=\bigoplus_i C_i$
where $C_i$ is the covariance matrix of $(X_{kl})_{(k,l)\in S_i}$. 
Therefore
$$
	v(X)^2 = \|\mathrm{Cov}(X)\| = \max_i \|C_i\| =
	 \max_i \frac{|S_i|}{d}.
$$
The assumption now immediately implies $v(X)(\log 
d)^{\frac{3}{2}}\lesssim \sigma(X)$.
\end{proof}

Lemma \ref{lem:pattern} shows that when $\max_i|S_i|\lesssim
\frac{d}{(\log d)^3}$, the 
dimensional factor in the noncommutative Khintchine inequality 
\eqref{eq:nck} is unnecessary. On the other hand, Gaussian Toeplitz 
matrices provide an example with $\max_i |S_i|=d$ for which the 
dimensional factor in the noncommutative Khintchine inequality is 
necessary: in this case $\sigma(X)=1$ and $\EE\|X\|\asymp 
\sqrt{\log d}$ \cite[\S 4.4]{Tro15}. Thus Lemma 
\ref{lem:pattern} is nearly the best one can hope for. This kind of 
``phase transition'' between regimes where the noncommutative Khintchine 
inequality is and is not accurate is a common feature that will be 
observed in several other examples.

For a general choice of pattern $S_1,\ldots,S_n$, the parameter 
$\sigma(X)$ may be difficult to compute explicitly. However, for special 
choices of patterns we can obtain much stronger information. The following 
simple example provides a model where strong asymptotic freeness arises 
for matrices that contain many dependent entries.

\begin{example}[Special patterned matrices]
\label{ex:specpat}
Suppose $S_1,\ldots,S_n$ satisfy the following:
\begin{enumerate}[1.]
\item Each $S_i$ is symmetric (that is, $(k,l)\in S_i\Leftrightarrow
(l,k)\in S_i$).
\item Each $S_i$ has at most one entry in each row of $[d]\times [d]$.
\item $\max_i |S_i|\le \frac{d}{(\log d)^4}$.
\end{enumerate}
The first assumption implies that each $A_i$ is a symmetric matrix.
The second assumption implies that $A_i^2$ is a diagonal matrix;
moreover, 
$$
	(\EE[X^2])_{kk} = \sum_i (A_i^2)_{kk} =
	\frac{1}{d}
	\sum_i 1_{S_i\text{ has an entry in row }k} =1
$$
for all $k$ as $S_1,\ldots,S_n$ partition $[d]\times[d]$, so that
$\EE[X^2]=\id$. The third assumption implies that 
$v(X)\le (\log d)^{-2}$. Matrices of this kind therefore satisfy the 
assumptions of Theorem \ref{thm:free}. Thus if $H_1^d,\ldots,H_m^d$ are 
independent matrices satisfying the above assumptions, then they are 
strongly asymptotically free as $d\to\infty$.
\end{example}

\subsubsection{Independent columns}

Our second example is the model where the columns $X_1,\ldots,X_d$ of the 
random matrix $X$ are independent centered Gaussian vectors with 
arbitrary covariance matrices $\Sigma_1,\ldots,\Sigma_d$. In this 
situation, all the relevant matrix parameters can be easily computed
in explicit form.

\begin{lem}
\label{lem:indepcol}
For the independent columns model, we have
$$
	\|\mathbf{E}[XX^*]\| = \Bigg\|\sum_{i=1}^d \Sigma_i\Bigg\|,\qquad
	\|\mathbf{E}[X^*X]\| = \max_i\tr[\Sigma_i],\qquad
	v(X)^2 = \max_i \|\Sigma_i\|.
$$
In particular,
$$
	\EE\|X\| \le (1+\varepsilon)\Bigg\{\Bigg\|\sum_{i=1}^d \Sigma_i
	\Bigg\|^{\frac{1}{2}}
	+ \max_i\tr[\Sigma_i]^{\frac{1}{2}}\Bigg\} +
	\frac{C}{\varepsilon}\max_i \|\Sigma_i\|^{\frac{1}{2}} (\log d)^{\frac{3}{2}}
$$
for any $\varepsilon>0$, where $C$ is a universal constant.
\end{lem}

\begin{proof}
It follows readily from the definition of $X$ that
$\mathbf{E}[XX^*] = \sum_i \Sigma_i$,
$\mathbf{E}[X^*X] = \sum_i \tr[\Sigma_i]\,e_ie_i^*$, and
$\mathrm{Cov}(X)=\bigoplus_i \Sigma_i$, which yields the first equation 
display. It remains to invoke Corollary \ref{cor:norm}, Lemma 
\ref{lem:freenck}, and Young's inequality.
\end{proof}

Lemma \ref{lem:indepcol} shows that we have $\EE\|X\|\asymp\sigma(X)$ in 
the independent column model as soon as the last term in the norm bound is 
dominated by either of the first two terms. For example, this is the case 
if each $\Sigma_i$ has sufficently large effective rank
$$
	\mathrm{rk}(\Sigma_i):=
	\frac{\tr[\Sigma_i]}{\|\Sigma_i\|} \gtrsim (\log d)^{3}.
$$
Conversely, when the effective rank is too small the dimensional factor in 
the noncommutative Khintchine inequality may be necessary: for example, in 
the special case $\Sigma_i=e_ie_i^*$ where $X$ is a diagonal matrix with 
i.i.d.\ standard Gaussians on the diagonal, it is readily seen that 
$\sigma(X)=1$ and $\EE\|X\|\asymp\sqrt{\log d}$.

On the other hand, we may have $\EE\|X\|\asymp\sigma(X)$ regardless of the 
effective rank when the first term in the norm bound dominates. For 
example, when $X$ has i.i.d.\ columns, that is, when 
$\Sigma_1=\cdots=\Sigma_d=\Sigma$, Lemma \ref{lem:indepcol} implies
$$
	\EE\|X\| \asymp \sqrt{d\|\Sigma\|}+\sqrt{\tr\Sigma}.
$$
This special case is well known, see, e.g., \cite[Lemma 5.4]{vH17}.

\subsubsection{Independent blocks}

Our third example is the model
\begin{equation}
\label{eq:iblock}
	X = \begin{bmatrix}
	X^{1,1} & \cdots & X^{1,m} \\
	\vdots & \ddots & \vdots \\
	X^{m,1} & \cdots & X^{m,m}
	\end{bmatrix}
\end{equation}
where $X^{i,j}$ are independent $r\times r$ random matrices.

\begin{lem}
\label{lem:iblock}
Consider the model \eqref{eq:iblock} where $X^{i,j}$ are independent 
centered Gaussian random matrices. Then we have
\begin{align*}
	\EE\|X\| &\le (1+\varepsilon)
	\Bigg\{\max_i\Bigg\|
	\sum_j \EE X^{i,j}(X^{i,j})^*\Bigg\|^{\frac{1}{2}} +
	\max_j\Bigg\|
	\sum_i \EE (X^{i,j})^*X^{i,j}\Bigg\|^{\frac{1}{2}}\Bigg\}\\
	&\qquad
	+\frac{C}{\varepsilon}\max_{ij}v(X^{i,j})\,(\log rm)^{\frac{3}{2}}
\end{align*}
for any $\varepsilon>0$, where $C$ is a universal constant.
\end{lem}

\begin{proof}
A simple computation shows that $\|\EE XX^*\|=\max_i\|\sum_j 
X^{i,j}(X^{i,j})^*\|$ and
$\|\EE X^*X\|=\max_j\|\sum_i (X^{i,j})^*X^{i,j}\|$.
Moreover, as the blocks $X^{i,j}$ are
independent, $\mathrm{Cov}(X)=\bigoplus_{i,j}\mathrm{Cov}(X^{i,j})$
and thus $v(X)^2 = \|\mathrm{Cov}(X)\|=\max_{i,j}v(X^{i,j})^2$.
It remains to invoke Corollary \ref{cor:norm}, Lemma
\ref{lem:freenck}, and Young's inequality.
\end{proof}

The independent block model \eqref{eq:iblock} may be viewed as 
intermediate between the independent entry model \eqref{eq:indep} and 
fully dependent random matrices. As a particularly simple example,
consider the case where $X^{i,j}$ are all i.i.d.\ copies of the same 
centered Gaussian random matrix $Z$. Then Lemma \ref{lem:iblock} yields
$$
	\EE\|X\| \lesssim \sqrt{m}\,\sigma(Z)+v(Z)\,
	(\log rm)^{\frac{3}{2}},
$$
so that $\EE\|X\|\asymp\sigma(X)$ as soon as $\sigma(Z)^2 \gtrsim 
\frac{\log(rm)^3}{m} v(Z)^2$. On the other hand, the case $m=1$ encodes 
any centered Gaussian matrix, for which the dimensional factor of the 
noncommutative Khintchine inequality cannot be removed.

\subsubsection{Gaussian on a subspace}

The examples discussed so far all feature a form of ``structured 
independence'', where certain subsets of entries are assumed to be 
independent. This is by no means necessary for the validity of our bounds. 
Our fourth example illustrates a simple situation that lacks
any independence.

A matrix with i.i.d.\ real Gaussian entries may be viewed equivalently as 
the model defined by the isotropic Gaussian distribution on 
$\M_d(\mathbb{R})$. This model may generalized as follows. Let 
$\mathcal{M}\subseteq \M_d(\mathbb{R})$ be any linear subspace of 
dimension $\dim\mathcal{M}=k$ of the space of $d\times d$ real matrices, 
and let $X$ be the random matrix defined by the isotropic Gaussian 
distribution on $\mathcal{M}$. Equivalently, 
$$
	X = \sum_{i=1}^k g_iA_i
$$
where $A_1,\ldots,A_k$ is any orthonormal basis of $\mathcal{M}$ (that is, 
$\tr[A_i^*A_j]=\delta_{ij}$) and $g_1,\ldots,g_k$ are i.i.d.\ real 
standard Gaussian variables. Note that this model has fully dependent 
entries when $\mathcal{M}$ is in general position.

\begin{lem}
\label{lem:gsub}
When $X$ is an isotropic real Gaussian matrix on a linear subspace 
$\mathcal{M}\subseteq \M_d(\mathbb{R})$, we have
$\EE\|X\|\asymp\sigma(X)$ as soon as 
$\dim\mathcal{M}\gtrsim d\log^3 d$.
\end{lem}

\begin{proof}
Let $\dim\mathcal{M}=k$. Then $\sigma(X)^2 \ge \ntr[\sum_i A_i^*A_i] = 
\frac{k}{d}$. On the other hand, note that $\mathrm{Cov}(X)=\sum_{i=1}^n 
\iota(A_i)\iota(A_i)^*$, where $\iota:\M_d(\mathbb{R})\to\mathbb{R}^{d^2}$ 
maps a matrix to its vector of entries. But here
$\iota(A_i)$ were assumed to be orthonormal, so $\mathrm{Cov}(X)$ is a 
projection matrix. Thus $v(X)^2=\|\mathrm{Cov}(X)\|=1$. As explained at 
the beginning of Section \ref{sec:exdep}, We therefore have 
$\EE\|X\|\asymp\sigma(X)$ as soon as $(\log d)^3\lesssim \frac{k}{d}$. 
\end{proof}

When $\mathcal{M}=\mathrm{span}\{e_ie_j^*:|i-j|\le r\}$, we have 
$\dim\mathcal{M}\asymp (r+1)d$, $\sigma(X)\asymp\sqrt{r+1}$, and 
$\EE\|X\|\ge \EE\max_{ij}|X_{ij}|\gtrsim \sqrt{\log d}$. Thus the 
conclusion of Lemma \ref{lem:gsub} may fail when $\dim\mathcal{M}\ll d\log 
d$. While this particular example is rather special (as $X$ has 
independent entries), the beauty of Lemma \ref{lem:gsub} is that it 
applies to \emph{any} $\mathcal{M}$.

\subsection{Generalized sample covariance matrices}
\label{sec:samplecov}

Let $X$ be any $d\times m$ random matrix with centered jointly Gaussian 
entries. We 
will refer to $XX^*$ as a generalized sample covariance matrix. Indeed, as 
$\frac{1}{m}XX^*=\frac{1}{m}\sum_{i=1}^m X_iX_i^*$ in terms of the columns 
$X_1,\ldots,X_m$ of $X$, we see that $\frac{1}{m}XX^*$ is a sample 
covariance matrix in the special case that the data $X_1,\ldots,X_m$ are 
i.i.d.\ (see, e.g., \cite{KL17}). In the general setting, one may still 
think of $\frac{1}{m}XX^*$ as a sample covariance matrix, but where the 
samples need not be independent or identically distributed.

The main question of interest in this setting is to estimate the deviation 
of the sample covariance matrix from the actual covariance matrix 
$\|XX^*-\EE XX^*\|$. We presently show that an estimate of this kind 
can be derived from Theorem \ref{thm:mainsp} using a simple variant of the 
linearization trick that is used in Theorem \ref{thm:free}. While 
linearization generally yields asymptotic results for any polynomial, the 
present example illustrates that nonasymptotic bounds can be derived for 
specific polynomials by a careful analysis of the linearization argument.
Alternatively, the interpolation method used in the proofs of our
main results can be adapted directly to yield quantitative bounds for 
polynomials (we do not pursue this here).

\begin{thm}
\label{thm:samplecov}
Let $A_1,\ldots,A_n$ be arbitrary $d\times m$ matrices with complex 
entries, and define $X$ and $X_{\rm free}$ as in \eqref{eq:model}
with $A_0=0$. Then we have
\begin{align*}
	\EE\|XX^*-\EE XX^*\| &\le
	\|X_{\rm free}X_{\rm free}^*-\EE XX^*\otimes \id\|
	\\ &\qquad
	+ C\{\sigma(X)\tilde v(X)\log^{\frac{3}{4}}(d+m)
	+ \tilde v(X)^2\log^{\frac{3}{2}}(d+m)\}, 
\end{align*}
where $C$ is a universal constant.
\end{thm}

The proof of Theorem \ref{thm:samplecov} will be given at the end of this 
section. To clarify its meaning, it is instructive to note that $\EE XX^*= 
(\mathrm{id}\otimes\tau)[X_{\rm free}X_{\rm free}^*]$; therefore, 
$\|X_{\rm free}X_{\rm free}^*-\EE XX^*\otimes\id\|$ is precisely the free 
analogue of $\|XX^*-\EE XX^*\|$.

In order to apply Theorem \ref{thm:samplecov} in concrete situations, we 
must be able to compute or bound its right-hand side. To this end, the 
following result may be viewed as the direct analogue of Lemma 
\ref{lem:freenck} in the present setting.

\begin{prop}
\label{prop:samplefree}
In the setting of Theorem \ref{thm:samplecov}, we have
$$
	\tfrac{1}{5}\Gamma
	\le
	\|X_{\rm free}X_{\rm free}^*-\EE XX^*\otimes \id\|
	\le
	\Gamma
$$
with
$$
	\Gamma:=
	2\|\EE[X\,\EE[X^*X]\,X^*]\|^{\frac{1}{2}}
	+ \|\EE X^*X\|.
$$
\end{prop}

\begin{proof}
We use the standard construction of a free semicircular family on Fock 
space, cf.\ \cite[pp.\ 102--108]{NS06} or \cite[\S 9.9]{Pis03} (this 
construction will not be used in our main results). Let 
$\mathcal{F}(\mathbb{C}^n):=\mathbb{C}\omega\oplus\bigoplus_{k=1}^\infty 
(\mathbb{C}^n)^{\otimes k}$ be the free Fock space over $\mathbb{C}^n$, where 
the unit vector $\omega$ is called the vacuum vector. For any 
$h\in\mathbb{C}^n$, the creation operator $l(h)\in 
B(\mathcal{F}(\mathbb{C}^n))$ is defined by setting
for any $x_1,\ldots,x_k\in\mathbb{C}^n$
$$
	l(h)\omega := h,\qquad\quad
	l(h)(x_1\otimes\cdots\otimes x_k)
	:=h\otimes x_1\otimes\cdots\otimes x_k.
$$
Then the self-adjoint operators $s_1,\ldots,s_n$ 
defined by $s_i = l(e_i)+l(e_i)^*$ form a free semicircular family with 
respect to the vacuum state $\tau(x):=\langle\omega,x\omega\rangle$
on $B(\mathcal{F}(\mathbb{C}^n))$.

As we assumed $A_0=0$, we may represent $X_{\rm free}=U+V$ with
$U := \sum_i A_i\otimes l(e_i)$ and $V := \sum_i A_i\otimes l(e_i)^*$.
The property $l(e_i)^*l(e_j)=\delta_{ij}\id$ (which is readily verified 
from the definition of $l(h)$) yields the identities
$$
	VV^* = 	
	\sum_i A_iA_i^*\otimes\id = \EE XX^*\otimes\id,\qquad
	U^*U = \sum_i A_i^*A_i\otimes\id = \EE X^*X\otimes\id,
$$
and
$$
	VU^*UV^* =
	\sum_i A_i\,
	\EE[X^*X]\, A_i^*\otimes\id
	=
	\EE[X\,\EE[X^*X]\,X^*]\otimes\id.
$$
We therefore obtain by the triangle inequality
$$
	\|X_{\rm free}X_{\rm free}^* - \EE XX^*\otimes\id\|  =
	\|UV^*+VU^* + UU^*\| \le
	2\|UV^*\|+\|U\|^2 = \Gamma,
$$
establishing the upper bound.

To prove the lower bound, note first that
\begin{multline*}
	\|X_{\rm free}X_{\rm free}^* - \EE XX^*\otimes\id\|
	\ge
	\sup_{\|v\|=1}\langle v\otimes\omega,
	(X_{\rm free}X_{\rm free}^* - \EE XX^*\otimes\id)^2 \,
	v\otimes\omega\rangle^{\frac{1}{2}} \\
	=
	\sup_{\|v\|=1}\langle v\otimes\omega,
	VU^*UV^*\,
	v\otimes\omega\rangle^{\frac{1}{2}}
	= \|\EE[X\,\EE[X^*X]\,X^*]\|^{\frac{1}{2}},
\end{multline*}
where we used $U^*(v\otimes\omega)=0$. On the other hand, by the
reverse triangle inequality
$$
	\|X_{\rm free}X_{\rm free}^* - \EE XX^*\otimes\id\|\ge
	\|U\|^2-2\|UV^*\| =
	\|\EE X^*X\|-2\|\EE[X\,\EE[X^*X]\,X^*]\|^{\frac{1}{2}}.
$$
The lower bound follows using
$\max(a,b)\ge \frac{4}{5}a+\frac{1}{5}b$.
\end{proof}

To illustrate these bounds, consider the case where the 
columns of $X$ are i.i.d.\ centered Gaussian vectors with covariance 
$\Sigma$ (so that $\frac{1}{m}XX^*$ is a classical sample 
covariance matrix). Then Theorem \ref{thm:samplecov} and
Proposition \ref{prop:samplefree} yield
\begin{multline*}
	\EE\|\tfrac{1}{m}XX^*-\Sigma\| \le
	\|\Sigma\|
	\bigg\{2\sqrt{\frac{\mathrm{rk}(\Sigma)}{m}}
	+ \frac{\mathrm{rk}(\Sigma)}{m}\bigg\} + \mbox{}
	\\
 	C\|\Sigma\|
	\bigg\{ 
	\bigg(1\vee \frac{\mathrm{rk}(\Sigma)}{m}\bigg)^{\frac{3}{4}}
	\frac{\log^{\frac{3}{4}}(d+m)}{m^{\frac{1}{4}}}
	+ 
	\bigg(1\vee \frac{\mathrm{rk}(\Sigma)}{m}\bigg)^{\frac{1}{2}}
	\frac{\log^{\frac{3}{2}}(d+m)}{m^{\frac{1}{2}}}
	\bigg\}
\end{multline*}
where $\mathrm{rk}(\Sigma):=\tr[\Sigma]/\|\Sigma\|$, and we used
Lemma \ref{lem:indepcol} to compute $\sigma(X)$ and $v(X)$. The leading 
term in this bound dominates when $\mathrm{rk}(\Sigma)$ is not too 
small. The latter restriction is not optimal: it was shown in \cite{KL17} 
that when $X$ has i.i.d.\ columns, 
$\EE\|\tfrac{1}{m}XX^*-\Sigma\|$ always agrees with the leading term in 
the above inequality up to a universal constant. On the other hand, our 
general bounds apply to arbitrary nonhomogeneous random matrices $X$, and 
yield the sharp constant in the leading-order term.
(In the special case that $X$ has independent Gaussian entries, a 
bound with a slightly weaker leading-order term was obtained in 
\cite{CHZ20}.)

We now turn to the proof of Theorem \ref{thm:samplecov}. The key idea is 
the following lemma, which provides an explicit linearization of the 
polynomial $(X,X^*)\mapsto XX^*+A$.

\begin{lem}
\label{lem:linquad}
Let $A_\varepsilon = (\|\EE XX^*\|+4\varepsilon^2)\id - \EE XX^*$, and 
define
$$
	\breve X_\varepsilon = \begin{bmatrix}
	0 & 0 & X & A^{\frac{1}{2}}_\varepsilon \\
	0 & 0 & 0 & 0 \\
	X^* & 0 & 0 & 0 \\
	A^{\frac{1}{2}}_\varepsilon & 0 & 0 & 0
	\end{bmatrix},
	\qquad\quad
	\breve X_{\rm free,\varepsilon} = \begin{bmatrix}
	0 & 0 & X_{\rm free} & A^{\frac{1}{2}}_\varepsilon\otimes\id \\
	0 & 0 & 0 & 0 \\
	X_{\rm free}^* & 0 & 0 & 0 \\
	A^{\frac{1}{2}}_\varepsilon\otimes\id & 0 & 0 & 0
	\end{bmatrix}.
$$
Then we have
\begin{align*}
	&\spc(\breve X_\varepsilon)\subseteq \spc(\breve X_{\rm 
	free,\varepsilon})+[-\varepsilon,\varepsilon]
	\qquad\Longrightarrow \\
	&\begin{cases}
	\lambda_{+}(XX^*+A_\varepsilon)^{\frac{1}{2}} \le 
	\lambda_{+}(X_{\rm free}X_{\rm free}^*
	+A_\varepsilon\otimes\id)^{\frac{1}{2}}+\varepsilon,\\
	\lambda_{-}(XX^*+A_\varepsilon)^{\frac{1}{2}} \ge 
	\lambda_{-}(X_{\rm free}X_{\rm free}^*
	+A_\varepsilon\otimes\id)^{\frac{1}{2}}-\varepsilon
	\end{cases}
\end{align*}
for any $\varepsilon>0$, where
$\lambda_+(Z):=\sup \spc(Z)$ and $\lambda_-(Z):=\inf\spc(Z)$.
\end{lem}

\begin{proof}
By Remark \ref{rem:sa}, we have
$$
	\spc(\breve X_\varepsilon)\cup\{0\} =
	\spc((XX^*+A_\varepsilon)^{\frac{1}{2}})
	\cup
	{-\spc((XX^*+A_\varepsilon)^{\frac{1}{2}})}
	\cup\{0\},
$$
and analogously for $\breve X_{\rm free,\varepsilon}$. 
If $\spc(\breve X_\varepsilon)\subseteq \spc(\breve X_{\rm
free,\varepsilon})+[-\varepsilon,\varepsilon]$, then clearly
$$
	\lambda_{+}(XX^*+A_\varepsilon)^{\frac{1}{2}} \le
        \lambda_{+}(X_{\rm free}X_{\rm free}^*
	+A_\varepsilon\otimes\id)^{\frac{1}{2}}+\varepsilon.
$$
On the other hand, as $\breve X_{\rm free,\varepsilon}$ can have a zero
eigenvalue, it follows that either
$$
	\lambda_{-}(XX^*+A_\varepsilon)^{\frac{1}{2}} \ge \lambda_{-}(X_{\rm
        free}X_{\rm free}^*
	+A_\varepsilon\otimes\id)^{\frac{1}{2}}-\varepsilon
$$
or $\lambda_{-}(XX^*+A_\varepsilon)^{\frac{1}{2}}\le\varepsilon$.
But the latter is impossible, as 
$\lambda_{-}(XX^*+A_\varepsilon)^{\frac{1}{2}}\ge 2\varepsilon$.
\end{proof}

We can now complete the proof of Theorem \ref{thm:samplecov}.

\begin{proof}[Proof of Theorem \ref{thm:samplecov}]
We adopt throughout the proof the notation and conclusions
Lemma \ref{lem:linquad}. 
By Remark \ref{rem:sa}, we have $\sigma_*(\breve X_\varepsilon)=
\sigma_*(X)$ and $\tilde v(\breve X_\varepsilon) \le 
2^{\frac{1}{4}}\tilde v(X)$.
We may therefore apply Theorem \ref{thm:mainsp} to $\breve X_\varepsilon$ 
to obtain
\begin{align*}
	&\mathbf{P}\big[
	\lambda_{+}(XX^*+A_{\varepsilon(t)})^{\frac{1}{2}} \le
	\lambda_{+}(X_{\rm free}X_{\rm free}^*
	+A_{\varepsilon(t)}\otimes\id)^{\frac{1}{2}}
	+\varepsilon(t),
\\ &
\phantom{\mathbf{P}[}
	\lambda_{-}(XX^*+A_{\varepsilon(t)})^{\frac{1}{2}} \ge
	\lambda_{-}(X_{\rm free}X_{\rm free}^*
	+A_{\varepsilon(t)}\otimes\id)^{\frac{1}{2}}
	-\varepsilon(t)\big]
	\ge 1-e^{-t^2}
\end{align*}
for all $t\ge 0$, where $\varepsilon(t)=c\{\tilde v(X)(\log^{\frac{3}{4}} 
(d+m)+\sigma_*(X)t\}$ for a universal constant $c$. (Note that
$\breve X_\varepsilon$ is $2(d+m)$-dimensional, but we may bound
$\log(2(d+m))\lesssim \log(d+m)$ as $d+m\ge 2$ for notational simplicity.)
Now note that 
$$
	\lambda_{\pm}(XX^*+A_\varepsilon) =
	\lambda_{\pm}(XX^*-\EE XX^*)+\|\EE XX^*\|+4\varepsilon^2,
$$
and analogously for $X_{\rm free}$. Moreover, we have
$$
	\lambda_-(X_{\rm free}X^*_{\rm free}+A_\varepsilon\otimes\id)\le
	\lambda_+(X_{\rm free}X^*_{\rm free}+A_\varepsilon\otimes\id)
	\le 5\sigma(X)^2 + 4\varepsilon^2
$$
by Lemma \ref{lem:freenck}. Thus we obtain
\begin{align*}
	&
	\lambda_{+}(XX^*+A_{\varepsilon})^{\frac{1}{2}} \le
	\lambda_{+}(X_{\rm free}X_{\rm free}^*
	+A_{\varepsilon}\otimes\id)^{\frac{1}{2}}
	+\varepsilon
	\qquad \Longrightarrow
	\\ &
	\lambda_{+}(XX^*-\EE XX^*) \le
	\lambda_{+}(X_{\rm free}X_{\rm free}^*
	-\EE XX^*\otimes\id) +
	2\varepsilon\sqrt{5\sigma(X)^2+4\varepsilon^2}
	+\varepsilon^2
\end{align*}
by squaring both sides of the first inequality and applying the previous 
two equation displays. Analogously, using $(y-\varepsilon)_+^2\ge 
y^2-2\varepsilon y-\varepsilon^2$
for $y,\varepsilon\ge 0$ yields
\begin{align*}
	&
	\lambda_{-}(XX^*+A_{\varepsilon})^{\frac{1}{2}} \ge
	\lambda_{-}(X_{\rm free}X_{\rm free}^*
	+A_{\varepsilon}\otimes\id)^{\frac{1}{2}}
	-\varepsilon
	\qquad \Longrightarrow
	\\ &
	\lambda_{-}(XX^*-\EE XX^*) \ge
	\lambda_{-}(X_{\rm free}X_{\rm free}^*
	-\EE XX^*\otimes\id) -
	2\varepsilon\sqrt{5\sigma(X)^2+4\varepsilon^2}
	-\varepsilon^2.
\end{align*}
But as $\|Z\|=\max(\lambda_+(Z),-\lambda_-(Z))$, we have shown that
$$
	\mathbf{P}\big[
	\|XX^*-\EE XX^*\| >
	\|X_{\rm free}X_{\rm free}^*
	-\EE XX^*\otimes\id\|
	+5\sigma(X)\varepsilon(t) + 5\varepsilon(t)^2
	\big]
	\le e^{-t^2}.
$$
The conclusion follows by integrating this tail bound and using
$\sigma_*(X)\le \tilde v(X)$.
\end{proof}

\subsection{Smallest singular value}
\label{sec:sigmamin}

The initial motivation for the results of this paper arose from the 
question whether classical matrix concentration inequalities can be 
sharpened. Consequently, the focus of our examples has been on norm bounds 
for various random matrix models. Unlike classical matrix concentration 
inequalities, however, our main results enable us to control the entire 
spectrum and not merely the spectral norm. This makes it possible to 
address questions that are outside the scope of classical matrix 
concentration inequalities.

As an illustration, let us derive in this section a bound on the 
smallest singular value $\svl_{\rm min}(X) := \inf\spc(|X|)$ of a 
general Gaussian random matrix $X$.

\begin{thm}
\label{thm:smsing}
Let $A_0,\ldots,A_n$ be arbitrary $d\times m$ matrices with complex 
entries, and define $X$ and $X_{\rm free}$ as in \eqref{eq:model}.
Then we have
$$
	\mathbf{P}[\svl_{\rm min}(X) \le
	\svl_{\rm min}(X_{\rm free}) - 
	C\tilde v(X)\log^{\frac{3}{4}}(d+m) -
	C\sigma_*(X)t] \le e^{-t^2}
$$
for all $t\ge 0$, where $C$ is a universal constant. In particular,
$$
	\EE[\svl_{\rm min}(X)] \ge 
	\svl_{\rm min}(X_{\rm free})
	-C\tilde v(X)\log^{\frac{3}{4}}(d+m).
$$
\end{thm}

\begin{proof}
The conclusion of Lemma \ref{lem:linquad} continues to hold \emph{verbatim}
if we define $A_\varepsilon := 4\varepsilon^2\id$ and exchange the roles 
of $X$ and $X^*$. Thus Theorem \ref{thm:mainsp} yields
$$
	\mathbf{P}[\lambda_-(|X|^2+4\varepsilon(t)^2\id)^{\frac{1}{2}}
	\le
	\lambda_-(|X_{\rm free}|^2+4\varepsilon(t)^2\id)^{\frac{1}{2}}
	- \varepsilon(t)] \le e^{-t^2}
$$
for all $t\ge 0$, where $\varepsilon(t):=
C\{\tilde v(X)\log^{\frac{3}{4}}(d+m) + \sigma_*(X)t\}$. The conclusion
follows as $\lambda_-(|X_{\rm 
free}|^2+4\varepsilon(t)^2\id)^{\frac{1}{2}}\ge \svl_{\rm min}(X_{\rm 
free})$ and $\lambda_-(|X|^2+4\varepsilon(t)^2\id)^{\frac{1}{2}}
\le \svl_{\rm min}(X)+2\varepsilon(t)$. 
The expectation bound follows by 
integrating the tail bound and $\sigma_*(X){\,\le\,}\tilde v(X)$.
\end{proof}

While $\svl_{\rm min}(X_{\rm free})$ can be computed using the methods 
of \cite[\S 5]{Leh99}, the following crude bound already yields 
nontrivial results in various examples.

\begin{lem}
\label{lem:smsingfree}
Consider the setting of Theorem \ref{thm:smsing} with $\EE X=0$. Then we 
have
$$
	\svl_{\rm min}(X_{\rm free}) \ge
	\svl_{\rm min}(\EE X^*X)^{\frac{1}{2}} -
	\|\EE XX^*\|^{\frac{1}{2}}.
$$
\end{lem}

\begin{proof}
We use the same Fock space construction $X_{\rm free}=U+V$ as in the 
proof of Proposition \ref{prop:samplefree}. Then we may estimate
by the reverse triangle inequality
$$
	\svl_{\rm min}(X_{\rm free}) =
	\inf_{\|x\|=1} \|(U+V)x\| \ge
	\inf_{\|x\|=1} \|Ux\| - \|V\|,	
$$
and the conclusion follows as
$\|Ux\|^2 = \langle x,(\EE X^*X\otimes \id)x\rangle$ and
$\|V\|^2=\|\EE XX^*\|$.
\end{proof}

The above results provide information on the smallest singular value of 
random matrices that may be nonhomogeneous and have dependent entries. 
Even suboptimal bounds on the smallest singular value in this setting are 
fundamentally outside the scope of classical matrix concentration 
inequalities. As a simple example, we consider a variant of the patterned 
matrices of Example \ref{ex:specpat}.

\begin{example}[Special patterned matrices]
\label{ex:smsmsp}
Let $g_1,\ldots,g_n$ be i.i.d.\ standard Gaussian variables, and let
$S_1,\ldots,S_n$ be a partition of the set $[d]\times[m]$ with $d\ge m$.
Then we can define the $d\times m$ patterned random matrix $X$ such that
$X_{jk} = g_i$ for $(j,k)\in S_i$. Let us assume in 
addition that each 
$S_i$ has at most one entry in each row and column of $[d]\times[m]$.
Then we may readily compute as in Example \ref{ex:specpat} that
$\EE[X^*X]=d\id$ and $\EE[XX^*]=m\id$, so that Theorem \ref{thm:smsing} 
and Lemma \ref{lem:smsingfree} yield
$$
	\EE[\svl_{\rm min}(X)] \ge
	\sqrt{d}-\sqrt{m} - C 
	d^{\frac{1}{4}}(\log d)^{\frac{3}{4}} \max_i|S_i|^{\frac{1}{4}}.
$$
When each $|S_i|=1$, that is, when $X$ is a rectangular matrix with 
i.i.d.\ standard Gaussian entries, such a bound is well known (e.g., 
\cite{DS01}) and is in agreement with the classical asymptotics of the 
smallest singular value in the proportional dimension regime $d\to\infty$, 
$m=\gamma d$ with $\gamma\in(0,1)$ due to Bai and Yin \cite{BY93}.
The present results show that the same bound remains valid to leading 
order even if we introduce considerable dependence among the 
matrix entries: for example, in the proportional dimension regime it 
suffices that $\max_i|S_i| \ll \frac{d}{(\log d)^3}$.
\end{example}

\begin{rem}
The above results are meaningful only when
$\svl_{\rm min}(X_{\rm free})>0$. When $\svl_{\rm min}(X_{\rm 
free})=0$ (for example, in square case $d=m$ of Example \ref{ex:smsmsp}), 
it may still be the case that $X$ is invertible a.s.\ even though $X_{\rm 
free}$ is not, but the problem of quantitatively estimating
$\svl_{\rm min}(X)$ in this setting is of a fundamentally different 
nature. At present, results of the latter kind for nonhomogeneous random 
matrices are known only under restrictive structural assumptions 
\cite{RZ16}.
\end{rem}

\section{Preliminaries}
\label{sec:prelim}

The aim of this section is to recall some mathematical background and to 
introduce a few basic estimates that will be used in the remainder of the 
paper. 

\subsection{Free probability}
\label{sec:free}

We begin by recalling some basic notions of free probability; the reader 
is referred to \cite{NS06} for an introduction to this topic. 

For our purposes, a \emph{unital $C^*$-algebra} may be thought of 
concretely as an algebra $\mathcal{A}$ of bounded operators on a complex 
Hilbert space which is self-adjoint ($a\in\mathcal{A}$ implies 
$a^*\in\mathcal{A}$), is closed in the operator norm, and contains the 
identity $\id\in\mathcal{A}$. A \emph{state} is a linear functional 
$\tau:\mathcal{A}\to\mathbb{C}$ that is positive $\tau(a^*a)\ge 0$ and 
unital $\tau(\id)=1$. A state is called \emph{faithful} if $\tau(a^*a)=0$ 
implies $a=0$.

\begin{defn}
A \emph{$C^*$-probability space} is a pair $(\mathcal{A},\tau)$, where
$\mathcal{A}$ is a unital $C^*$-algebra and $\tau$ is a faithful state.
\end{defn}

The simplest example of a $C^*$-probability space is 
$(\M_d(\mathbb{C}),\ntr)$. The introduction of general $C^*$-probability 
spaces enables us to extend computations involving matrices and traces to 
infinite-dimensional operators. The assumption that $\tau$ is faithful 
ensures that $\|a\|=\lim_{p\to\infty} \tau(|a|^p)^{\frac{1}{p}}$ 
\cite[Proposition 3.17]{NS06}.

The basic infinite-dimensional object of interest in this paper is a free 
semicircular family. We will define this notion combinatorially as in 
\cite[p.\ 128]{NS06}. For any integer $p$, denote by $\mathrm{P}_2([p])$ 
the collection of all pairings of $[p]:=\{1,\ldots,p\}$, that is, of 
partitions of $[p]$ each of whose blocks consists of exactly two elements. 
We denote by $\mathrm{NC}_2([p])\subseteq\mathrm{P}_2([p])$ the collection 
of those pairings $\pi$ that are \emph{noncrossing}, i.e., that do not 
contain $\{i,j\},\{k,l\}\in\pi$ so that $i<k<j<l$.

\begin{defn}
\label{defn:freefam}
A family $s_1,\ldots s_n\in\mathcal{A}$ of self-adjoint elements in a 
$C^*$-probability space $(\mathcal{A},\tau)$ is called a \emph{free
semicircular family} if
$$
	\tau(s_{k_1}\cdots s_{k_p}) =
	\sum_{\pi\in\mathrm{NC}_2([p])}
	\prod_{\{i,j\}\in\pi}\delta_{k_ik_j}
$$
for every $p\ge 1$, $k_1,\ldots,k_p\in[n]$.
\end{defn}

The elements $s_i$ are ``semicircular'' in the sense that for 
$p\in\mathbb{N}$,
$$
	\tau(s_i^p) = |\mathrm{NC}_2([p])| =
	\int_{-2}^2 x^p\cdot\tfrac{1}{2\pi}\sqrt{4-x^2}\,dx
$$
are the moments of the standard semicircle distribution, cf.\ \cite[p.\ 
123 and p.\ 29]{NS06}. The latter is precisely the limiting spectral 
distribution of large Wigner matrices. In particular, note that
$\|s_i\|=\lim_{p\to\infty} \tau(s_i^{2p})^{\frac{1}{2p}}=2$.

More generally, the weak asymptotic freeness theorem of Voiculescu 
\cite{Voi91} states that a free semicircular family arises as the limiting 
object associated to independent Wigner matrices. A self-contained proof 
of this fact may be readily obtained as a special case of the 
argument in Section \ref{sec:pfweakfree} below.

\begin{thm}[Voiculescu]
\label{thm:voic}
Let $G_1^N,\ldots,G_n^N$ be independent standard Wigner matrices in the 
sense of Definition \ref{defn:wigner}. Then we have
$$
	\lim_{N\to\infty}\EE[\ntr(G_{k_1}^N\cdots G_{k_p}^N)] =
	\tau(s_{k_1}\cdots s_{k_p})
$$
for every $p\ge 1$, $k_1,\ldots,k_p\in[n]$.
\end{thm}

We now turn our attention to the basic random matrix model 
\eqref{eq:model} of this paper. In the proofs of our main results, it will 
suffice to 
consider self-adjoint coefficient matrices 
$A_0,\ldots,A_n\in\M_d(\mathbb{C})_{\rm sa}$ due to Remark \ref{rem:sa}.
In addition to $X$ and $X_{\rm free}$ defined in \eqref{eq:model}, we also 
introduce the intermediate model
\begin{equation}
\label{eq:nmodel}
	X^N := A_0\otimes\id + \sum_{i=1}^n A_i\otimes G_i^N,
\end{equation}
where $G_1^N,\ldots,G_n^N$ are independent 
standard Wigner matrices of dimension $N$. Theorem \ref{thm:voic}
enables us to compute the limiting spectral statistics of $X^N$.

\begin{cor}
\label{cor:weakfree}
Let $A_0,\ldots,A_n\in \M_d(\mathbb{C})_{\rm sa}$. Then
$$
	\lim_{N\to\infty}\EE[\ntr f(X^N)] =
	({\ntr}\otimes\tau)(f(X_{\rm free}))
$$
for any polynomial or bounded continuous function 
$f:\mathbb{R}\to\mathbb{C}$.
\end{cor}

\begin{proof}
For the function $f(x)=x^p$ with $p\in\mathbb{N}$, we compute explicitly
$$
	\EE\ntr[(X^N)^p] =
	\sum_{i_1,\ldots,i_p=1}^n
	\ntr(A_{i_1}\cdots A_{i_p}) \EE[\ntr G_{i_1}\cdots G_{i_p}] 
	\xrightarrow{N\to\infty}
	({\ntr}\otimes\tau)(X_{\rm free}^p)
$$
by Theorem \ref{thm:voic}. The conclusion extends to any polynomial $f$ by 
linearity. For bounded continuous $f$, it remains to note that as 
$\|X_{\rm free}\|\le 2\sum_{i=0}^n\|A_i\|<\infty$, moment 
convergence implies weak convergence \cite[p.\ 116]{NS06}.
\end{proof}

We finally discuss a number of methods to compute or estimate the spectral 
statistics of $X_{\rm free}$. First, we note that the moments of $X_{\rm 
free}$ are readily computed using Definition \ref{defn:freefam}: for every 
$p\in\mathbb{N}$, we obtain
\begin{equation}
\label{eq:xfreemoments}
	({\ntr}\otimes\tau)(X_{\rm free}^p) =
	\sum_{\pi\in\mathrm{NC}_2([p])}
	\sum_{(i_1,\ldots,i_p)\sim\pi}
	\ntr(A_{i_1}\cdots A_{i_p}),
\end{equation}
where $(i_1,\ldots,i_p)\sim\pi$ denotes that
$i_k=i_l$ for every $\{k,l\}\in\pi$.

An explicit expression for the norm $\|X_{\rm free}\|$ was given in Lemma 
\ref{lem:lehner} above. This fundamental result was proved by Lehner 
\cite[Corollary 1.5]{Leh99}, where it is formulated only in the case that 
$A_0\ge 0$ is positive semidefinite. However, the general formulation is 
readily derived from this special case.

\begin{proof}[Proof of Lemma \ref{lem:lehner}]
We first note that $t:=\|X_{\rm free}\|\ge 
\|(\mathrm{id}\otimes\tau)(X_{\rm free})\|=\|A_0\|$. Thus
$X_{\rm free}+t\id\ge 0$ and $A_0+t\id\ge 0$.
Applying \cite[Corollary 1.5]{Leh99} yields
$$
	\|X_{\rm free}+t\id\| = 
	\inf_{Z>0}
	\Bigg\|Z^{-1}+ A_0+t\id + \sum_{i=1}^n A_i Z A_i
	\Bigg\|,
$$
where the infimum may be further restricted to $Z$ for which the matrix in
the norm on the right-hand side is a multiple of the identity. But as 
$X_{\rm free}+t\id\ge 0$, we have $\|X_{\rm free}+t\id\| =\lambda_{\rm 
max}(X_{\rm free})+t$, and analogously for the norm on the right-hand 
side. It remains to use that $\|X_{\rm free}\|=\lambda_{\rm max}(X_{\rm 
free})\vee{-\lambda_{\max}(-X_{\rm free})}$.
\end{proof}

Finally, the estimates on $\|X_{\rm free}\|$ in Lemma \ref{lem:freenck} 
were proved by Pisier \cite[p.\ 208]{Pis03} in the case $A_0=0$ (the proof 
is very similar to that of Proposition \ref{prop:samplefree} above). The 
extension to general $A_0$ follows immediately, however, using 
$\|A_0\|\le \|X_{\rm free}\|\le \|X_{\rm free}-A_0\otimes\id\|+\|A_0\|$ 
(the first inequality was explained above in the proof of Lemma 
\ref{lem:lehner}, and the second is the triangle inequality).

\subsection{Matrix parameters}
\label{sec:wxvx}

The aim of this section is to develop some basic properties of the 
parameters $\sigma(X),\sigma_*(X),v(X)$ defined in Section 
\ref{sec:mainspec}, and of the matrix alignment parameter $w(X)$ that 
was defined in Section \ref{sec:overviewpf}.

\subsubsection{The matrix alignment parameter}

We will in fact need a somewhat more general parameter than $w(X)$ in our 
proofs, so we begin by defining the relevant notion. Let 
$A_0,\ldots,A_n,A_0',\ldots,A_n'\in \M_d(\mathbb{C})_{\rm sa}$, and define
the random matrices $X=A_0+\sum_{i=1}^n g_iA_i$ and $X'=A_0'+\sum_{i=1}^n 
g_i'A_i'$ as in \eqref{eq:model}. We define 
$$
	w(X,X') := \sup_{U,V,W}\Bigg\|
	\sum_{i,j=1}^n A_iUA_j'VA_iWA_j'\Bigg\|^{\frac{1}{4}},
$$
where the supremum is taken over all unitary matrices 
$U,V,W\in\M_d(\mathbb{C})$. Note that the definition of $w(X,X')$ does not 
involve $A_0,A_0'$, that $w(X,X')=w(X',X)$ (by taking the adjoint inside 
the norm), and that $w(X,X')$ depends only on the marginal
distributions of $X$ and $X'$ (in particular, the 
Gaussian random variables $(g_i)$ and $(g_i')$ that define $X,X'$ may
have an arbitrary dependence).
Note also that we only defined $w(X,X')$ 
for self-adjoint coefficient matrices $A_i,A_i'$; the definition may be 
generalized to non-self-adjoint matrices, but this will not be needed in 
the sequel. In agreement with the notation of Section 
\ref{sec:overviewpf}, we let $w(X):=w(X,X)$.

The matrix alignment parameter $w(X)$ was introduced by Tropp in 
\cite{Tro18} to quantify the contribution of crossings to the moments of 
$X$. A key idea of \cite{Tro18} is that upper bounds in terms of $w(X)$ 
may be obtained by complex interpolation. The following variant of this 
idea suffices for our purposes.

\begin{lem}
\label{lem:complexint}
Let $Y^{(1)},\ldots,Y^{(4)}$ be arbitrary $d\times d$ complex random 
matrices, 
and let $p_1,\ldots,p_4\ge 1$ satisfy $\sum_{k=1}^4\frac{1}{p_k}=1$.
Then we have
$$
	\Bigg|
	\sum_{i,j=1}^n\EE[\ntr A_i Y^{(1)} A_j' Y^{(2)} A_i Y^{(3)} A_j'
	Y^{(4)}]
	\Bigg|\le
	w(X,X')^4 \prod_{k=1}^4 \EE[\ntr |Y^{(k)}|^{p_k}]^{\frac{1}{p_k}}.
$$
\end{lem}

\begin{proof}
We aim to show that $F(Y_1,\ldots,Y_4):= 
\sum_{i,j}\EE[\ntr A_i Y_1 A_j' Y_2 A_i 
Y_3 A_j' Y_4]$ satisfies $|F(Y_1,\ldots,Y_4)|\le w(X,X')^4\|Y_1\|_{p_1}\cdots 
\|Y_4\|_{p_4}$, where $\|Y\|_p := \EE[\ntr |Y|^p]^{\frac{1}{p}}$ denotes 
the $L_p(S_p)$-norm. Recall that the spaces $L_p(S_p)$ form a complex 
interpolation scale $L_r(S_r)=(L_p(S_p),L_q(S_q))_{\theta}$ with 
$\frac{1}{r}=\frac{1-\theta}{p}+\frac{\theta}{q}$ \cite[\S 2]{Pis03b}. 
By the classical complex interpolation theorem for multilinear maps
\cite[\S 10.1]{Cal64}, the map
$$
	\bigg(\frac{1}{p_1},\ldots,\frac{1}{p_4}\bigg)
	\mapsto \log \sup_{Y_1,\ldots,Y_4}
	\frac{|F(Y_1,\ldots,Y_4)|}{\|Y_1\|_{p_1}\cdots\|Y_4\|_{p_4}}
$$
is convex, and thus its maximum over 
$\Delta:=\{(\frac{1}{p_1},\ldots,\frac{1}{p_4})\in[0,1]^4:\sum_{k=1}^4\frac{1}{p_k}=1\}$
is attained at one of the extreme points of $\Delta$.
It therefore suffices to prove the conclusion in the case 
that $p_i=1$ for some $i$. By cyclic permutation of the trace, we may 
assume $p_4=1$ and $p_1,p_2,p_3=\infty$. But in this case
$$
	\sup_{\substack{\|Y_k\|_\infty\le 1\\ k=1,2,3}}
	\sup_{\|Y_4\|_1\le 1}
	|F(Y_1,\ldots,Y_4)| =
	\sup_{\substack{\|Y_k\|\le 1\\k=1,2,3}}
	\Bigg\|
	\sum_{i,j=1}^n A_i Y_1 A_j' Y_2 A_i Y_3 A_j'\Bigg\|
	= w(X,X')^4
$$
follows from the fact that every $Y\in\M_d(\mathbb{C})$ with $\|Y\|\le 1$ 
is a convex combination of unitaries (by singular value decomposition and 
the fact that any vector $x\in\mathbb{R}^d$ with $\|x\|_\infty\le 1$ is a 
convex combination of vectors in $\{-1,+1\}^d$).
\end{proof}

\subsubsection{Bounding the matrix alignment parameter}

The aim of this section is to prove the following bound on
the matrix alignment parameter.

\begin{prop}
\label{prop:vxwx}
We have $w(X,X')^4 \le v(X)\sigma(X)v(X')\sigma(X')$.
\end{prop}

To this end, we will require two simple observations.

\begin{lem}
\label{lem:orthook}
In the proof of Proposition \ref{prop:vxwx}, there is no loss of 
generality in assuming that
$\tr[A_iA_j]=0$ and $\tr[A_i'A_j']=0$ for all $i\ne j$.
In particular, this assumption implies $v(X)=\max_i \|A_i\|_{\rm HS}$ and 
$v(X')=\max_i \|A_i'\|_{\rm HS}$.
\end{lem}

\begin{proof}
We first note that the parameters 
$\sigma(X),v(X),w(X,X')$ only depend on the distributions of the random 
matrices $X,X'$, and not on their representations in terms of $A_i,A_i'$. 
This is evident from the expressions for $\sigma(X)$ and $v(X)$ given in
section \ref{sec:mainspec}, and as $w(X,X') = 
\sup_{U,V,W}\|\EE[XUX''VXWX'']\|^{\frac{1}{4}}$ where $X''$ is a copy of 
$X'$ that is independent of $X$.
It therefore suffices to find random matrices $Y,Y'$ that are 
equidistributed with $X,X'$ and satisfy the desired properties.

To this end, note first that $\M_d(\mathbb{C})_{\rm sa}$ is a real vector 
space of dimension $d^2$, endowed with the Hilbert-Schmidt inner product. 
Moreover, the distribution of $X$ is a real Gaussian measure on this 
space. If we denote by $C_1,\ldots,C_{d^2}\in\M_d(\mathbb{C})_{\rm sa}$ the 
(unnormalized) orthogonal eigenvectors of the corresponding covariance 
matrix, it follows that $X$ has the same distribution as $Y=A_0+\sum_i g_i 
C_i$, and $\tr[C_iC_j]=0$ for $i\ne j$ by construction. Finally, note that 
$\mathrm{Cov}(Y)=\sum_i \iota(C_i)\iota(C_i)^*$, where 
$\iota:\M_d(\mathbb{C})\to\mathbb{C}^{d^2}$ maps a matrix to its vector of 
entries. As the vectors $\iota(C_i)$ are orthogonal in $\mathbb{C}^{d^2}$, 
they are also eigenvectors of $\mathrm{Cov}(Y)$. It follows that $v(Y)^2 = 
\|\mathrm{Cov}(Y)\|=\max_i\|C_i\|_{\rm HS}^2$. The analogous construction 
applies to $X'$.
\end{proof}

\begin{lem}
\label{lem:avgbasis}
Let $B_1,\ldots,B_{d^2}\in\M_d(\mathbb{C})$ satisfy
$\tr[B_i^*B_j]=\delta_{ij}$ for all $1\le i\le j\le n$. Then we
have $\sum_{i=1}^{d^2}B_i^*YB_i = \tr[Y]\id$ for every 
$Y\in\M_d(\mathbb{C})$.
\end{lem}

\begin{proof}
Note that $\sum_{i=1}^{d^2}B_i^*YB_i=\EE H^*YH$, where
$H=\sum_{i=1}^{d^2}h_iB_i$ and $h_1,\ldots,h_{d^2}$ are i.i.d.\ 
standard complex Gaussians. Thus by unitary invariance of the complex 
Gaussian distribution, we may replace $B_1,\ldots,B_{d^2}$ by any other 
orthonormal basis of $\M_d(\mathbb{C})$. It follows that
$\sum_{i=1}^{d^2}B_i^*YB_i =
\sum_{k,l=1}^d e_ke_l^*Ye_le_k^*=\tr[Y]\id$.
\end{proof}

We now complete the proof of Proposition \ref{prop:vxwx}.

\begin{proof}[Proof of Proposition \ref{prop:vxwx}]
By Lemma \ref{lem:orthook}, we can assume that $\tr[A_iA_j]=0$ and 
$\tr[A_i'A_j']=0$ for $i\ne j$. In particular, we may choose an 
orthonormal basis $B_1,\ldots,B_{d^2}$ of $\M_d(\mathbb{C})$ so that 
$A_i=\|A_i\|_{\rm HS}B_i$ for $i=1,\ldots,n$.

Now note that we can estimate by Cauchy-Schwarz
\begin{align*}
	w(X,X')^4 &=
	\sup_{U,V,W}
	\sup_{\|x\|,\|y\|\le 1}
	\Bigg|
	\sum_{i=1}^n\Bigg\langle U^*A_ix,
	\sum_{j=1}^n A_j' V A_i W A_j'y\Bigg\rangle
	\Bigg| \\
	&\le
	\Bigg(
	\sup_{\|x\|\le 1}
	\sum_{i=1}^n \|A_ix\|^2\Bigg)^{\frac{1}{2}}
	\Bigg(
	\sup_{V,W}
	\sup_{\|y\|\le 1}
	\sum_{i=1}^n
	\Bigg\|
	\sum_{j=1}^n A_j' V A_i W A_j' y\Bigg\|^2
	\Bigg)^{\frac{1}{2}}.
\end{align*}
Furthermore,
\begin{align*}
	\sum_{i=1}^n
	\Bigg\|
	\sum_{j=1}^n A_j' V A_i W A_j' y\Bigg\|^2
	&\le
	\max_i\|A_i\|_{\rm HS}^2
	\sum_{i=1}^{d^2}
	\Bigg\|
	\sum_{j=1}^n A_j' V B_i W A_j' y\Bigg\|^2
	\\
	&=
	\max_i\|A_i\|_{\rm HS}^2
	\sum_{j,k=1}^n
	\langle y,A_j' A_k' y\rangle\tr[A_j' A_k'] 
	\\
	&\le
	\max_i\|A_i\|_{\rm HS}^2
	\max_i\|A_i'\|_{\rm HS}^2
	\sum_{j=1}^n \|A_j'y\|^2,
\end{align*}
where we used Lemma \ref{lem:avgbasis} in the equality and
$\tr[A_i'A_j']=0$ for $i\ne j$ in the second inequality. It remains to 
note that $\sup_{\|x\|\le 1}
\sum_{i=1}^n \|A_ix\|^2=\sigma(X)^2$ and 
$\max_i\|A_i\|_{\rm HS}=v(X)$ by Lemma \ref{lem:orthook}, and
analogously for $X'$.
\end{proof}

\subsubsection{Self-adjoint dilation}
\label{sec:sa}

While we defined $w(X,X')$ only for self-adjoint $X,X'$, we may extend the 
resulting inequalities to the general case by self-adjoint dilation as 
explained in Remark \ref{rem:sa}. For completeness, we presently 
provide proofs of the claims made in Remark \ref{rem:sa}.
We first prove the following.

\begin{lem}
\label{lem:sadil}
Let $T$ be a bounded operator on a Hilbert space $H$, and 
denote by $\breve T$ the self-adjoint operator on
$H\oplus H$ defined by
$$
	\breve T = 
	\begin{bmatrix} 0 & T \\ T^* & 0
	\end{bmatrix}.
$$
Then $\spc(\breve T)\cup\{0\} = \spc(|T|) \cup {-\spc(|T|)}\cup\{0\}$.
\end{lem}

\begin{proof}
Let $T=V|T|$ be the polar decomposition of $T$, where $V$ is a
partial isometry with initial space $(\ker T)^\perp$ and final space
$\mathrm{cl}(\mathop{\mathrm{ran}}T)=(\ker T^*)^\perp$. As $TT^* = 
V|T|^2V^* = VT^*TV^*$, it follows that
$\spc(T^*T)\cup\{0\}=\spc(TT^*)\cup\{0\}$. Thus
\begin{equation}
\label{eq:dilsq}
	\breve T^2 = \begin{bmatrix} TT^* & 0 \\ 0 & T^*T
	\end{bmatrix}
\end{equation}
implies that $\spc(|\breve T|)\cup\{0\} = \spc(|T|)\cup \{0\}$.
On the other hand, as
$$
	U^*\breve T U = -\breve T,\qquad\quad
	U=
	\begin{bmatrix} \id & 0 \\ 0 & -\id 
        \end{bmatrix}
$$
and $U$ is unitary, we have $\spc(\breve T)=-\spc(\breve T)$. The 
conclusion follows.
\end{proof}

We now verify that $\sigma(X),\sigma_*(X),v(X)$ are 
well behaved under dilation.

\begin{lem}
In the setting of Remark \ref{rem:sa}, we have
$$
        \sigma(\breve X)=\sigma(X),\qquad
        \sigma_*(\breve X)=\sigma_*(X),\qquad
        v(X)\le v(\breve X)\le \sqrt{2}\,v(X).   
$$
\end{lem}

\begin{proof}
We begin by noting that by \eqref{eq:dilsq}
$$
	\sigma(\breve X)^2 =
	\|\EE \breve X^2\|=
	\left\|
	\begin{bmatrix} \EE XX^* & 0 \\ 0 & \EE X^*X
	\end{bmatrix}\right\|=\sigma(X)^2.
$$
Next, note that
$$
	\sigma_*(\breve X)^2 =
	\sup_{\|v_1\|^2+\|v_2\|^2=1} \sup_{\|w_1\|^2+\|w_2\|^2=1}
	\EE[|\langle v_1,Xw_2\rangle + \langle v_2,X^*w_1\rangle|^2].
$$
Thus clearly $\sigma_*(\breve X)\ge\sigma_*(X)$,
while by the triangle inequality
$$
	\sigma_*(\breve X) \le
	\sigma_*(X)
	\sup_{\|v_1\|^2+\|v_2\|^2=1} \sup_{\|w_1\|^2+\|w_2\|^2=1}
	(\|v_1\|\|w_2\|+\|v_2\|\|w_1\|) =
	\sigma_*(X).
$$
Finally, note that
$$
	v(\breve X)^2 =
	\sup_{\|M\|_{\rm HS}^2+\|N\|_{\rm HS}^2=1}
	\EE[|{\tr}[XM]+\tr[X^*N]|^2],
$$
so that $v(X)\le v(\breve X)\le\sqrt{2}\,v(X)$ follows in the same manner 
as for $\sigma_*(X)$.
\end{proof}

\subsection{Gaussian analysis}
\label{sec:conc}

We now recall some Gaussian tools that will be used in the sequel. 
The following is classical \cite[Lemma 1.3.1]{Tal11}.

\begin{lem}[Gaussian interpolation]
\label{lem:ginterp}
Let $Y$ and $Z$ be independent centered Gaussian vectors in $\mathbb{R}^n$ 
with covariance matrices $\Sigma^Y$ and $\Sigma^Z$, respectively. Let
$$
	Y_t = \sqrt{t}\,Y + \sqrt{1-t}\,Z
$$
for $t\in[0,1]$. Then we have 
$$
	\frac{d}{dt}\EE[f(Y_t)] =
	\frac{1}{2}\sum_{i,j=1}^n (\Sigma_{ij}^Y-\Sigma_{ij}^Z) \,
	\EE\bigg[\frac{\partial^2 f}{\partial x_i\partial x_j}
	(Y_t)\bigg]
$$
for any smooth $f:\mathbb{R}^n\to\mathbb{C}$ with derivatives of 
polynomial growth.
\end{lem}

A special case is the following (see, e.g., \cite[\S 5.5]{Led01}).

\begin{cor}[Gaussian covariance identity]
\label{cor:gcov}
Let $Y,Z$ be independent centered Gaussian vectors in $\mathbb{R}^n$
with covariance matrix $\Sigma$, and let
$$
	Y_t'=t\,Y+\sqrt{1-t^2}\,Z
$$
for $t\in[0,1]$. Then we have
$$
	\EE[f(Y)g(Y)]-\EE[f(Y)]\,\EE[g(Y)] =
	\int_0^1
	\EE[\langle \nabla f(Y),\Sigma\,\nabla g(Y_t')\rangle]
	\,dt
$$
for any smooth $f,g:\mathbb{R}^n\to\mathbb{C}$ with derivatives of 
polynomial growth.
\end{cor}

\begin{proof}
Let $Y,Z,Z'$ be independent centered Gaussian vectors with covariance 
matrix $\Sigma$, and let
$G=(Y,Y)$, $G'=(Z,Z')$, and $G_t = \sqrt{t}\,G+\sqrt{1-t}\,G'$.
Then
$$
	\EE[f(Y)g(Y)]-\EE[f(Y)]\,\EE[g(Y)] =
	\int_0^1 \frac{d}{dt}\EE[H(G_t)]\,dt,
$$
where $H(x,y)=f(x)g(y)$.
The conclusion follows from Lemma \ref{lem:ginterp} and the fact that
$(\sqrt{t}\,Y+\sqrt{1-t}\,Z,\sqrt{t}\,Y+\sqrt{1-t}\,Z')$ is 
equidistributed with $(Y,Y_t')$.
\end{proof}

We finally recall the following \cite[p.\ 41]{Led01}.

\begin{lem}[Gaussian concentration]
\label{lem:gconc}
Let $Y$ be a standard Gaussian vector in $\mathbb{R}^n$, and let
$f:\mathbb{R}^n\to\mathbb{R}$ be an $L$-Lipschitz function. Then
$$
	\mathbf{P}[f(Y)\ge \EE f(Y)+t] \le e^{-t^2/2L^2}\quad\mbox{for all 
}t\ge 0.
$$
\end{lem}

It is instructive to spell out the application of Gaussian concentration 
to \eqref{eq:model}, which explains the significance of the parameter 
$\sigma_*(X)$.

\begin{cor}
\label{cor:gconcmtx}
Consider the model \eqref{eq:model} with 
$A_0,\ldots,A_n\in\M_d(\mathbb{C})$, and let
$F:\M_d(\mathbb{C})\to\mathbb{R}$ be $L$-Lipschitz with respect to the
operator norm. Then
$$
	\mathbf{P}[F(X)\ge\EE F(X)+t] \le e^{-t^2/2L^2\sigma_*(X)^2}
	\quad\mbox{for all }t\ge 0.
$$
If $A_0,\ldots,A_n\in\M_d(\mathbb{C})_{\rm sa}$, it suffices
to assume $F$ is $L$-Lipschitz on $\M_d(\mathbb{C})_{\rm sa}$.
\end{cor}

\begin{proof}
We may write $F(X)=f(g_1,\ldots,g_n):=F(A_0+\sum_i g_i A_i)$. Thus
\begin{align*}
	|f(x)-f(y)| &\le
	L\bigg\|\sum_i (x_i-y_i)A_i\bigg\| =
	L\sup_{\|v\|=\|w\|=1}\bigg|
	\sum_i (x_i-y_i)\langle v,A_i w\rangle\bigg|
	\\ &\le
	L\sigma_*(X)\|x-y\| 	
\end{align*}
by Cauchy-Schwarz and the definition of $\sigma_*(X)$ (cf.\ Section 
\ref{sec:mainspec}). The conclusion follows by applying Lemma 
\ref{lem:gconc} to $f(g_1,\ldots,g_n)$.
\end{proof}

\section{Spectral statistics}
\label{sec:spec}

The next three sections contain the proofs of the main results of this 
paper. In the present section, we begin by proving our bounds on the 
spectral statistics that were formulated in Section \ref{sec:mainstat}. 
These results illustrate the main proof technique of this paper in its 
simplest form. The support of the spectrum will be investigated in the 
next section using a more involved variant of the same method.

\subsection{The basic construction}
\label{sec:construction}

Throughout the proofs of our main results in Sections \ref{sec:mainspec} 
and \ref{sec:mainstat}, we will fix 
$A_0,\ldots,A_n\in\M_d(\mathbb{C})_{\rm sa}$ and let $X$ and $X_{\rm 
free}$ be defined as in \eqref{eq:model}. (Where relevant, the extension 
to the non-self-adjoint case will be done at the end of the proof using 
Remark \ref{rem:sa}.)

Let $G_1^N,\ldots,G_n^N$ be independent standard Wigner matrices as in 
Definition \ref{defn:wigner}, and let $D_1^N,\ldots,D_n^N$ be independent 
$N\times N$ diagonal matrices with i.i.d.\ standard Gaussians on the 
diagonal. We define for $q\in[0,1]$ the random matrix
\begin{equation} 
\label{eq:xq}
	X^N_q := A_0\otimes\id+
	\sum_{i=1}^n A_i\otimes (\sqrt{q}\,D_i^N +
	\sqrt{1-q}\,G_i^N).
\end{equation}
Note that $X^N_0=X^N$ as defined in \eqref{eq:nmodel}. On the other hand,
$X^N_1$ is a block-diagonal matrix with i.i.d.\ copies of $X$ on the 
diagonal. In particular, we have
\begin{equation} 
\label{eq:thepoint}
\begin{aligned}
	\EE[\ntr h(X^N_1)] &=
	\EE[\ntr h(X)],\\
	\EE[\ntr h(X^N_0)] &= \EE[\ntr h(X^N)]
\end{aligned}
\end{equation}
for any function $h:\mathbb{R}\to\mathbb{C}$. The basic idea behind our 
proofs is to interpolate 
between $\EE[\ntr h(X^N_1)]$ and $\EE[\ntr h(X^N_0)]$ using Lemma 
\ref{lem:ginterp}.

To simplify the expressions that will arise in the analysis, it will be 
convenient to define for $y=(y_{irs})_{1\le i\le n,1\le s\le r\le N}$ the 
notation
$$
	X^N(y) := 
	A_0\otimes\id + \sum_{i=1}^n \sum_{1\le s\le r\le N}
	y_{irs}\,A_{irs},\qquad\quad
	A_{irs} := A_i\otimes E_{rs},
$$
where $E_{rs}$ are as defined in Section \ref{sec:exind}. Moreover, let 
$Y,Z$ be centered Gaussian vectors all of 
whose entries $Y_{irs}=(D_i^N)_{rs}$ and $Z_{irs}=(G_i^N)_{rs}$ are 
independent with variances
$\delta_{rs}$ and $\frac{1}{N}$, respectively.
Then $X_q^N = X^N(\sqrt{q}\,Y+\sqrt{1-q}\,Z)$.

\subsection{Proof of Theorem \ref{thm:moments}}

In order to prove Theorem \ref{thm:moments}, we apply the above program
to the moments. We begin with a simple computation.

\begin{lem}
\label{lem:momentraw}
For any $p\in\mathbb{N}$, we have
$$
	\frac{d}{dq}\EE[\ntr (X_q^N)^{2p}] =
	p\sum_{k=0}^{2p-2}\sum_i \sum_{r\ge s}
	\bigg(\delta_{rs}-\frac{1}{N}\bigg)
	\EE[\ntr A_{irs} (X_q^N)^k A_{irs}
	(X_q^N)^{2p-2-k}].
$$
\end{lem}

\begin{proof}
Let $Y=(Y_{irs})_{i\in[n],r\ge s}$ and
$Z=(Z_{irs})_{i\in[n],r\ge s}$ be the Gaussian vectors defined above.
As both these vectors have independent entries, their covariance matrices 
$\Sigma^Y$ and $\Sigma^Z$ are diagonal 
with $\mathrm{Var}(Y_{irs})=\delta_{rs}$ and 
$\mathrm{Var}(Z_{irs})=\frac{1}{N}$. Applying Lemma 
\ref{lem:ginterp} to the function
$f(y)=\ntr X^N(y)^{2p}$ therefore yields
$$
	\frac{d}{dq}\EE[\ntr (X_q^N)^{2p}] =
	\frac{1}{2}\sum_i \sum_{r\ge s}
	\bigg(\delta_{rs}-\frac{1}{N}\bigg)
	\EE\bigg[\frac{\partial^2 f}{\partial y_{irs}^2}
	(\sqrt{q}\,Y+\sqrt{1-q}\,Z)\bigg].
$$
The conclusion follows by a straightforward computation.
\end{proof}

As was explained in Section \ref{sec:overviewpf}, we expect that the 
interpolation between $X$ and $X_{\rm free}$ will be controlled only by 
the crossings in the moment formulae. This is however not immediately 
obvious from the expression in Lemma \ref{lem:momentraw}. To make this 
phenomenon visible, we need a simple lemma.

\begin{lem}
\label{lem:ptrace}
$\EE[h(X_q^N)] = \EE[(\mathrm{id}\otimes\ntr)(h(X_q^N))]\otimes\id$
for every $h:\mathbb{R}\to\mathbb{C}$.
\end{lem}

\begin{proof}
The distributions of $D_i^N$ and $G_i^N$ are invariant under conjugation 
by any signed permutation matrix. Therefore, if we let $\Pi$ be an 
$N\times N$ signed permutation matrix chosen uniformly at random 
(independently of $X_q^N$), then
$$
	\EE[h(X_q^N)]  = \EE[h((\id\otimes\Pi)^*X_q^N(\id\otimes\Pi))]
	= \EE[(\id\otimes\Pi)^*h(X_q^N)(\id\otimes\Pi)].
$$
It remains to note that $\EE[(\id\otimes\Pi)^*M(\id\otimes\Pi)]=
(\mathrm{id}\otimes\ntr)(M)\otimes\id$ for any matrix $M$ (this is 
elementary when $M=A\otimes B$, and extends to general $M$ by linearity).
\end{proof}

The key observation is now the following.

\begin{cor}
\label{cor:momentrawmf}
For any $p\in\mathbb{N}$, we have
$$
	p\sum_{k=0}^{2p-2}\sum_i \sum_{r\ge s}
	\bigg(\delta_{rs}-\frac{1}{N}\bigg)
	\ntr A_{irs} \EE[(X_q^N)^k] A_{irs}
	\EE[(X_q^N)^{2p-2-k}] = 0.
$$
\end{cor}

\begin{proof}
Note first that $E_{rs}^2=E_{rr}^2+E_{ss}^2$ for $r\ne s$. Thus
\begin{multline*}
	(A_i\otimes E_{rs}) \EE[(X_q^N)^k] (A_i\otimes E_{rs}) =\\
	(A_i\otimes E_{rr}) \EE[(X_q^N)^k] (A_i\otimes E_{rr}) +
	(A_i\otimes E_{ss}) \EE[(X_q^N)^k] (A_i\otimes E_{ss}) 
\end{multline*}
for $r\ne s$ by Lemma \ref{lem:ptrace}. Summing over $r>s$ yields
\begin{align*}
	\frac{1}{N}
	\sum_{r>s} A_{irs} \EE[(X_q^N)^k] A_{irs} &=
	\frac{1}{N}
	\sum_{r>s} 
	(A_{irr} \EE[(X_q^N)^k] A_{irr} +
	A_{iss} \EE[(X_q^N)^k] A_{iss})
	\\ &=
	\bigg(1-\frac{1}{N}\bigg) \sum_{r} 
	A_{irr} \EE[(X_q^N)^k] A_{irr}.
\end{align*}
The latter identity may be equivalently written as
$$
	\sum_{r\ge s}
	\bigg(\delta_{rs}-\frac{1}{N}\bigg) A_{irs} \EE[(X_q^N)^k] A_{irs} 
	=0,
$$
from which the conclusion follows readily.
\end{proof}

By combining Lemma \ref{lem:momentraw} and Corollary 
\ref{cor:momentrawmf}, we can apply Corollary \ref{cor:gcov} to make 
crossings appear (the latter idea is already present in \cite{Tro18}).
Recall that the parameters $w(X)$ and $w(X,X')$ were defined in
Section \ref{sec:wxvx}.

\begin{lem}
\label{lem:momentdiffeq}
For any $p\in\mathbb{N}$, we have
$$
	\bigg|\frac{d}{dq}\EE[\ntr (X_q^N)^{2p}]\bigg|
	\le \frac{4}{3}p^4\{qw(X_1^N)^4 + w(X_0^N,X_1^N)^4 + 
	(1-q)w(X_0^N)^4\}
	\EE[\ntr (X_q^N)^{2p-4}].
$$
\end{lem}

\begin{proof}
Recall that the random vectors $Y,Z$ with
$Y_{irs}=(D_i^N)_{rs}$ and $Z_{irs}=(G_i^N)_{rs}$ were defined in
section \ref{sec:construction}. Let $Y',Z'$ be independent copies of 
$Y,Z$, and define
$$
	X_{qt}^N=X^N\big(t\big\{\sqrt{q}\,Y+\sqrt{1-q}\,Z\big\} +
	\sqrt{1-t^2}\big\{\sqrt{q}\,Y'+\sqrt{1-q}\,Z'\big\}\big).
$$
Note that the random vector $\sqrt{q}\,Y+\sqrt{1-q}\,Z$ has independent 
entries, so its covariance matrix $\Sigma$ is diagonal with 
$\mathrm{Var}(\sqrt{q}\,Y_{irs}+\sqrt{1-q}\,Z_{irs})=q\delta_{rs}+
\frac{1-q}{N}$. We can therefore apply Corollary \ref{cor:gcov} to compute
for $1\le k\le 2p-3$
\begin{align*}
	&
	\EE[(X_q^N)^k_{ab}\,
	(X_q^N)^{2p-2-k}_{cd}] -
	\EE[(X_q^N)^k_{ab}]\,
	\EE[(X_q^N)^{2p-2-k}_{cd}] =\mbox{} \\ &
\quad
	\sum_{l=0}^{k-1}
	\sum_{m=0}^{2p-3-k}
	\sum_i
	\sum_{r\ge s} \bigg(q \delta_{rs} + \frac{1-q}{N}\bigg)
	\cdot\mbox{}\\
	&\qquad
\quad
	\int_0^1 
	\EE\big[\big((X_q^N)^l
	A_{irs}
	(X_q^N)^{k-1-l}\big)_{ab}\,
	\big((X_{qt}^N)^{m}
	A_{irs}
	(X_{qt}^N)^{2p-3-k-m}\big)_{cd}\big]\,
	dt.
\end{align*}
Combining this identity with Lemma \ref{lem:momentraw} and Corollary 
\ref{cor:momentrawmf} yields
\begin{align*}
	&\frac{d}{dq}\EE[\ntr (X_q^N)^{2p}] \\
	& = p\sum_{k=0}^{2p-2}\sum_{i'} \sum_{r'\ge s'}
	\bigg(\delta_{r's'}-\frac{1}{N}\bigg)
	\EE[\ntr A_{i'r's'} (X_q^N)^k A_{i'r's'}
	(X_q^N)^{2p-2-k}]
\\
	&\qquad
	-
	p\sum_{k=0}^{2p-2}\sum_{i'} \sum_{r'\ge s'}
	\bigg(\delta_{r's'}-\frac{1}{N}\bigg)
	\ntr A_{i'r's'} \EE[(X_q^N)^k] A_{i'r's'}
	\EE[(X_q^N)^{2p-2-k}] 
\\
	&=
	p\sum_{k=1}^{2p-3}
	\sum_{l=0}^{k-1}
	\sum_{m=0}^{2p-3-k}
	\int_0^1 
	\sum_{i,i'}
	\sum_{r\ge s}
	\sum_{r'\ge s'}
	\bigg( q\delta_{rs}\delta_{r's'} +
	\frac{1-q}{N}\delta_{r's'} -
	\frac{q}{N}\delta_{rs} -
	\frac{1-q}{N^2}\bigg)
\cdot\mbox{}  \\
	&\qquad\qquad
	\EE[\ntr A_{i'r's'}(X_q^N)^l
	A_{irs}
	(X_q^N)^{k-1-l}
	A_{i'r's'}
	(X_{qt}^N)^{m}
	A_{irs}
	(X_{qt}^N)^{2p-3-k-m}]\,
	dt,
\end{align*}
where we used that the $k=0$ and $k=2p-2$ terms in the middle expression
cancel.
We can now apply Lemma \ref{lem:complexint} with
$$
	p_1=\frac{2p-4}{l},\qquad
	p_2=\frac{2p-4}{k-1-l},\qquad
	p_3=\frac{2p-4}{m},\qquad
	p_4=\frac{2p-4}{2p-3-k-m},
$$
to bound, for example,
\begin{multline*}
	\bigg|
	\sum_{i,i'}
	\sum_{r\ge s}
	\sum_{r'\ge s'}
	\frac{q}{N}\delta_{rs}\,
	\EE[\ntr 
	A_{i'r's'}(X_q^N)^l
        A_{irs}
        (X_q^N)^{k-1-l} 
        A_{i'r's'}
        (X_{qt}^N)^{m} \cdot \mbox{}\\
        A_{irs}
        (X_{qt}^N)^{2p-3-k-m}]\bigg| 
	\le 
	qw(X_0^N,X_1^N)^4\EE[\ntr (X_q^N)^{2p-4}],
\end{multline*}
where we used that $\mathrm{Var}(Y_{irs})=\delta_{rs}$,
$\mathrm{Var}(Z_{irs})=\frac{1}{N}$, and that
$X_q^N$ and $X_{qt}^N$ are equidistributed.
The remaining three terms in the integral can be bounded analogously.
To conclude, it remains to note that $\sum_{k=1}^{2p-3} 
k(2p-2-k) = \binom{2p-1}{3} \le \frac{4}{3}p^3$.
\end{proof}

Before we can complete the proof of Theorem \ref{thm:moments}, we must 
compute the matrix parameters associated to $X_q^N$.

\begin{lem}
\label{lem:momentparm}
For every $q,N$, we have
$$
	\sigma(X_q^N)=\sigma(X),\qquad\quad
	v(X_1^N)=v(X),\qquad\quad
	v(X_0^N)=v(X)\sqrt{\frac{2}{N}}.
$$
\end{lem}

\begin{proof}
As $\EE[(D_i^N)^2] = \EE[(G_i^N)^2]=\id$, we have
$\EE[(X_q^N-\EE X_q^N)^2]=\sum_i A_i^2\otimes\id$ and thus
$\sigma(X_q^N)^2=\|\EE[(X_q^N-\EE X_q^N)^2]\|=\sigma(X)^2$.

Next, note that $X_1^N$ is a block-diagonal matrix with i.i.d.\ copies of 
$X$ on the diagonal. Therefore
$v(X_1^N)^2=\|\mathrm{Cov}(X_1^N)\|=\|\mathrm{Cov}(X)\|=v(X)^2$. On the 
other hand, $X_0^N-\EE[X_0^N]$ is a symmetric block matrix whose blocks on 
and above the diagonal are i.i.d.\ copies of the matrix
$N^{-\frac{1}{2}}(X-\EE[X])$.
We can therefore compute
$v(X_0^N)^2=\|\mathrm{Cov}(X_0^N)\|=2N^{-1}\|\mathrm{Cov}(X)\|=2N^{-1}v(X)^2$.
\end{proof}

We can now conclude the proof.

\begin{proof}[Proof of Theorem \ref{thm:moments}]
Assume first that $A_0,\ldots,A_n\in\M_d(\mathbb{C})_{\rm sa}$ are 
self-adjoint. Applying Lemma \ref{lem:momentdiffeq}, the chain 
rule, and Proposition \ref{prop:vxwx} yields
\begin{align*}
	\bigg|\frac{d}{dq}\EE[\ntr (X_q^N)^{2p}]^{\frac{2}{p}}\bigg|
	&= \frac{2}{p}\EE[\ntr (X_q^N)^{2p}]^{\frac{2}{p}-1}
	\bigg|\frac{d}{dq}\EE[\ntr (X_q^N)^{2p}]\bigg|
\\
	&\le 
	\frac{8}{3}p^3\{qw(X_1^N)^4 + w(X_0^N,X_1^N)^4 + 
	(1-q)w(X_0^N)^4\} \\
	&\le 
	\frac{8}{3}p^3\{q \tilde v(X_1^N)^4 + \tilde v(X_0^N)^2
	\tilde v(X_1^N)^2 + 
	(1-q)\tilde v(X_0^N)^4\},
\end{align*}
where we used that $\EE[\ntr (X_q^N)^{2p-4}]\le
\EE[\ntr (X_q^N)^{2p}]^{1-\frac{2}{p}}$ by H\"older's inequality.
Thus
\begin{multline*}
	|\EE[\ntr X^{2p}]^{\frac{1}{2p}}-
	\EE[\ntr (X^N)^{2p}]^{\frac{1}{2p}}|
	\le
	|\EE[\ntr X^{2p}]^{\frac{2}{p}}-
	\EE[\ntr (X^N)^{2p}]^{\frac{2}{p}}|^{\frac{1}{4}}
\\
	=
	\bigg| \int_0^1 \frac{d}{dq}\EE[\ntr (X_q^N)^{2p}]^{\frac{2}{p}}
	\,dq\bigg|^{\frac{1}{4}} 
\le
        \left(\frac{4}{3}\right)^{\frac{1}{4}}
        p^{\frac{3}{4}}\{\tilde v(X_1^N)^2 +
        \tilde v(X_0^N)^2\}^{\frac{1}{2}},
\end{multline*}
where we used $x-y = (x^4-y^4+y^4)^{\frac{1}{4}}-y \le
(x^4-y^4)^{\frac{1}{4}}$ for $x\ge y\ge 0$ and
\eqref{eq:thepoint}. But note that Lemma 
\ref{lem:momentparm} implies $\tilde v(X_1^N)=\tilde v(X)$ and
$\tilde v(X_0^N)=2^{\frac{1}{4}}N^{-\frac{1}{4}}\tilde v(X)$. We may 
therefore let $N\to\infty$ in the above inequality and use Corollary 
\ref{cor:weakfree} to obtain
$$
	|\EE[\ntr X^{2p}]^{\frac{1}{2p}}-
	({\ntr}\otimes\tau)(X_{\rm free}^{2p})^{\frac{1}{2p}}|
	\le 
	\left(\frac{4}{3}\right)^{\frac{1}{4}}
	p^{\frac{3}{4}}\tilde v(X).
$$
Finally, we extend the conclusion to non-self-adjoint
$A_0,\ldots,A_n\in\M_d(\mathbb{C})$ by applying the above inequality to 
the self-adjoint model $\breve X$ defined in Remark \ref{rem:sa}.
As $\EE[\ntr \breve X^{2p}]=\EE[\ntr |X|^{2p}]$ and
$({\ntr}\otimes\tau)(\breve X_{\rm free}^{2p}) =
({\ntr}\otimes\tau)(|X_{\rm free}|^{2p})$ by \eqref{eq:dilsq}, and as 
$\tilde v(\breve X) \le 2^{\frac{1}{4}}\tilde v(X)$, the conclusion 
follows readily (using $(\frac{8}{3})^{\frac{1}{4}}\le 2$).
\end{proof}

\begin{rem}
\label{rem:tildew}
When $A_0,\ldots,A_n\in\M_d(\mathbb{C})_{\rm sa}$ are self-adjoint, we may 
obtain a slightly better bound in the proof of Theorem \ref{thm:moments}
by neglecting to apply Proposition \ref{prop:vxwx} to $w(X_1^N)$. In this 
case, the parameter $\tilde v(X)$ in the final bound is replaced by
$\sup_N w(X_1^N)$. The analogous improvement is possible for most results 
of this paper. However, as $\sup_N w(X_1^N)$ is very difficult to compute 
in any concrete situation, we have formulated our main results in terms of 
the computable quantity $\tilde v(X)$.
\end{rem}

\subsection{Proof of Theorem \ref{thm:stieltjes}}

Once the basic method of proof has been understood, it may be readily 
adapted to control spectral statistics other than the moments. We 
presently adapt the method of the previous section to the matrix-valued 
Stieltjes transform. Note that Theorem \ref{thm:stieltjes} assumes 
$A_0,\ldots,A_n\in\M_d(\mathbb{C})_{\rm sa}$.

\begin{lem}
\label{lem:stieltjesraw}
For any $Z\in\M_d(\mathbb{C})$, $\mathrm{Im}\,Z>0$ and 
$M\in\M_d(\mathbb{C})\otimes\M_N(\mathbb{C})$, we have
\begin{multline*}
	\frac{d}{dq}\EE[\ntr M(\tilde Z - X_q^N)^{-1}] = \\
	\sum_i \sum_{r\ge s} 
	\bigg(\delta_{rs}-\frac{1}{N}\bigg)
	\EE[\ntr 
	A_{irs}
	(\tilde Z- X_q^N)^{-1}A_{irs}
	(\tilde Z - X_q^N)^{-1}M
	(\tilde Z - X_q^N)^{-1}]
\end{multline*}
and
$$
	\sum_i \sum_{r\ge s} 
	\bigg(\delta_{rs}-\frac{1}{N}\bigg)
	\ntr 
	A_{irs}
	\EE[(\tilde Z- X_q^N)^{-1}]A_{irs}
	\EE[(\tilde Z - X_q^N)^{-1}M
	(\tilde Z - X_q^N)^{-1}]=0,
$$
where we defined $\tilde Z=Z\otimes\id\in 
\M_d(\mathbb{C})\otimes\M_N(\mathbb{C})$.
\end{lem}

\begin{proof}
The first identity follows from
Lemma \ref{lem:ginterp} with
$f(y)=\ntr M(\tilde Z - X^N(y))^{-1}$.
The second identity follows as
$\EE[(\tilde Z - X_q^N)^{-1}]=
\EE[(\mathrm{id}\otimes\ntr)(\tilde Z - X_q^N)^{-1}]\otimes\id$
holds by precisely the same proof as that of Lemma \ref{lem:ptrace}.
\end{proof}

We can now proceed as in Lemma \ref{lem:momentdiffeq}.

\begin{lem}
\label{lem:stieltjesdiffeq}
For any $Z\in\M_d(\mathbb{C})$, $\mathrm{Im}\,Z>0$ we have
$$
	\bigg\|
	\frac{d}{dq}\EE[(Z\otimes\id - X_q^N)^{-1}]\bigg\|
	\le
	2\|(\mathrm{Im}\,Z)^{-5}\|
	\{qw(X_1^N)^4 + w(X_0^N,X_1^N)^4 +
        (1-q)w(X_0^N)^4\}.
$$
\end{lem}

\begin{proof}
Define $X_{qt}^N$ as in the proof of Lemma \ref{lem:momentdiffeq},
and denote $R:=(Z\otimes\id - X_q^N)^{-1}$ and
$R_t:=(Z\otimes\id - X_{qt}^N)^{-1}$ for simplicity.
Corollary \ref{cor:gcov} and
Lemma \ref{lem:stieltjesraw} yield
\begin{align*}
	&\frac{d}{dq}\EE[\ntr M(Z\otimes\id - X_q^N)^{-1}] =
	\mbox{}\\
	&
	\int_0^1
	\sum_{i,i'}  \sum_{r\ge s} \sum_{r'\ge s'}
	\bigg(q\delta_{rs}\delta_{r's'} +
	\frac{1-q}{N}\delta_{r's'} -
	\frac{q}{N}\delta_{rs} - \frac{1-q}{N^2}\bigg)\cdot\mbox{}
	\\
	&\qquad
	\big\{	
	\EE[\ntr 
	A_{i'r's'}
	R
	A_{irs}
	R
	A_{i'r's'}
	R_t
	A_{irs}
	R_t
	M
	R_t]
	\\
	&\qquad\quad+
	\EE[\ntr 
	A_{i'r's'}
	R
	A_{irs}
	R
	A_{i'r's'}
	R_t
	M
	R_t
	A_{irs}
	R_t]
	\big\}\,dt.
\end{align*}
Now apply Lemma \ref{lem:complexint} with $p_1=p_2=p_3=\infty$ and
$p_4=1$ to the first expectation in the integral, and with
$p_1=p_2=p_4=\infty$ and $p_3=1$ to the second expectation. This yields, 
in the same manner as in the proof of Lemma \ref{lem:momentdiffeq}, that
\begin{align*}
	&\bigg|\frac{d}{dq}\EE[\ntr M(Z\otimes\id - X_q^N)^{-1}]\bigg|
	\\ &\qquad\le
	2\{qw(X_1^N)^4 + w(X_0^N,X_1^N)^4 +
        (1-q)w(X_0^N)^4\}\|\|R\|\|_\infty^3\EE[\ntr |RMR|].
\end{align*}
But as $\|R\|\le \|(\mathrm{Im}\,Z)^{-1}\|$ (see, e.g.,
\cite[Lemma 3.1]{HT05}), we obtain
\begin{multline*}
	\bigg|\ntr M
	\frac{d}{dq}\EE[(Z\otimes\id - X_q^N)^{-1}]\bigg|
	\\ \le
	2\|(\mathrm{Im}\,Z)^{-5}\|\{qw(X_1^N)^4 + w(X_0^N,X_1^N)^4 +
        (1-q)w(X_0^N)^4\}
	\ntr |M|.
\end{multline*}
The conclusion follows by taking the supremum over all $M$ with
$\ntr |M|\le 1$.
\end{proof}

Integrating the above differential inequality yields the following.

\begin{lem}
\label{lem:stieltjesintg}
For any $Z\in\M_d(\mathbb{C})$, $\mathrm{Im}\,Z>0$ we have
$$
	\|
	\EE[(Z - X)^{-1}]-
	\EE[
	(\mathrm{id}\otimes\ntr)
	(Z\otimes\id - X^N)^{-1}]
	\|
	\le
	(1+N^{-\frac{1}{2}})^2
	\tilde v(X)^4\|(\mathrm{Im}\,Z)^{-5}\|.
$$
\end{lem}

\begin{proof}
Integrating Lemma \ref{lem:stieltjesdiffeq} and using
Proposition \ref{prop:vxwx} yields
$$
	\|
	\EE[(Z\otimes\id - X_1^N)^{-1}]-
	\EE[(Z\otimes\id - X_0^N)^{-1}]
	\|
	\le
	\{\tilde v(X_1^N)^2+\tilde v(X_0^N)^2\}^2
	\|(\mathrm{Im}\,Z)^{-5}\|.
$$
As $X_0^N=X^N$, we have
$\EE[(Z\otimes\id - X_0^N)^{-1}]=\EE[(\mathrm{id}\otimes\ntr)(Z\otimes\id
- X^N)^{-1}]\otimes\id$ as in Lemma \ref{lem:ptrace}. Similarly, as
$X_1^N$ is block-diagonal with i.i.d.\ copies of $X$ on the diagonal, we 
have $\EE[(Z\otimes\id - X_1^N)^{-1}]=
\EE[(Z - X)^{-1}]\otimes\id$ as in Lemma \ref{lem:ptrace}. The conclusion 
follows readily from these observations and
Lemma \ref{lem:momentparm}.
\end{proof}

It remains to take the limit $N\to\infty$ in Lemma 
\ref{lem:stieltjesintg}. While Corollary \ref{cor:weakfree} does not apply 
directly here, its proof may be readily extended to the present setting.

\begin{lem}
\label{lem:weakfreept}
For any $Z\in\M_d(\mathbb{C})$, $\mathrm{Im}\,Z>0$ we have
$$
	\lim_{N\to\infty}
	\|
	\EE[
	(\mathrm{id}\otimes\ntr)
	(Z\otimes\id - X^N)^{-1}]
	-
	(\mathrm{id}\otimes\tau)
	(Z\otimes\id - X_{\rm free})^{-1}
	\| = 0.
$$
\end{lem}

\begin{proof}
As we aim to establish convergence as $N\to\infty$ in $\M_d(\mathbb{C})$ 
with a fixed finite dimension $d$, it suffices to show that
$$
	\lim_{N\to\infty}
	\langle v,
	\{
	\EE[
	(\mathrm{id}\otimes\ntr)
	(Z\otimes\id - X^N)^{-1}]
	-
	(\mathrm{id}\otimes\tau)
	(Z\otimes\id - X_{\rm free})^{-1}
	\}\,v\rangle = 0
$$
for all $v\in\mathbb{C}^d$ with $\|v\|=1$. Moreover, if we define
\begin{align*}
	\tilde X^N &:=
	(\mathrm{Im}\,Z\otimes\id)^{-1/2}\{
	X^N - \mathrm{Re}\,Z\otimes\id 
	\}(\mathrm{Im}\,Z\otimes\id)^{-1/2},
	\\
	\tilde X_{\rm free} &:=
	(\mathrm{Im}\,Z\otimes\id)^{-1/2}\{
	X_{\rm free} -
	\mathrm{Re}\,Z\otimes\id\}
	(\mathrm{Im}\,Z\otimes\id)^{-1/2}
\end{align*}
where $\mathrm{Re}\,Z:=\frac{1}{2}(Z+Z^*)$,
it clearly suffices to show that
$$
	\lim_{N\to\infty}
	\langle v,
	\{
	\EE[
	(\mathrm{id}\otimes\ntr)
	(i\id - \tilde X^N)^{-1}]
	-
	(\mathrm{id}\otimes\tau)
	(i\id - \tilde X_{\rm free})^{-1}
	\}\,v\rangle = 0
$$
for all $v\in\mathbb{C}^d$ with $\|v\|=1$.
By the spectral theorem, there are probability measures
$\mu_N,\mu$ (which depend on the choice of $v$) so that
$$
	\int h\,d\mu_N =
        \langle v,\EE[
        (\mathrm{id}\otimes\ntr)
        (h(\tilde X^N))]\,v\rangle,\qquad
	\int h\,d\mu =
	\langle v,(\mathrm{id}\otimes\tau)(h(\tilde X_{\rm free}))\,v\rangle
$$
for $h:\mathbb{R}\to\mathbb{C}$.
Theorem \ref{thm:voic} yields $\int x^p\,d\mu_N\to\int x^p\,d\mu$ for 
$p\in\mathbb{N}$ as in the proof of Corollary \ref{cor:weakfree}.
As $\|\tilde X_{\rm free}\|<\infty$, the measure $\mu$ has bounded 
support. Thus moment convergence implies weak convergence \cite[p.\ 
116]{NS06}, concluding the proof.
\end{proof}

\begin{proof}[Proof of Theorem \ref{thm:stieltjes}]
The conclusion follows immediately by taking $N\to\infty$ in Lemma
\ref{lem:stieltjesintg} and using Lemma \ref{lem:weakfreept}.
\end{proof}

\subsection{Proof of Corollary \ref{cor:c6}}

The deduction of Corollary \ref{cor:c6} from Theorem \ref{thm:stieltjes} 
follows by applying general facts about Stieltjes transforms that may be 
found in \cite[\S 6]{HT05}. For convenience, we formulate a general 
statement.

\begin{lem}
\label{lem:stcp}
Let $\mu,\nu$ be probability measures on $\mathbb{R}$ with Stieltjes
transforms
$$
        s_\mu(z) := \int \frac{1}{z-x}\,\mu(dx),\qquad
        s_\nu(z) := \int \frac{1}{z-x}\,\nu(dx).
$$ 
Suppose that
$$
        |s_\mu(z)-s_\nu(z)| \le \frac{K}{(\mathrm{Im}\,z)^p}
$$
for some $K\ge 0$, $p\in\mathbb{N}$, and all $z\in\mathbb{C}$ with 
$\mathrm{Im}\,z>0$. Then
$$
        \bigg|\int h\,d\mu-\int h\,d\nu\bigg| \le
        \frac{(\sqrt{2})^{p+1}K}{p!\pi}
	\int_\infty^\infty \bigg|
        \bigg(1+\frac{d}{dx}\bigg)^{p+1}h(x)\bigg|\,dx
	\lesssim
	K\|h\|_{W^{p+1,1}(\mathbb{R})}
$$
for every $h\in W^{p+1,1}(\mathbb{R})$.
\end{lem}

\begin{proof}
Let $h\in C^\infty_c(\mathbb{R})$.
Following \emph{verbatim} the proof of
\cite[Theorem 6.2]{HT05} yields
$$      
        \bigg|\int h\,d\mu-\int h\,d\nu\bigg| \le
        \frac{1}{\pi} \limsup_{y\downarrow 0} 
        \int_\infty^\infty \bigg|
        \bigg(1+\frac{d}{dx}\bigg)^{p+1} h(x)\bigg|
        |I_{p+1}(x+iy)|\,dx
$$
with    
$$
        |I_{p+1}(z)| \le
        \frac{1}{p!} 
        \int_0^\infty
        \frac{K}{(\mathrm{Im}\,z+t)^p}
        (\sqrt{2}t)^p e^{-t} \sqrt{2}\, dt
	\le
        \frac{(\sqrt{2})^{p+1}K}{p!}.
$$
That the integral may be bounded up to a universal constant by the Sobolev 
norm $\|h\|_{W^{p+1,1}(\mathbb{R})}$ follows as 
$\binom{p+1}{k}\frac{(\sqrt{2})^{p+1}}{p!} \lesssim 1$ for all $0\le k\le 
p+1$. The conclusion finally extends to general $h\in 
W^{p+1,1}(\mathbb{R})$ by routine approximation arguments.
\end{proof}

We can now conclude the proof.

\begin{proof}[Proof of Corollary \ref{cor:c6}]
Theorem \ref{thm:stieltjes} implies
$$
	|\EE[\ntr (z\id -X)^{-1}] - ({\ntr}\otimes\tau)(z\id-X_{\rm free})^{-1}|
	\le
	\frac{\tilde v(X)^4}{(\mathrm{Im}\,z)^5}
$$
for all $z\in\mathbb{C}$ with $\mathrm{Im}\,z>0$. Applying Lemma 
\ref{lem:stcp} with $p=5$ to the spectral distributions of $X$ and $X_{\rm 
free}$ immediately yields the conclusion.
\end{proof}

\section{Concentration of the spectrum}
\label{sec:norm}

The aim of this section is to prove our main results on the support of the 
spectrum that were formulated in Section \ref{sec:mainspec}. The general 
scheme of proof is the same as in the previous section, but some new 
ingredients are needed here.

\subsection{Moments of the resolvent}

The proof of Theorem \ref{thm:mainsp} is based on an analysis of large
moments of the resolvent $\EE[\ntr |z\id - X|^{-2p}]$. In the present 
section, we will prove an analogue of Theorem \ref{thm:stieltjes} for 
these higher moments. 

\begin{thm}
\label{thm:resmoments}
Let $A_0,\ldots,A_n\in\M_d(\mathbb{C})_{\rm sa}$. Then we have 
$$
	|\EE[\ntr |z\id - X|^{-2p}]^{\frac{1}{2p}} -
	({\ntr}\otimes\tau)(|z\id - X_{\rm free}|^{-2p})^{\frac{1}{2p}}|
	\le
	\frac{(p+2)^3}{3}
	\frac{\tilde v(X)^4}{(\mathrm{Im}\,z)^5}
$$
for every $p\in\mathbb{N}$ and $z\in\mathbb{C}$, $\mathrm{Im}\,z>0$.
\end{thm}

The proof of Theorem \ref{thm:resmoments} is similar to that of Theorems 
\ref{thm:moments} and \ref{thm:stieltjes}. Throughout this section, we 
adopt without further comment the constructions and notation of Section 
\ref{sec:construction}. In particular, $X_q^N$ is defined as in 
\eqref{eq:xq}.

\begin{lem}
\label{lem:resraw}
For any $p\in\mathbb{N}$ and $z\in\mathbb{C}$, $\mathrm{Im}\,z>0$,
we have
\begin{align*}
&	\frac{d}{dq} \EE[\ntr |z\id - X_q^N|^{-2p}] =
	p
	\sum_i \sum_{r\ge s} \bigg(\delta_{rs}-\frac{1}{N}\bigg)
	\cdot\mbox{}
\\
	&\qquad\Bigg\{\sum_{k=0}^p
	\mathrm{Re}\,
	\EE[\ntr 
	A_{irs}
	(z\id - X_q^N)^{-k-1}
	A_{irs}
	(z\id - X_q^N)^{-p-1+k}
	(\overline{z}\id - X_q^N)^{-p}
	]
\\
&\qquad\quad
	+
	\sum_{k=0}^{p-1}
	\mathrm{Re}\,
	\EE[\ntr 
	A_{irs}
	(z\id - X_q^N)^{-p-1}
	(\overline{z}\id - X_q^N)^{-k-1}
	A_{irs}
	(\overline{z}\id - X_q^N)^{-p+k}
	]\Bigg\}
\end{align*}
and
\begin{align*}
&
	0 = p
	\sum_i \sum_{r\ge s} \bigg(\delta_{rs}-\frac{1}{N}\bigg)
	\cdot\mbox{}
\\
	&\quad\Bigg\{\sum_{k=0}^p
	\mathrm{Re}\,
	\ntr 
	A_{irs}
	\EE[(z\id - X_q^N)^{-k-1}]
	A_{irs}
	\EE[(z\id - X_q^N)^{-p-1+k}
	(\overline{z}\id - X_q^N)^{-p}
	]
\\
&\quad\quad
	+
	\sum_{k=0}^{p-1}
	\mathrm{Re}\,
	\ntr 
	A_{irs}
	\EE[(z\id - X_q^N)^{-p-1}
	(\overline{z}\id - X_q^N)^{-k-1}]
	A_{irs}
	\EE[(\overline{z}\id - X_q^N)^{-p+k}
	]\Bigg\}.
\end{align*}
\end{lem}

\begin{proof}
The first identity follows by applying Lemma \ref{lem:ginterp} to the 
function
$$
	f(y)=\ntr |z\id - X^N(y)|^{-2p} =
	\ntr[(z\id - X^N(y))^{-p} (\overline{z}\id - X^N(y))^{-p}].
$$
The second identity follows by applying Lemma \ref{lem:ptrace}.
\end{proof}

We can now proceed as in Lemma \ref{lem:momentdiffeq}.

\begin{lem}
\label{lem:resdiffeq}
For any $p\in\mathbb{N}$ and $z\in\mathbb{C}$, $\mathrm{Im}\,z>0$,
we have
\begin{align*}
	&\bigg|\frac{d}{dq} \EE[\ntr |z\id - X_q^N|^{-2p}]\bigg|
	\\ &\le
	\frac{4}{3}p(p+2)^3
 	\{ q w(X_1^N)^4 + w(X_0^N,X_1^N)^4 + (1-q)w(X_0^N)^4\}
	\EE[\ntr |z\id - X_q^N|^{-2p-4}].
\end{align*}
\end{lem}

\begin{proof}
Define $X_{qt}^N$ as in the proof of Lemma \ref{lem:momentdiffeq}, and 
denote $R:=(z\id - X_q^N)^{-1}$ and $R_t:=(z\id - X_{qt}^N)^{-1}$.
Applying Corollary \ref{cor:gcov} and Lemma \ref{lem:resraw} yields
%\begingroup
%\allowdisplaybreaks
\begin{align*}
&	\frac{d}{dq} \EE[\ntr |z\id - X_q^N|^{-2p}] = \\
&
	p\,\mathrm{Re}\int_0^1
	\sum_{i,i'} 
	\sum_{r\ge s} 
	\sum_{r'\ge s'}
	\bigg(
	q\delta_{rs}\delta_{r's'}
	+\frac{1-q}{N} \delta_{rs}
	-\frac{q}{N}\delta_{r's'}
	-\frac{1-q}{N^2}
	\bigg)
	\cdot\mbox{}
\\
	&
\Bigg\{
	\sum_{k=0}^{p-1}
	\sum_{l=0}^p
	\sum_{m=0}^{p-k-1}
\!\!
	\EE[\ntr 
	A_{irs}
	R^{l+1} A_{i'r's'} R^{p-l+1}
	R^{*(k+1)}
	A_{irs}
	R_t^{*(m+1)} A_{i'r's'} R_t^{*(p-k-m)}
	] + \mbox{}
\\
&
	\sum_{k=0}^{p-1}
	\sum_{l=0}^k
	\sum_{m=0}^{p-k-1}
\!\!
	\EE[\ntr 
	A_{irs}
	R^{p+1}
	R^{*(l+1)} A_{i'r's'} R^{*(k-l+1)}
	A_{irs}
	R_t^{*(m+1)} A_{i'r's'} R_t^{*(p-k-m)}
	] + \mbox{}
\\
&
	\sum_{k=0}^p
	\sum_{l=0}^{k}
	\sum_{m=0}^{p-k}
	\EE[\ntr 
	A_{irs}
	R^{l+1} A_{i'r's'} R^{k-l+1}
	A_{irs}
	R_t^{m+1}
	A_{i'r's'} 
	R_t^{p-k-m+1}
	R_t^{*p}
	] + \mbox{}
\\
&
	\sum_{k=0}^p
	\sum_{l=0}^{k}
	\sum_{m=0}^{p-1}
	\EE[\ntr 
	A_{irs}
	R^{l+1} A_{i'r's'} R^{k-l+1}
	A_{irs}
	R_t^{p+1-k}
	R_t^{*(m+1)}
	A_{i'r's'} 
	R_t^{*(p-m)}
	] 
	\Bigg\}dt.
\end{align*}%
%\endgroup
We can now apply Lemma \ref{lem:complexint} as in the proof of 
Lemma \ref{lem:momentdiffeq} to bound
\begin{align*}
	&\bigg|\frac{d}{dq} \EE[\ntr |z\id - X_q^N|^{-2p}]\bigg|
	\\ &\le
	p\binom{2p+3}{3}
 	\{ q w(X_1^N)^4 + w(X_0^N,X_1^N)^4 + (1-q)w(X_0^N)^4\}
	\EE[\ntr |z\id - X_q^N|^{-2p-4}].
\end{align*}
The conclusion follows using $\binom{2p+3}{3} \le \frac{4}{3}(p+2)^3$.
\end{proof}

We can now complete the proof.

\begin{proof}[Proof of Theorem \ref{thm:resmoments}]
Lemma \ref{lem:resdiffeq}, the chain rule, and
Proposition \ref{prop:vxwx} yield
$$
	\bigg|\frac{d}{dq} \EE[\ntr |z\id - X_q^N|^{-2p}]^{\frac{1}{2p}}
	\bigg|
	\le
	\frac{2}{3}
	\frac{(p+2)^3}{(\mathrm{Im}\,z)^5}
 	\{ q \tilde v(X_1^N)^4 + \tilde v(X_0^N)^2\tilde v(X_1^N)^2 + 
	(1-q)\tilde v(X_0^N)^4\},
$$
where we used that
$$
	\EE[\ntr |z\id - X_q^N|^{-2p-4}]\le
	\frac{
	\EE[\ntr |z\id - X_q^N|^{-2p+1}]
	}{(\mathrm{Im}\,z)^5}
	\le
	\frac{
	\EE[\ntr |z\id - X_q^N|^{-2p}]^{1-\frac{1}{2p}} 
	}{(\mathrm{Im}\,z)^5}
$$
using $\| |z\id - X_q^N|^{-1}\|\le (\mathrm{Im}\,z)^{-1}$ and
H\"older's inequality. Integrating yields
$$
	|
	\EE[\ntr |z\id - X|^{-2p}]^{\frac{1}{2p}}-
	\EE[\ntr |z\id - X^N|^{-2p}]^{\frac{1}{2p}}
	|
	\le
	\frac{(1 + 2N^{-1})^2}{3}
	\frac{(p+2)^3\tilde v(X)^4}{(\mathrm{Im}\,z)^5}
$$
using \eqref{eq:thepoint} and Lemma \ref{lem:momentparm}.
It remains to let $N\to\infty$ using
Corollary \ref{cor:weakfree}.
\end{proof}

\subsection{Proof of Theorem \ref{thm:mainsp}}

The basic observation behind the proof is the following.
For any $D\subseteq\mathbb{C}$ and 
$z\in\mathbb{C}$, denote $d(z,D):=\inf_{z'\in D}|z-z'|$. Then
\begin{equation}
\label{eq:distspec}
	\| (z\id - X)^{-1} \| = \frac{1}{d(z,\spc(X))},
\end{equation}
and analogously for $X_{\rm free}$. The following device will enable us
to deduce concentration of the spectrum from resolvent inequalities.

\begin{lem}
\label{lem:resspec}
Let $K,L\ge 0$, and let $A,B$ be self-adjoint operators such that
$$
	\| (z\id - A)^{-1} \| \le
	C \| (z\id - B)^{-1} \| + \frac{K}{(\mathrm{Im}\,z)^5} +
	\frac{L}{(\mathrm{Im}\,z)^2}
$$
for all $z=\lambda+i\varepsilon$ with $\lambda\in\spc(A)$ and
$\varepsilon = (4K)^{\frac{1}{4}}\vee 4L$. Then
$$
	\spc(A) \subseteq \spc(B)+2C\varepsilon[-1,1].
$$
\end{lem}

\begin{proof}
By \eqref{eq:distspec}, the assumption states that
$$
	\frac{1}{\varepsilon}
	\le
	\frac{C}{\sqrt{\varepsilon^2+d(\lambda,\spc(B))^2}}
	+ \frac{K}{\varepsilon^5} +
	\frac{L}{\varepsilon^2}\qquad
	\mbox{for all }\lambda\in\spc(A).
$$
If $d(\lambda,\spc(B))>2C\varepsilon$, we would have
$\frac{1}{2}<\frac{K}{\varepsilon^4} + \frac{L}{\varepsilon} \le
\frac{1}{2}$,
which entails a contradiction. Thus we have shown that 
$d(\lambda,\spc(B))\le 2C\varepsilon$
for all $\lambda\in\spc(A)$.
\end{proof}

Our aim is now to show that the condition of Lemma \ref{lem:resspec} holds 
with high probability for $A=X$ and $B=X_{\rm free}$. To this end, we 
begin by showing that the relevant condition holds with high probability 
for a given $z\in\mathbb{C}$.

\begin{lem}
\label{lem:ressingle}
Fix $z\in\mathbb{C}$ with $\mathrm{Im}\,z>0$. Then
$$
	\mathbf{P}\bigg[\|(z\id - X)^{-1}\| \ge
	\sqrt{e}
	\|(z\id - X_{\rm free})^{-1}\|
	+
	\sqrt{e}
	\frac{(\log d+3)^3}{3}
	\frac{\tilde v(X)^4}{(\mathrm{Im}\,z)^5}
	+ \frac{\sigma_*(X)}{(\mathrm{Im}\,z)^2}t\bigg]
	\le e^{-\frac{t^2}{2}}
$$
for all $t\ge 0$.
\end{lem}

\begin{proof}
Using that $\ntr |M| \ge \frac{1}{d}\|M\|$ for every 
$M\in\M_d(\mathbb{C})$, Theorem \ref{thm:resmoments} yields
$$
	d^{-\frac{1}{2p}}
	\EE\|(z\id - X)^{-1}\|
	\le
	\|(z\id - X_{\rm free})^{-1}\|
	+
	\frac{(p+2)^3}{3}
	\frac{\tilde v(X)^4}{(\mathrm{Im}\,z)^5}
$$
for every $p\in\mathbb{N}$. Choosing $p=\lceil \log d\rceil$ yields
$$
	\EE\|(z\id - X)^{-1}\|
	\le
	\sqrt{e}
	\|(z\id - X_{\rm free})^{-1}\|
	+
	\sqrt{e}
	\frac{(\log d+3)^3}{3}
	\frac{\tilde v(X)^4}{(\mathrm{Im}\,z)^5}.
$$
It remains to note that $F(X)=\|(z\id - X)^{-1}\|$ satisfies
\begin{equation}
\label{eq:reslip}
	|F(X)-F(Y)| \le
	\|(z\id - X)^{-1}(X- Y)(z\id - Y)^{-1}\|
	\le \frac{\|X-Y\|}{(\mathrm{Im}\,z)^2}
\end{equation}
for $X,Y\in\M_d(\mathbb{C})_{\rm sa}$,
so that the conclusion follows from Corollary \ref{cor:gconcmtx}.
\end{proof}

We must now show that $\|(z\id - X)^{-1}\|$ is small with high probability 
simultaneously for all $z=\lambda+i\varepsilon$ with $\lambda\in\spc(X)$.
To create the requisite uniformity in $z$, we first need a crude \emph{a 
priori} bound on the spectrum of $X$.

\begin{lem}
\label{lem:crudesupp}
For any $t\ge 0$, we have
$$
	\mathbf{P}[\spc(X)\subseteq
	\spc(A_0) + \sigma_*(X)\{d+t\}[-1,1]]
	\ge 1- e^{-\frac{t^2}{2}}.
$$
\end{lem}

\begin{proof}
By Weyl's inequality, we have 
$|\lambda_i(X)-\lambda_i(A_0)|\le\|X-A_0\|$ for every $i$, where 
$\lambda_i(X)$ denotes the $i$th largest eigenvalue of $X$. Thus
$$
	\spc(X) \subseteq \spc(A_0) + 
	\|X-A_0\|[-1,1].
$$
By Cauchy-Schwarz, we can crudely bound
$$
	\|X-A_0\| =
	\sup_{\|v\|=\|w\|=1}
	\Bigg|\sum_{i=1}^n g_i \langle v,A_i w\rangle\Bigg|
	\le \sigma_*(X)\|g\|.
$$
Thus we have shown
$$
	\mathbf{P}[\spc(X)\subseteq
	\spc(A_0) + \sigma_*(X)\{d+t\}[-1,1]]
	\ge
	\mathbf{P}[\|g\|\le d+t].
$$
But note that the argument in the proof of Lemma \ref{lem:orthook} shows
that we may assume $n\le d^2$ without loss of generality. Thus
$\EE\|g\|\le \sqrt{n}\le d$. It remains to note that
$$
	\mathbf{P}[\|g\|\ge d+t] \le 
	\mathbf{P}[\|g\|\ge \EE\|g\| + t] 
	\le e^{-\frac{t^2}{2}}
$$
by Lemma \ref{lem:gconc}.
\end{proof}

We are now ready to prove a uniform analogue of Lemma \ref{lem:ressingle}.

\begin{lem}
\label{lem:resuniform}
Fix $\varepsilon>0$. Then
\begin{multline*}
	\mathbf{P}\bigg[
	\|(z\id - X)^{-1}\| \le
	\sqrt{e}
	\|(z\id - X_{\rm free})^{-1}\|
	+ 
	\sqrt{e}
	\frac{(\log d+3)^3}{3}
	\frac{\tilde v(X)^4}{(\mathrm{Im}\,z)^5}
\\
	+ (\sqrt{e}+2) \frac{\sigma_*(X)}{(\mathrm{Im}\,z)^2}
	(4\sqrt{\log d}+t)
	\quad\mbox{for all }z\in \spc(X)+i\varepsilon
	\bigg] \ge
	1-e^{-\frac{t^2}{2}}
\end{multline*}
for all $t\ge 0$.
\end{lem}

\begin{proof}
Define the (nonrandom) set
$$
	\Omega_t := \spc(A_0) + \sigma_*(X)\{d+t\}[-1,1]\subset\mathbb{R}.
$$
As $A_0$ has at most $d$ distinct eigenvalues, $\Omega_t$
is the union of at most $d$ intervals of length $2\sigma_*(X)\{d+t\}$.
We can therefore find $\mathcal{N}_t\subset\Omega_t$ of cardinality
$|\mathcal{N}_t|\le \frac{2d(d+t)}{t}$ such that each 
$\lambda\in\Omega_t$ satisfies $d(\lambda,\mathcal{N}_t)\le\sigma_*(X)t$.

Now note that we can estimate as in \eqref{eq:reslip}
$$
	|\|(z\id -X)^{-1}\|-\|(z'\id-X)^{-1}\|| \le
	\frac{|z-z'|}{\mathrm{Im}\,z\cdot\mathrm{Im}\,z'},
$$
and similarly for $X_{\rm free}$.
We therefore obtain
\begin{alignat*}{2}
	&\mathbf{P}\bigg[
	\|(z\id - X)^{-1}\| \le
	\sqrt{e}
	\|(z\id - && X_{\rm free})^{-1}\|
	+ 
	\sqrt{e}
	\frac{(\log d+3)^3}{3}
	\frac{\tilde v(X)^4}{(\mathrm{Im}\,z)^5}
\\
& &&
	+ (\sqrt{e}+2) \frac{\sigma_*(X)}{(\mathrm{Im}\,z)^2}t
	\quad\mbox{for all }z\in \Omega_t+i\varepsilon
	\bigg] \ge\mbox{}
\\
	&
	\mathbf{P}\bigg[\|(z\id - X)^{-1}\| \le
	\sqrt{e}
	\|(z\id - && X_{\rm free})^{-1}\|
	+ 
	\sqrt{e}
	\frac{(\log d+3)^3}{3}
	\frac{\tilde v(X)^4}{(\mathrm{Im}\,z)^5} 
\\
& &&
	+ \frac{\sigma_*(X)}{(\mathrm{Im}\,z)^2}t
	\quad\mbox{for all }z\in \mathcal{N}_t+i\varepsilon
	\bigg]
	\ge
	1-|\mathcal{N}_t|e^{-\frac{t^2}{2}},
\end{alignat*}
where we used that $\mathrm{Im}\,z=\mathrm{Im}\,z'=\varepsilon$ for
$z,z'\in\Omega_t+i\varepsilon$ in the first inequality, and we used the 
union bound and Lemma \ref{lem:ressingle} in the second inequality. In 
particular,
\begin{multline*}
	\mathbf{P}\bigg[
	\|(z\id - X)^{-1}\| \le
	\sqrt{e}
	\|(z\id - X_{\rm free})^{-1}\|
	+ 
	\sqrt{e}
	\frac{(\log d+3)^3}{3}
	\frac{\tilde v(X)^4}{(\mathrm{Im}\,z)^5}
\\
	+ (\sqrt{e}+2) \frac{\sigma_*(X)}{(\mathrm{Im}\,z)^2}t
	\quad\mbox{for all }z\in \spc(X)+i\varepsilon
	\bigg] \ge
	1-(|\mathcal{N}_{t}|+1)e^{-\frac{t^2}{2}}
\end{multline*}
by Lemma \ref{lem:crudesupp}.
It remains to note that 
$(|\mathcal{N}_{t+a}|+1)e^{-\frac{(t+a)^2}{2}}\le e^{-\frac{t^2}{2}}$
if we choose $a=4\sqrt{\log d}$ (recalling the standing assumption
$d\ge 2$).
\end{proof}

The proof of Theorem \ref{thm:mainsp} now follows readily.

\begin{proof}[Proof of Theorem \ref{thm:mainsp}]
Combining Lemmas \ref{lem:resspec} and \ref{lem:resuniform} yields
$$
	\mathbf{P}\big[
	\spc(X)\subseteq \spc(X_{\rm free}) +
	C\{
	\tilde v(X)(\log d)^{\frac{3}{4}}
	+ \sigma_*(X)(\sqrt{\log d}+t)\}
	[-1,1]\big] \ge
	1-e^{-t^2}
$$
for all $t\ge 0$, where $C$ is a universal constant.
It remains to note that we can estimate
$\sigma_*(X)\sqrt{\log d}\lesssim\tilde v(X)(\log d)^{\frac{3}{4}}$
as $\sigma_*(X)\le \tilde v(X)$.
\end{proof}

\subsection{Proof of Corollary \ref{cor:norm}}

The deduction of Corollary \ref{cor:norm} from Theorem \ref{thm:mainsp}
is nearly immediate; we spell out the details for completeness.

\begin{proof}[Proof of Corollary \ref{cor:norm}]
When $A_0,\ldots,A_n\in\M_d(\mathbb{C})_{\rm sa}$ are self-adjoint, the 
probability bound follows immediately from Theorem \ref{thm:mainsp}. This 
bound extends directly to general $A_0,\ldots,A_n\in\M_d(\mathbb{C})$
by Remark \ref{rem:sa}. The bound on the expectation is now obtained by
integrating the probability bound. More precisely, we have
\begin{align*}
	&\EE[(\|X\|-\|X_{\rm free}\|-C\tilde v(X)(\log d)^{\frac{3}{4}})_+]
	\\
	&=
	\int_0^\infty
	\mathbf{P}[\|X\|\ge \|X_{\rm free}\|
	+C\tilde v(X)(\log d)^{\frac{3}{4}}+s]\,ds
	\\
	&\le
	\int_0^\infty
	e^{-s^2/C^2\sigma_*(X)^2}\,ds
	=C'\sigma_*(X)
\end{align*}
for a universal constant $C'$. It follows that
$$
	\EE\|X\|
	\le
	\|X_{\rm free}\|+C\tilde v(X)(\log d)^{\frac{3}{4}}
	+ C'\sigma_*(X).
$$
It remains to note that as $\sigma_*(X)\le\tilde v(X)$, the last term may 
be eliminated at the expense of choosing a slightly larger universal
constant $C$.
\end{proof}

\section{Strong asymptotic freeness}
\label{sec:strongfree}

The aim of this section is to prove our results on asymptotic freeness 
that were formulated in Section \ref{sec:mainsaf}. The proof of Theorem 
\ref{thm:free} is divided into two parts. In Section \ref{sec:pfweakfree} 
we will prove weak asymptotic freeness (part \textit{a}). This part of the 
proof is elementary and uses only the basic estimates of Section 
\ref{sec:wxvx}; when specialized to Wigner matrices, it yields a 
self-contained proof of Voiculescu's Theorem \ref{thm:voic}. In Section 
\ref{sec:pfstrongfree}, we will prove strong asymptotic freeness (part 
\textit{b}) by combining Theorem \ref{thm:mainsp} with the linearization 
trick of \cite{HT05} and concentration estimates. Finally, Corollary 
\ref{cor:oldie} will be deduced from Theorem \ref{thm:free} in Section 
\ref{sec:pfcoroldie}.

\subsection{Weak asymptotic freeness}
\label{sec:pfweakfree}

The aim of this section is to prove part \textit{a} of Theorem 
\ref{thm:free}. By linearity of the trace, it evidently suffices to 
assume
$$
	p(H_1,\ldots,H_m)=H_{k_1}\cdots H_{k_q}
$$
is a monomial of degree $q$ for some $q\in\mathbb{N}$ and $1\le 
k_1,\ldots,k_q\le m$. This assumption will be made throughout the 
proof of part \textit{a} of Theorem \ref{thm:free}.

Throughout this section, we let $H_1^N,\ldots,H_m^N$ be defined as in 
Theorem \ref{thm:free}. We begin with some preliminary observations. 
First, we note the following.

\begin{lem}
\label{lem:wafnck}
We have $\sup_{N,k}\EE[\ntr |H_k^N-\EE[H_k^N]|^q]^{\frac{1}{q}} <\infty$
for every $q\in\mathbb{N}$.
\end{lem}

\begin{proof}
By assumption, $\sigma(H_k^N)^2 = \|\EE (H_k^N-\EE[H_k^N])^2\| = 1+o(1)$. 
The conclusion follows by the noncommutative Khintchine inequality, cf.\
\cite[\S 9.8]{Pis03} or \cite[\S 3.1]{vH17}.
\end{proof}

Before we proceed to the main part of the proof, we perform a simple 
reduction: we show that it suffices to assume $\EE[H_k^N]=0$. This 
elementary observation will avoid unnecessary notational complications.

\begin{lem}
\label{lem:wafred}
Denote $\bar H_k^N := H_k^N - \EE[H_k^N]$. Then we have 
$$
	\lim_{N\to\infty}
	\EE \ntr |H_{k_1}^N\cdots H_{k_q}^N-\bar H_{k_1}^N\cdots 
	\bar H_{k_q}^N|
	= 0.
$$
\end{lem}

\begin{proof}
Note that
$$
	H_{k_1}^N\cdots H_{k_q}^N-\bar H_{k_1}^N\cdots 
        \bar H_{k_q}^N =
	\sum_{l=1}^q
	\bar H_{k_1}^N\cdots
	\bar H_{k_{l-1}}^N
	\EE[H_{k_l}^N]
	H_{k_{l+1}}^N\cdots H_{k_q}^N.
$$
Thus
$$
	\EE\ntr|H_{k_1}^N\cdots H_{k_q}^N-\bar H_{k_1}^N\cdots 
        \bar H_{k_q}^N|
	\le
	q\max_{k,l}\|\EE[H_k^N]\|
	\{(\EE \ntr |\bar H_l^N|^q)^{\frac{1}{q}}+
	\|\EE[H_l^N]\|
	\}^{q-1}
$$
by H\"older's inequality. As $\|\EE[H_k^N]\|=o(1)$, it remains to note 
that $\EE\ntr|\bar H_k^N|^q$ is uniformly bounded as $N\to\infty$
by Lemma \ref{lem:wafnck}.
\end{proof}

By Lemma \ref{lem:wafred}, we can assume without loss of generality in the 
remainder of the proof of part \textit{a} of Theorem \ref{thm:free} that 
$\EE[H_k^N]=0$ for all $k$.

We now turn the the main part of the proof. The basic tool we will use is 
the classical Wick formula for Gaussian moments \cite[Theorem 22.3]{NS06}, 
which should be compared with its free counterpart in Definition 
\ref{defn:freefam}.

\begin{lem}[Wick formula]
\label{lem:wick}
Let $g_1,\ldots,g_n$ be i.i.d.\ standard Gaussians. Then
$$
	\EE[g_{k_1}\cdots g_{k_q}] =
	\sum_{\pi\in\mathrm{P}_2([q])}
	\prod_{\{i,j\}\in\pi} \delta_{k_ik_j}
$$
for every $q\ge 1$ and $k_1,\ldots,k_q\in[n]$.
\end{lem}

From the Wick formula, we deduce the following.

\begin{cor}
\label{cor:superwick}
Suppose $\EE[H_k^N]=0$ for all $k\in[m]$, and let
$\mathrm{k}=(k_1,\ldots,k_q)$. Then
$$
	\EE[\ntr H_{k_1}^N\cdots H_{k_q}^N] =
	\sum_{\pi\in\mathrm{P}_2([q])}
	\EE[\ntr H_{1|\pi,\mathrm{k}}^N\cdots H_{q|\pi,\mathrm{k}}^N]
	\prod_{\{r,s\}\in\pi} \delta_{k_rk_s},
$$
where $H_{1|\pi,\mathrm{k}}^N,\ldots,H_{q|\pi,\mathrm{k}}^N$ are jointly 
Gaussian random matrices defined as follows:
\begin{enumerate}[1.]
\item $H_{r|\pi,\mathrm{k}}^N$ has the same distribution as $H_{k_r}^N$.
\item $H_{r|\pi,\mathrm{k}}^N=H_{s|\pi,\mathrm{k}}^N$ if $\{r,s\}\in\pi$.
\item $H_{r|\pi,\mathrm{k}}^N$ and $H_{s|\pi,\mathrm{k}}^N$ are 
independent if $r\ne s$, $\{r,s\}\not\in\pi$.
\end{enumerate}
\end{cor}

\begin{proof}
As $\EE[H_k^N]=0$, we may write
$$
	H_k^N = \sum_{i=1}^n g_{ki}\,A_{ki},
$$
where $g_{ki}$ are i.i.d.\ standard Gaussians and 
$A_{ki}\in\M_d(\mathbb{C})_{\rm sa}$. Then 
$$
	\EE[\ntr H_{1|\pi,\mathrm{k}}^N\cdots H_{q|\pi,\mathrm{k}}^N]
	\prod_{\{r,s\}\in\pi} \delta_{k_rk_s}
	=
	\sum_{i_1,\ldots,i_q}
	\ntr A_{k_1i_1}\cdots A_{k_qi_q}
	\prod_{\{r,s\}\in\pi} \delta_{k_rk_s}\delta_{i_ri_s}
$$
by construction. On the other hand
$$
	\EE[\ntr H_{k_1}^N\cdots H_{k_q}^N] =
	\sum_{i_1,\ldots,i_q}
	\ntr A_{k_1i_1}\cdots A_{k_qi_q}
	\sum_{\pi\in\mathrm{P}_2([q])}
	\prod_{\{r,s\}\in\pi} \delta_{k_rk_s}\delta_{i_ri_s}
$$
by Lemma \ref{lem:wick}, completing the proof.
\end{proof}

The main idea that gives rise to weak asymptotic freeness is that the 
terms in Corollary \ref{cor:superwick} that correspond to crossing 
pairings are asymptotically negligible. This will follow readily from the 
following lemma.

\begin{lem}
\label{lem:wafcr}
In the setting of Corollary \ref{cor:superwick}, we have
$$
	|\EE[\ntr H_{1|\pi,\mathrm{k}}^N\cdots H_{q|\pi,\mathrm{k}}^N]|
	\le
	\max_{k,l} w(H_k^N,H_l^N)^4
	\max_k \EE[\ntr |H_k^N|^{q-4}]
$$
for any crossing pairing 
$\pi\in\mathrm{P}_2([q])\backslash\mathrm{NC}_2([q])$ such that
$k_r=k_s$ for all $\{r,s\}\in\pi$.
\end{lem}

\begin{proof}
By assumption, the exist $\{r_1,s_1\},\{r_2,s_2\}\in\pi$ such that
$r_1<r_2<s_1<s_2$. Computing the expectation with respect to these 
indices only yields
\begin{align*}
	&\EE[\ntr H_{1|\pi,\mathrm{k}}^N\cdots H_{q|\pi,\mathrm{k}}^N] =
\\
	&\quad\sum_{i,j}
	\EE[\ntr 
	H_{1|\pi,\mathrm{k}}^N\cdots
	H_{r_1-1|\pi,\mathrm{k}}^N A_{k_{r_1}i}
	H_{r_1+1|\pi,\mathrm{k}}^N\cdots
        H_{r_2-1|\pi,\mathrm{k}}^N A_{k_{r_2}j} 
	H_{r_2+1|\pi,\mathrm{k}}^N\cdots
	\\
	&\qquad\qquad\quad
        H_{s_1-1|\pi,\mathrm{k}}^N 
	A_{k_{r_1}i}
	H_{s_1+1|\pi,\mathrm{k}}^N\cdots
        H_{s_2-1|\pi,\mathrm{k}}^N A_{k_{r_2}j}
	H_{s_2+1|\pi,\mathrm{k}}^N\cdots
	H_{q|\pi,\mathrm{k}}^N],
\end{align*}
where we used the notation in the proof of Corollary 
\ref{cor:superwick}. Cyclically permuting the trace, applying Lemma 
\ref{lem:complexint}, and using H\"older's inequality yields
$$
	|\EE[\ntr H_{1|\pi,\mathrm{k}}^N\cdots H_{q|\pi,\mathrm{k}}^N]|
	\le
	w(H_{k_{r_1}}^N,H_{k_{r_2}}^N)^4
	\prod_{l\in[q]\backslash\{r_1,r_2,s_1,s_2\}}
	\EE[\ntr |H_{k_l}^N|^{q-4}]^{\frac{1}{q-4}}.
$$
The conclusion follows readily.
\end{proof}

On the other hand, the assumption $\|\EE[(H_k^N)^2]-\id\|\to 0$ implies 
the following.

\begin{lem}
\label{lem:wafnc}
In the setting of Corollary \ref{cor:superwick}, we have
$$
	\lim_{N\to\infty}
	\EE[\ntr H_{1|\pi,\mathrm{k}}^N\cdots H_{q|\pi,\mathrm{k}}^N]
	=1
$$
for any noncrossing pairing $\pi\in\mathrm{NC}_2([q])$ such that
$k_r=k_s$ for all $\{r,s\}\in\pi$.
\end{lem}

\begin{proof}
Any noncrossing pairing $\pi\in\mathrm{NC}_2([q])$ must contain at least 
one adjacent pair $\{r,r+1\}\in\pi$. By cyclic permutation of the trace, we 
may assume $\{q-1,q\}\in\pi$. Computing the expectation with respect to 
this 
pair yields
$$
	\EE[\ntr H_{1|\pi,\mathrm{k}}^N\cdots H_{q|\pi,\mathrm{k}}^N] =
	\EE[\ntr 
	H_{1|\pi,\mathrm{k}}^N\cdots
	H_{q-2|\pi,\mathrm{k}}^N\EE[(H_{k_q}^N)^2]
	].
$$
In particular, we obtain using H\"older's inequality
\begin{multline*}
	|\EE[\ntr H_{1|\pi,\mathrm{k}}^N\cdots H_{q|\pi,\mathrm{k}}^N]-
	\EE[\ntr H_{1|\pi,\mathrm{k}}^N\cdots H_{q-2|\pi,\mathrm{k}}^N]|
\\
	\le
	\|\EE[(H_{k_q}^N)^2]-\id\|
	\prod_{k=1}^{q-2}
	\EE[\ntr |H_k^N|^{q-2}]^{\frac{1}{q-2}}.
\end{multline*}
As $\pi\backslash\{\{q-1,q\}\}\in\mathrm{NC}_2([q-2])$, we
may iterate this 
procedure to obtain
$$
	|\EE[\ntr H_{1|\pi,\mathrm{k}}^N\cdots H_{q|\pi,\mathrm{k}}^N]-1|
	\le
	\frac{q}{2}\max_k
	\|\EE[(H_k^N)^2]-\id\|
	\max_k\max_{l\le q}\EE[\ntr |H_k^N|^l].
$$
The conclusion follows as $\|\EE[(H_k^N)^2]-\id\|\to 0$ as $N\to\infty$
by assumption, while $\EE[\ntr |H_k^N|^l]$ is uniformly bounded for all
$l\le q$ and $N\ge 1$ by Lemma \ref{lem:wafnck}.
\end{proof}

The proof of weak asymptotic freeness is now readily completed.

\begin{proof}[Proof of Theorem \ref{thm:free}: part \textit{a}]
By Lemma \ref{lem:wafred}, we may assume without loss of generality
that $\EE[H_k^N]=0$ for all $k,N$. By Lemma \ref{lem:wafnc} and 
Definition \ref{defn:freefam}, we have
$$
	\lim_{N\to\infty}
	\sum_{\pi\in\mathrm{NC}_2([q])}
	\EE[\ntr H_{1|\pi,\mathrm{k}}^N\cdots H_{q|\pi,\mathrm{k}}^N]
	\prod_{\{r,s\}\in\pi} \delta_{k_rk_s} =
	\tau(s_{k_1}\cdots s_{k_q}).
$$
On the other hand, Lemma \ref{lem:wafcr} and
Proposition \ref{prop:vxwx} yield
\begin{multline*}
	\Bigg|
	\sum_{\pi\in\mathrm{P}_2([q])\backslash\mathrm{NC}_2([q])}
	\EE[\ntr H_{1|\pi,\mathrm{k}}^N\cdots H_{q|\pi,\mathrm{k}}^N]
	\prod_{\{r,s\}\in\pi} \delta_{k_rk_s}
	\Bigg|
\\	\le
	|\mathrm{P}_2([q])|
	\max_k
	v(H_k^N)^2
	\sigma(H_k^N)^2
	\max_k\EE[\ntr |H_k^N|^{q-4}].
\end{multline*}
As $\sigma(H_k^N)$ and $\EE[\ntr |H_k^N|^{q-4}]$ are uniformly bounded as 
$N\to\infty$ by Lemma
\ref{lem:wafnck}, the assumption $v(H_k^N)=o(1)$ implies the
right-hand side vanishes as
$N\to\infty$. Thus 
$$
	\lim_{N\to\infty}\EE[\ntr H_{k_1}^N\cdots H_{k_q}^N]
	= \tau(s_{k_1}\cdots s_{k_q})
$$
for all $q\in\mathbb{N}$ and $1\le k_1,\ldots,k_q\le m$
by Corollary \ref{cor:superwick}. The 
conclusion extends immediately to any noncommutative polynomial 
$p(H_1^N,\ldots H_m^N)$ by linearity.
\end{proof}

\subsection{Strong asymptotic freeness}
\label{sec:pfstrongfree}

The main idea behind the proof of part \textit{b} of Theorem 
\ref{thm:free} is that the behavior of polynomials can be controlled by 
that of associated random matrices of the form \eqref{eq:model}. We have 
already encountered a very simple form of such a linearization argument in 
Lemma \ref{lem:linquad}, where it was used to obtain nonasymptotic bounds 
for sample covariance matrices. As we are presently interested in 
asymptotics, we can directly invoke the abstract linearization argument of 
Haagerup and Thorbj{\o}rnsen \cite[Lemma 1 and pp.\ 758--760]{HT05}.

\begin{thm}[Haagerup-Thorbj{\o}rnsen]
\label{thm:htlin}
Suppose that for every $\varepsilon>0$, $d'\in\mathbb{N}$, and
$A_0,\ldots,A_m\in\M_{d'}(\mathbb{C})_{\rm sa}$,
the following holds almost surely:
$$
	\spc(A_0\otimes\id + \textstyle{\sum_{k=1}^m}A_k\otimes H_k^N)
	\subseteq
	\spc(A_0\otimes\id + \textstyle{\sum_{k=1}^m}A_k\otimes s_k)
	+ [-\varepsilon,\varepsilon]
$$
eventually as $N\to\infty$. Then
$$
	\limsup_{N\to\infty}\|p(H_1^N,\ldots,H_m^N)\|
	\le \|p(s_1,\ldots,s_m)\|\quad\mbox{a.s.}
$$
for every noncommutative polynomial $p$.
\end{thm}

Let again $H_1^N,\ldots,H_m^N$ be defined as in
Theorem \ref{thm:free}. Then we may write
$$
	H_k^N = B_{k0}^N + \sum_{i=1}^{n_k^N} g_{ki}^N B_{ki}^N,
$$
where $n_k^N\in\mathbb{N}$, $B_{ki}^N\in\M_{d(N)}(\mathbb{C})_{\rm sa}$,
and $(g_{ki}^N)_{k\in[m],i\in[n_k^N]}$ are i.i.d.\ standard Gaussians for 
each $N$ (we need not specify the joint distribution for different $N$, 
but we assume all random matrices have been placed on a single probability 
space). Let us fix in the following any $d'\in\mathbb{N}$ and 
$A_0,\ldots,A_m\in\M_{d'}(\mathbb{C})_{\rm sa}$, and define
$$
	\Xi^N := A_0\otimes\id + \sum_{k=1}^m A_k\otimes H_k^N =
	A_0\otimes\id +
	\sum_{k=1}^m
	A_k\otimes B_{k0}^N +
	\sum_{k=1}^m
	\sum_{i=1}^{n_k^N}
	(A_k\otimes B_{ki}^N)\,g_{ki}^N
$$
and its free analogue
$$
	\Xi^N_{\rm free} := 
	A_0\otimes\id +
	\sum_{k=1}^m
	A_k\otimes B_{k0}^N +
	\sum_{k=1}^m
	\sum_{i=1}^{n_k^N}
	A_k\otimes B_{ki}^N\otimes s_{ki},
$$
where $(s_{ki})_{k,i}$ is a free semicircular family.
Then we have the following.

\begin{lem}
\label{lem:saffirst}
If $v(H_k^N)=o((\log d(N))^{-\frac{3}{2}})$ as $N\to\infty$ for all $k$, 
then
$$
	\spc(A_0\otimes\id + \textstyle{\sum_{k=1}^m}A_k\otimes H_k^N) 
	\subseteq \spc(\Xi^N_{\rm free}) +
	[-\varepsilon,\varepsilon]
$$
eventually as $N\to\infty$ a.s.\ for every $\varepsilon>0$.
\end{lem}

\begin{proof}
As $H_1^N,\ldots,H_m^N$ are independent, we have
$$
	\mathrm{Cov}(\Xi^N)=\sum_{k=1}^m\mathrm{Cov}(A_k\otimes H_k^N)
	=\sum_{k=1}^m \iota(A_k)\iota(A_k)^*\otimes\mathrm{Cov}(H_k^N),
$$
where
$\iota:\M_d(\mathbb{C})\to\mathbb{C}^{d^2}$ maps a matrix to its vector of
entries.
As $A_1,\ldots,A_m$ are fixed, it follows that
$v(\Xi^N) = \|\mathrm{Cov}(\Xi^N)\|^{\frac{1}{2}} = o((\log 
d(N))^{-\frac{3}{2}})$.
On the other hand, 
$$
	\sigma(\Xi^N)^2 = 
	\Bigg\|\sum_{k=1}^m A_k^2\otimes \EE[(H_k^N)^2]\Bigg\|,
$$
so $\|\EE[(H_k^N)^2]-\id\|=o(1)$ implies that $\sigma(\Xi^N)=O(1)$.
Therefore
$$
	\PP\big[\spc(\Xi^N) \subseteq \spc(\Xi^N_{\rm free}) +
	\varepsilon_N[-1,1]
	\big]
	\ge 1-e^{-(\log N)^3}
$$
by Theorem \ref{thm:mainsp} and $d(N)\ge N$, where 
$$
	\varepsilon_N := C\{\tilde v(\Xi^N)(\log d'd(N))^{\frac{3}{4}}
	+\sigma_*(\Xi^N)(\log d(N))^{\frac{3}{2}}\} = o(1)
$$
as $\sigma_*(\Xi^N)\le v(\Xi^N)$.
It remains to note that as $\sum_N e^{-(\log N)^3}<\infty$, the conclusion 
follows from the Borel-Cantelli lemma.
\end{proof}

On the other hand, $\|\EE[(H_k^N)^2]-\id\|=o(1)$ ensures that the spectrum 
of $\Xi^N_{\rm free}$ concentrates around that of $A_0\otimes\id + 
\sum_{k=1}^m A_k\otimes s_k$. This is the analogue in the present setting 
of Lemma \ref{lem:wafnc} in the previous section. We first prove a 
special case.

\begin{lem}
\label{lem:safncsimple}
In the special case that $\EE[H_k^N]=0$ and $\EE[(H_k^N)^2]=\id$ for all 
$k$,
$$
	\spc(\Xi^N_{\rm free}) =
	\spc(A_0\otimes\id + \textstyle{\sum_{k=1}^m}A_k\otimes s_k).
$$
\end{lem}

\begin{proof}
In the present setting, we may write
$$
	\Xi^N_{\rm free} =
	A_0\otimes\id +
	\sum_{k=1}^m 
	A_k\otimes H^N_{k,{\rm free}},
$$
where
$$
	H^N_{k,{\rm free}} =
	\sum_{i=1}^{n_k^N} B_{ki}^N\otimes s_{ki}
$$
satisfies $(\mathrm{id}\otimes\tau)((H^N_{k,{\rm free}})^2)=
\sum_i (B_{ki}^N)^2 = \id$. By Definition \ref{defn:freefam}, we may 
compute
$$
	({\ntr}\otimes\tau)(H_{k_1,{\rm free}}^N\cdots
	H_{k_q,{\rm free}}^N) =
	\sum_{\pi\in\mathrm{NC}_2([q])}
	\sum_{i_1,\ldots,i_q}
	\ntr(B_{k_1i_1}^N\cdots B_{k_qi_q}^N)
	\prod_{\{r,s\}\in\pi}
	\delta_{k_rk_s}\delta_{i_ri_s}.
$$
It follows exactly as in the proof of Lemma \ref{lem:wafnc} that
$$
	({\ntr}\otimes\tau)(H_{k_1,{\rm free}}^N\cdots
        H_{k_q,{\rm free}}^N) =
	\tau(s_{k_1}\cdots s_{k_q})
$$
for all $q\in\mathbb{N}$, $1\le k_1,\ldots,k_q\le m$, and $N\ge 1$.
In particular, it follows that
$$
	({\ntr}\otimes\tau)((\Xi_{\rm free}^N)^q) =
	({\ntr}\otimes\tau)((A_0\otimes\id + 
	{\textstyle\sum_{k=1}^m}A_k\otimes s_k)^q)
$$
for all $q\in\mathbb{N}$. As $\Xi_{\rm free}^N$ is a bounded operator, the
equality of all moments implies that the spectral distributions of 
$\Xi_{\rm free}^N$ and $A_0\otimes\id + {\textstyle\sum_{k=1}^m}A_k\otimes 
s_k$ coincide. Therefore, as ${\ntr}\otimes\tau$ is a faithful state, 
their spectra coincide as well.
\end{proof}

The general case now follows by a perturbation argument.

\begin{lem}
\label{lem:safnc}
When $\|\EE[H_k^N]\|=o(1)$ and $\|\EE[(H_k^N)^2]-\id\|=o(1)$ for all $k$,
$$
	\spc(\Xi^N_{\rm free}) 
	\subseteq 
	\spc(A_0\otimes\id + \textstyle{\sum_{k=1}^m}A_k\otimes s_k) +
	[-\varepsilon,\varepsilon]
$$
eventually as $N\to\infty$ for every $\varepsilon>0$.
\end{lem}

\begin{proof}
Define
$$
	\tilde\Xi^N_{\rm free} :=
	A_0\otimes\id +
	\sum_{k=1}^m 
	A_k\otimes \tilde H^N_{k,{\rm free}},
$$
where
$$
	\tilde H^N_{k,{\rm free}} =
	\frac{
	\sum_{i=1}^{n_k^N} B_{ki}^N\otimes s_{ki}
	+ \big(\|\EE[(H_k^N)^2]\|\id-\EE[(H_k^N)^2]\big)^{\frac{1}{2}}
	\otimes\tilde s_k
	}{
	\|\EE[(H_k^N)^2]\|^{\frac{1}{2}}
	}
$$
and $(s_{ki},\tilde s_k)_{k,i}$ is a free semicircular family.
As by construction $(\mathrm{id}\otimes\tau)(\tilde H^N_{k,{\rm free}})=0$
and $(\mathrm{id}\otimes\tau)((\tilde H^N_{k,{\rm free}})^2)=\id$,
Lemma \ref{lem:safncsimple} implies that
$$
	\spc(\tilde\Xi^N_{\rm free}) = 
	\spc(A_0\otimes\id + {\textstyle\sum_{k=1}^m}A_k\otimes s_k).
$$
Next, we estimate
$$
	\|\Xi_{\rm free}^N-\tilde\Xi_{\rm free}^N\| \le
	\sum_{k=1}^m \|A_k\| \{\|\EE[H_k^N]\| +
	\|H_{k,{\rm free}}^N-\tilde H_{k,{\rm free}}^N\|\},
$$
where $H_{k,{\rm free}}^N$ is defined in the proof of Lemma 
\ref{lem:safncsimple}. Moreover, we have
$$
	\|H_{k,{\rm free}}^N-\tilde H_{k,{\rm free}}^N\| \le
	\Bigg|1-\frac{1}{\|\EE[(H_k^N)^2]\|^{\frac{1}{2}}}\Bigg|
	\|H_{k,\rm free}^N\| +
	\frac{2\big\|
	\|\EE[(H_k^N)^2]\|\id-\EE[(H_k^N)^2]
	\big\|^{\frac{1}{2}}}{\|\EE[(H_k^N)^2]\|^{\frac{1}{2}}}
$$
using $\|\tilde s_k\|=2$. Now note that
$\|\EE[H_k^N]\|=o(1)$ and $\|\EE[(H_k^N)^2]-\id\|=o(1)$ imply 
$\|H_{k,{\rm free}}^N\|=O(1)$ by Lemma \ref{lem:freenck}. Thus the
above expressions yield
$$
	\lim_{N\to\infty}\|\Xi_{\rm free}^N-\tilde\Xi_{\rm free}^N\|=0.
$$
In particular, this implies by \eqref{eq:reslip} that
$$
	\|(z\id - \Xi_{\rm free}^N)^{-1}\| \le
	\|(z\id - \tilde\Xi_{\rm free}^N)^{-1}\| + 
	\frac{\varepsilon}{(\mathrm{Im}\,z)^2}
$$
for all $z\in\mathbb{C}$, $\mathrm{Im}\,z>0$ holds
eventually as $N\to\infty$ for every $\varepsilon>0$. The
conclusion now follows by invoking Lemma \ref{lem:resspec}.
\end{proof}

Before we can conclude the proof, we require a concentration argument.

\begin{lem}
\label{lem:safconc}
If $v(H_k^N)=o((\log d(N))^{-\frac{3}{2}})$ as $N\to\infty$ for all $k$,
then
\begin{align*}
	&\lim_{N\to\infty}
	|\|p(H_1^N,\ldots,H_m^N)\|-
	\EE[\|p(H_1^N,\ldots,H_m^N)\|]|=0\quad\mbox{a.s.},
\\	&\lim_{N\to\infty}
	|\ntr p(H_1^N,\ldots,H_m^N)-
	\EE[\ntr p(H_1^N,\ldots,H_m^N)]|=0\quad\mbox{a.s.}
\end{align*}
for every noncommutative polynomial $p$.
\end{lem}

\begin{proof}
Fix a noncommutative polynomial $p$ of degree $q$. Define a function $f$ 
either as $f(g)=\|p(H_1^N,\ldots,H_m^N)\|$ or $f(g)=\ntr 
p(H_1^N,\ldots,H_m^N)$, where $g=(g_{ki}^N)_{k\in[m],i\in[n_k^N]}$.
We may assume without loss of generality that $n_k^N\le d(N)^2$ as
in the proof of Lemma \ref{lem:orthook}, so the random vector $g$ has 
dimension at most $md(N)^2$.

We begin by estimating as in the proofs of Lemma 
\ref{lem:wafred} and Corollary \ref{cor:gconcmtx} that
$$
	|f(g)-f(g')| \le L\|g-g'\|,\qquad\quad
	L=C(p) 4^{q-1} \max_k \sigma_*(H_k^N)
$$
for all $g,g'\in\Omega$, where
$$
	\Omega :=
	\{g:\|H_k^N\|\le 4\mbox{ for all }k\}
$$
and $C(p)$ is a constant that depends only on the polynomial $p$.

By Corollary \ref{cor:norm} and a union
bound, we can estimate
$$
	\mathbf{P}[\Omega^c] \le
	\sum_{k=1}^m \mathbf{P}[\|H_k^N\|>4]
	\le
	m e^{-(\log d(N))^3}
$$
eventually as $N\to\infty$, where we used that
$\sigma_*(H_k^N)\le v(H_k^N)=o((\log d(N))^{-\frac{3}{2}})$ and 
$\|H_{k,{\rm free}}^N\|\le \|\EE[H_k^N]\|+2\sigma(H_k^N)=2+o(1)$ 
by Lemma \ref{lem:freenck}.

As $f$ is $L$-Lipschitz on $\Omega$, the 
classical Lipschitz extension theorem of Kirszbraun ensures the existence 
of a globally 
$L$-Lipschitz function $\tilde f$ such that $\tilde f(g)=f(g)$ for 
$g\in\Omega$. We can therefore estimate for sufficiently large $N$
\begin{align*}
	|\EE[f(g)]-\EE[\tilde f(g)]| &=
	|\EE[(f(g)-\tilde f(g))1_{\Omega^c}]|
	\\
	&
	\le \mathbf{P}[\Omega^c]^{\frac{1}{2}}
	\{(\EE|f(g)|^2)^{\frac{1}{2}}+
	(\EE|\tilde f(g)|^2)^{\frac{1}{2}}\}.	
	\\
	&
	\le \mathbf{P}[\Omega^c]^{\frac{1}{2}}
	\{(\EE|f(g)|^2)^{\frac{1}{2}}+
	|f(0)| + L\sqrt{m}d(N)\},
\end{align*}
where we used Cauchy-Schwarz and that $0\in\Omega$ for sufficiently large 
$N$. Now note that $(\EE|f(g)|^2)^{\frac{1}{2}} \lesssim 
1+\max_k(\EE\|H_k^N\|^{2q})^{\frac{1}{2q}}$ by H\"older's inequality, with 
a universal constant depending on $p$ only. It therefore follows from 
Corollary \ref{cor:norm} that $(\EE|f(g)|^2)^{\frac{1}{2}}$ is uniformly
bounded as $N\to\infty$. As $|f(0)|$ is clearly also uniformly bounded,
the estimate $\mathbf{P}[\Omega^c]\le me^{-(\log d(N))^3}$ implies that
$$
	|\EE[f(g)]-\EE[\tilde f(g)]|=o(1)
$$
as $N\to\infty$. On the other hand, we can compute
\begin{align*}
	\mathbf{P}[|f(g)-\EE[\tilde f(g)]|\ge L\log N]
	&\le
	\mathbf{P}[\Omega^c] +
	\mathbf{P}[|\tilde f(g)-\EE[\tilde f(g)]|\ge L\log N]
	\\ &\le
	m e^{-(\log N)^3} + 2e^{-\frac{(\log N)^2}{2}}
\end{align*}
by Lemma \ref{lem:gconc} and $d(N)\ge N$. Thus
$$
	|f(g)-\EE[f(g)]| \le L\log N+o(1)
$$
eventually as $N\to\infty$ a.s.\ by the Borel-Cantelli lemma. But as
$\sigma_*(H_k^N)\le v(H_k^N)=o((\log N)^{-\frac{3}{2}})$, we
have $L\log N=o(1)$ as $N\to\infty$, and the proof is complete.
\end{proof}

We can now complete the proof of Theorem \ref{thm:free}.

\begin{proof}[Proof of Theorem \ref{thm:free}: part \textit{b}]
Theorem \ref{thm:htlin} and Lemmas \ref{lem:saffirst} and 
\ref{lem:safnc} yield
$$
	\limsup_{N\to\infty}\|p(H_1^N,\ldots,H_m^N)\|
	\le \|p(s_1,\ldots,s_m)\|\quad\mbox{a.s.}
$$
for every noncommutative polynomial $p$. On the other hand, combining
part \textit{a} of Theorem \ref{thm:free} with Lemma \ref{lem:safconc} 
yields that
$$
	\lim_{N\to\infty} \ntr p(H_1^N,\ldots,H_m^N) =
        \tau(p(s_1,\ldots,s_m))\quad\mbox{a.s.}
$$
The latter implies
\begin{align*}
	\liminf_{N\to\infty}\|p(H_1^N,\ldots,H_m^N)\|
	&\ge
	\liminf_{N\to\infty}\ntr(|p(H_1^N,\ldots,H_m^N)|^{2r})^{\frac{1}{2r}}
	\\ &=
	\tau(|p(s_1,\ldots,s_m)|^{2r})^{\frac{1}{2r}}
\end{align*}
a.s.\ for every $r\in\mathbb{N}$, where we used that 
$|p(H_1^N,\ldots,H_m^N)|^{2r}$
is again a noncommutative polynomial. Letting $r\to\infty$ shows that
$$
	\lim_{N\to\infty}\|p(H_1^N,\ldots,H_m^N)\|
	= \|p(s_1,\ldots,s_m)\|\quad\mbox{a.s.}
$$
It remains to note that
$$
	\lim_{N\to\infty}\EE\|p(H_1^N,\ldots,H_m^N)\|
	= \|p(s_1,\ldots,s_m)\|
$$
now follows from Lemma \ref{lem:safconc}.
\end{proof}

\subsection{Proof of Corollary \ref{cor:oldie}}
\label{sec:pfcoroldie}

We finally deduce Corollary \ref{cor:oldie}.

\begin{proof}[Proof of Corollary \ref{cor:oldie}]
Applying Theorem \ref{thm:free} to $p(H^N)=(H^N)^r$ yields
$$
	\lim_{N\to\infty}\|H^N\|=\|s\|\quad\mbox{and}\quad
	\lim_{N\to\infty}\ntr[(H^N)^r] = \tau(s^r)\quad\mbox{a.s.}
$$
for every $r\in\mathbb{N}$, where $s$ is a semicircular variable. As
$$
	\ntr[(H^N)^r] =
	\int x^r\,\mu_{H^N}(dx),\qquad
	\tau(s^r) = \int x^r \,\mu_{\rm sc}(dx),
$$
and as $\mu_{\rm sc}$ has bounded support, the first conclusion follows as 
moment convergence implies weak convergence \cite[p.\ 116]{NS06}. The
second conclusion follows as $\|s\|=2$.
\end{proof}

\section{Discussion and further questions}
\label{sec:disc}

The aim of this final section is to discuss a number of broader questions 
that arise from our main results. We first discuss in some detail to what 
extent the parameter $v(X)$ that quantifies noncommutativity in our bounds 
is natural, and whether one might hope to improve fundamentally on this 
parameter. We then proceed to highlight a number of further questions that 
arise from our results.

\subsection{A canonical parameter 
\texorpdfstring{$\sigma_{**}(X)$}{sigma**(X)} cannot exist}

\subsubsection{Is \texorpdfstring{$v(X)$}{v(X)} a natural parameter?}

In all the results of this paper, the presence of noncommutativity and of 
``intrinsic freeness'' was quantified by the parameter $v(X)$. The utility 
of this parameter is amply demonstrated by the various examples in Section 
\ref{sec:ex}: for example, in the independent entry model, 
$v(X)\asymp\max_{ij}b_{ij}$ recovers precisely the small parameter that 
controls the previously known behavior \eqref{eq:bvh} in this setting, 
while various models in Section \ref{sec:exdep} illustrate the 
significance and near-optimality of our bounds in dependent situations.

Nonetheless, it is not difficult to find examples where both $v(X)$, and 
the slightly improved parameter $\sup_N w(X_1^N)$ discussed in Remark 
\ref{rem:tildew}, fail to capture the correct behavior of Gaussian random 
matrices. A particularly disconcerting aspect of these parameters is the 
following. Let $X$ be any random matrix of the form \eqref{eq:model}; then 
$X\otimes\id$ is again a model of this form, where we tensor on any 
finite-dimensional identity matrix. Tensoring on an identity clearly has 
no effect on the spectrum of the matrix: in particular, 
$\spc(X\otimes\id)=\spc(X)$ and $\sigma(X\otimes\id)=\sigma(X)$. This 
invariance fails dramatically, however, for the parameters $v(X)$ and 
$w(X)$.

\begin{lem}
\label{lem:disturbing}
Let $\id_N$ be the identity in $\M_N(\mathbb{C})$. Then for any
self-adjoint $d\times d$ random matrix $X$ of the form \eqref{eq:model},
we have
\begin{alignat*}{2}
	&v(X\otimes\id_N) = \sqrt{N}v(X)\quad &&\mbox{for }N\ge 1,
	\\
	&w(X\otimes\id_N) = \sigma(X)\quad &&\mbox{for }N\ge d.
\end{alignat*}
\end{lem}

\begin{proof}
We have $\mathrm{Cov}(X\otimes A)=\mathrm{Cov}(X)\otimes 
\iota(A)\iota(A)^*$ for any deterministic matrix $A$, where 
$\iota:\M_d(\mathbb{C})\to\mathbb{C}^{d^2}$ maps a matrix to its vector of 
entries. Thus $v(X\otimes A)^2 = 
v(X)^2\|A\|_{\rm HS}^2$, and the first claim follows as $\|\id_N\|_{\rm 
HS}=\sqrt{N}$.

To prove the second claim, let $N\ge d$, and define 
$U\in\M_d(\mathbb{C})\otimes\M_N(\mathbb{C})$ by
$U(e_i\otimes e_j)=e_j\otimes e_i$ for $i,j\in[d]$ and
$U(e_i\otimes e_j)=0$ otherwise.
Then $\|U\|=1$ and
$$
	\sum_{i,j}(A_i\otimes\id)U(A_j\otimes\id)U(A_i\otimes\id)
	U(A_j\otimes\id)U =
	\sum_i A_i^2 \otimes P\bigg(
	\sum_i A_i^2 \bigg)P^*,
$$
where $P:\mathbb{C}^d\to\mathbb{C}^N$ denotes the canoncial
embedding $Pe_i=e_i$. Thus $w(X\otimes\id)\ge\sigma(X)$ by the last 
equation display in the proof of Lemma \ref{lem:complexint}.
On the other hand,
$w(X\otimes\id)\le\sigma(X\otimes\id)=\sigma(X)$ by 
\cite[Proposition 3.2]{Tro18}.
\end{proof}

Lemma \ref{lem:disturbing} shows that no matter how well our bounds 
capture the behavior of the random matrix $X$, applying our results to 
$X\otimes\id_d$ can never yield any improvement over the noncommutative 
Khintchine inequlity \eqref{eq:nck}---despite that tensoring an identity 
has no effect on the spectrum of the matrix. This observation may lead one 
to conjecture that the theory of this paper should admit a far-reaching 
improvement, in which $v(X)$ is replaced by a ``natural'' parameter that 
captures correctly the behavior of the spectrum. For example, it was 
conjectured in \cite{Tro18,vH17,Ban15} that there exist bounds of the kind 
that are studied in this paper, where the parameter $v(X)$ is replaced by 
the ``natural'' parameter $\sigma_*(X)$.

Somewhat surprisingly, such conjectures turn out to be ill-founded. We 
will presently show that the kind of behavior that is captured by Lemma 
\ref{lem:disturbing} is a fundamental feature of any bound of the form 
\eqref{eq:secondorder}.

\subsubsection{An impossibility theorem}

Suppose we are given a matrix parameter $\sigma_{**}(X)$ such that the
inequality for $d\times d$ centered Gaussian random matrices 
\begin{equation}
\label{eq:seconduniv}
	\EE\|X\|\le C\sigma(X) + C\sigma_{**}(X)(\log d)^{\beta}
\end{equation}
is valid for universal constants $C,\beta>0$.
In view of the above discussion, we may aim to find an inequality 
\eqref{eq:seconduniv} that respects the simplest properties of the 
spectral norm: the triangle inequlity $\|X+Y\|\le\|X\|+\|Y\|$; unitary 
invariance $\|U^*XU\|=\|X\|$; and tensor invariance 
$\|X\otimes\id\|=\|X\|$. Note that all three properties are satisfied also 
by the parameter $\sigma(X)$. In order for 
\eqref{eq:seconduniv} to respect these properties, one would 
have to assume that the parameter $\sigma_{**}(X)$ satisfies these 
properties up to a universal constant. Let us formalize these requirements 
as follows:
\begin{enumerate}
\itemsep\abovedisplayskip
\item $\sigma_{**}(X_1+X_2)\le C'\{\sigma_{**}(X_1)+\sigma_{**}(X_2)\}$.
\item $\sigma_{**}(U^*XU)\le C'\sigma_{**}(X)$ for any
non-random unitary matrix $U$.
\item $\sigma_{**}(X\otimes\id_N) \le C'\sigma_{**}(X)$ for
any $N\in\mathbb{N}$.
\end{enumerate}
Here $C'$ always denotes a universal constant.

The noncommutative Khintchine inequality \eqref{eq:nck}, which 
corresponds to the case $\sigma_{**}(X)=\sigma(X)$, satisfies all the 
above requirements but does not capture any noncommutativity. We therefore 
introduce as a further assumption that the second term of
\eqref{eq:seconduniv} becomes negligible at least in the simplest model of 
random matrix theory, the standard Wigner matrices $G^N$ of Definition 
\ref{defn:wigner}.
\begin{enumerate}
\itemsep\abovedisplayskip
\setcounter{enumi}{3}
\item $\sigma_{**}(G^N) = o((\log N)^{-\beta})$ as $N\to\infty$.
\end{enumerate}
Remarkably, the above very natural properties prove to be
mutually contradictory.

\begin{prop}
\label{prop:impossible}
Suppose that \eqref{eq:seconduniv} is valid for some universal constants 
$C,\beta$. Then at least one of the properties (1)--(4) must fail for any 
choice of $C'$.
\end{prop}

\begin{proof}
Let $G_1^N,\ldots,G_n^N$ be i.i.d.\ standard Wigner matrices of dimension 
$N$, and consider the $N^n$-dimensional Gaussian random matrix
$$
	X_{n,N} =
	\sum_{k=1}^n
	\underbrace{\id_N\otimes\cdots\otimes\id_N}_{k-1}
	\mbox{}\otimes G_k^N
	\otimes	\mbox{}
	\underbrace{\id_N\otimes\cdots\otimes\id_N}_{n-k}.
$$
We will show that if properties (1)--(4) hold for some universal constant 
$C'\ge 1$, this entails a contradiction. Indeed, properties (1)--(3) 
yield
\begin{align*}
	\sigma_{**}(X_{n,N}) 
	& \stackrel{(1)}{\le}\,
	\sum_{k=1}^n (C')^k\,
	\sigma_{**}(\id_{N^{k-1}}\otimes
	G_k^N\otimes \id_{N^{n-k}}) \\
	& \stackrel{(2)}{\le}\,
	\sum_{k=1}^n(C')^{k+1}\,
	\sigma_{**}(G_k^N\otimes \id_{N^{n-1}}) \\
	& \stackrel{(3)}{\le}\,
	\sum_{k=1}^n (C')^{k+2}\,
	\sigma_{**}(G_k^N).
\end{align*}
while we may readily compute $\sigma(X_{n,N})=\sqrt{n}$. Thus we obtain
$$
	\limsup_{N\to\infty}
	\EE\|X_{n,N}\| \le C\sqrt{n}
$$
by \eqref{eq:seconduniv} and property (4).

On the other hand, the tensor product structure of $X_{n,N}$ implies that
$$
	\|X_{n,N}\| \ge
	\lambda_{\rm max}(X_{n,N}) =
	\sum_{k=1}^n \lambda_{\rm max}(G_k^N)
$$
pointwise, where $\lambda_{\rm max}$ denotes the maximal eigenvalue.
We therefore obtain
$$
	2n\le
	\limsup_{N\to\infty}
        \EE\|X_{n,N}\| \le C\sqrt{n}
$$
by Corollary \ref{cor:oldie}.
As $n$ is arbitrary, this yields the desired contradiction.
\end{proof}

A special case of Proposition \ref{prop:impossible} disproves the 
conjecture made in 
\cite{Tro18,vH17,Ban15}: the parameter $\sigma_*(X)$ satisfies all four 
properties (1)--(4), and thus an inequality of the form 
\eqref{eq:seconduniv} with $\sigma_{**}(X)=\sigma_*(X)$ cannot hold.

More generally, Proposition \ref{prop:impossible} shows that no parameter 
$\sigma_{**}(X)$ can be expected to avoid the kind of ``unnatural'' 
behavior that was identified in Lemma \ref{lem:disturbing}. The 
construction in the proof of Proposition \ref{prop:impossible} suggests a 
clear explanation of why this must be the case. The summands in the 
definition of $X_{n,N}$ behave as independent variables in the classical 
(commutative) sense, as opposed to free independence. However, if 
properties (1)--(4) hold, such models can give rise to a small parameter 
$\sigma_{**}(X)$, so that \eqref{eq:seconduniv} would imply that they 
behave as their free counterparts up to a universal constant. These two 
phenomena stand in contradiction.

\subsubsection{The dimension threshold}

The second identity of Lemma \ref{lem:disturbing} shows that our results 
fail to capture any noncommutative behavior when we tensor a random matrix 
$X$ by an identity of the same dimension. On the other hand, for standard 
Wigner matrices $G^N$, we have $\sigma(G^N\otimes\id_{D(N)})=1$ and
$$
	v(G^N\otimes\id_{D(N)}) \asymp 
	\sqrt{\frac{D(N)}{N}}\ll \sigma(G^N\otimes\id_{D(N)})
$$
as soon as $D(N)\ll N$. Thus the case where a random matrix is 
tensored by an identity of proportional dimension appears as 
the threshold at which our ability to capture ``intrinsic freeness'' 
breaks down.

This phenomenon has an unexpected connection to certain 
questions in the theory of operator algebras. In the rest of this section, 
let $G^N_1,\ldots,G^N_m,H^{N}_1,\ldots,H^{N}_m$ be independent GUE matrices 
(that is, self-adjoint $N\times N$ matrices with i.i.d.\ centered complex 
Gaussian variables of variance $\frac{1}{N}$ on and above the diagonal). 
In the recent work \cite{Hay20}, it was shown that if strong 
convergence
\begin{multline*}
	\lim_{N\to\infty}
	\|p(G^N_1\otimes\id_N,\ldots,
	G^N_m\otimes\id_N,
	\id_N\otimes H^N_1,\ldots,\id_N\otimes H^N_m)\| = \\
	\|p(s_1\otimes\id,\ldots,s_m\otimes\id,\id\otimes s_1,
	\ldots,\id\otimes s_m)\|\quad\mbox{a.s.}
\end{multline*}
were to hold for all polynomials $p$,\footnote{%
Throughout this section $\otimes$ always denotes the minimal tensor 
product of $C^*$-algebras.} this would settle a 
conjecture of Peterson and Thom in the theory of Von Neumann 
algebras. Using the results of this paper, a slightly weaker fact can be 
proved. As the following result is only tangentially related to 
the rest of this paper, we will sketch its proof.

\begin{prop}
\label{prop:hayes}
We have
\begin{multline*}
	\lim_{N\to\infty}
	\|p(G^N_1\otimes\id_{D(N)},\ldots,
	G^N_m\otimes\id_{D(N)},
	\id_N\otimes H^{D(N)}_1,\ldots,\id_N\otimes H^{D(N)}_m)\| = \\
	\|p(s_1\otimes\id,\ldots,s_m\otimes\id,\id\otimes s_1,
	\ldots,\id\otimes s_m)\|\quad\mbox{a.s.}
\end{multline*}
for every noncommutative polynomial $p$, provided 
$D(N)=o\big(\frac{N}{(\log N)^3}\big)$.
\end{prop}

While this does not suffice for the purpose of \cite{Hay20}, which 
requires $D(N)=N$, the result was previously known only for 
$D(N)=o(N^{\frac{1}{3}})$ \cite[Theorem 1.2]{CGP19}.\footnote{%
After the initial version of this paper appeared, a complete solution of
the Peterson-Thom conjecture was proposed in \cite{BC22} using 
methods specific to GUE matrices. We retain 
Proposition~\ref{prop:hayes} to illustrate what may be achieved by the 
completely general methods of this paper.}

\begin{proof}[Sketch of proof of Proposition \ref{prop:hayes}]
Fix a dimension $d'\in\mathbb{N}$ and self-adjoint matrices
$A_0,\ldots,A_m,B_1,\ldots,B_m\in\M_{d'}(\mathbb{C})_{\rm sa}$.
Define the random matrix
$$
	X^N =A_0\otimes\id_N\otimes\id_{D(N)}+
	\sum_{k=1}^m
	A_k\otimes G^N_k\otimes \id_{D(N)}+
	\sum_{k=1}^m
	B_k\otimes \id_N\otimes H^{D(N)}_k.
$$
The assumption on $D(N)$ implies that $v(\sum_{k=1}^m
A_k\otimes G^N_k\otimes \id_{D(N)})=o((\log N)^{-\frac{3}{2}})$. 
As $(G^N_k)_{k\le m}$ and $(H^{D(N)}_k)_{k\le m}$ are independent, we can 
apply Theorem \ref{thm:mainsp} conditionally
on $(H^{D(N)}_k)_{k\le m}$, Lemma \ref{lem:safncsimple}, and the 
Borel-Cantelli lemma to show that
$$
	\spc(X^N) \subseteq 
	\spc(A_0\otimes\id\otimes\id_{D(N)} +
	{\textstyle \sum_{k=1}^m}
	A_k\otimes s_k\otimes\id_{D(N)} +
	{\textstyle \sum_{k=1}^m} B_k\otimes \id\otimes H^{D(N)}_k)
	+[-\varepsilon,\varepsilon]
$$
eventually as $N\to\infty$ a.s.\ for every $\varepsilon>0$.

On the other hand, let $\mathcal{A}$ be the unital $C^*$-algebra generated
by $\{s_1,\ldots,s_m\}$. Then $\M_{d'}(\mathbb{C})\otimes\mathcal{A}$ is
an exact $C^*$-algebra, cf.\ \cite[p.\ 27]{Hay20} and the references 
therein. Therefore, \cite[Theorem 9.1]{HT05} and \cite[Proposition 
2.1]{CM14} imply that
\begin{multline*}
	\spc(A_0\otimes\id\otimes\id_{D(N)} +
	{\textstyle \sum_{k=1}^m}
	A_k\otimes s_k\otimes\id_{D(N)} +
	{\textstyle \sum_{k=1}^m} B_k\otimes \id\otimes H^{D(N)}_k)
	\subseteq
\\
	\spc(A_0\otimes\id\otimes\id +
	{\textstyle \sum_{k=1}^m}
	A_k\otimes s_k\otimes\id +
	{\textstyle \sum_{k=1}^m} B_k\otimes \id\otimes s_k)
	+[-\varepsilon,\varepsilon]
\end{multline*}
eventually as $N\to\infty$ a.s.\ for every $\varepsilon>0$. 
Linearization as in Theorem \ref{thm:htlin} yields
\begin{multline*}
	\limsup_{N\to\infty}
	\|p(G^N_1\otimes\id_{D(N)},\ldots,
	G^N_m\otimes\id_{D(N)},
	\id_N\otimes H^{D(N)}_1,\ldots,\id_N\otimes H^{D(N)}_m)\| \\
	\le\|p(s_1\otimes\id,\ldots,s_m\otimes\id,\id\otimes s_1,
	\ldots,\id\otimes s_m)\|\quad\mbox{a.s.}
\end{multline*}
for every noncommutative polynomial $p$. The reverse inequality follows 
from weak asymptotic freeness of $(G^N_k)_{k\le m}$ and 
$(H^{D(N)}_k)_{k\le m}$ and concentration of measure as in the 
analogous part of the proof 
of Theorem \ref{thm:free}.
\end{proof}

\subsection{Further questions}

We conclude this paper by highlighting some basic questions that 
arise from our main results.

\subsubsection{Sharp inequalities}

As was explained in the previous section, there cannot exist a canonical 
inequality of the form \eqref{eq:secondorder} that captures correctly the 
structure of all Gaussian random matrices. However, even if we restrict 
attention to parameters such as $v(X)$, the main results of this paper 
fall slightly short of recovering the previously known results for the 
independent entry model: the logarithmic term $(\log d)^{\frac{3}{2}}$ in 
\eqref{eq:normindep} is slightly worse than the term $\sqrt{\log d}$ in 
\eqref{eq:bvhsharp}.

The power on the logarithm is relevant only for models that are 
right at the threshold where ``intrinsic freeness'' breaks 
down, and is insignificant in most applications. It is nonetheless an 
interesting question whether the results of this paper can be refined 
so that they recover previously known results such as \eqref{eq:bvhsharp}
as a special case. This would be the case, for example, if one could prove 
that
$$
	\EE\|X\| \,\stackrel{?}{\le}\, \|X_{\rm free}\| + Cv(X)\sqrt{\log d}.
$$
Corollary \ref{cor:norm} falls short of such a bound in two ways: it has a 
suboptimal power on the logarithm $(\log d)^{\frac{3}{4}}$, and 
it involves the parameter $\tilde v(X)$ rather than $v(X)$. (Replacing 
$\tilde v(X)$ by $\sup_N w(X_1^N)$, as in Remark 
\ref{rem:tildew}, would not suffice to recover the behavior of the 
independent entry model, cf.\ \cite[\S 3.8]{Tro18}.)

Somewhat surprisingly, however, it turns out that many results of 
this paper are already optimal even for the independent entry model. For 
example, if $X$ is a standard Wigner matrix of dimension $d$, then 
$\sigma(X)=1$ and $v(X) = 2^{1/2}d^{-1/2}$, so that Theorem 
\ref{thm:stieltjes} shows that the matrix Stieltjes transforms 
satisfy 
$$
	\|G(Z)-G_{\rm free}(Z)\| \lesssim d^{-1}\|(\mathrm{Im}\,Z)^{-5}\|.
$$
However, it is shown in \cite[Theorem 4.4]{Sch05} that the 
$d^{-1}$ rate is sharp in this example. Thus the conclusion of Theorem 
\ref{thm:stieltjes} is essentially optimal in this sense, and in 
particular it is impossible 
to replace $\tilde v(X)$ by $v(X)$ in this result. In fact, this 
optimality can be traced back to the most basic ingredient of the proofs 
in this paper: one may readily verify that in the example of a standard 
Wigner matrix 
$$
	\Bigg|\sum_{ij} \ntr[A_iA_jA_iA_j]\Bigg|\asymp \frac{1}{d},
$$
so that the bounds of Lemma \ref{lem:complexint} and Proposition 
\ref{prop:vxwx} are already the best possible.
In view of these examples, it seems likely that the general methods of 
this paper cannot be significantly improved by technical refinements 
alone: our methods show that the entire spectrum of $X$ behaves as that of 
$X_{\rm free}$, and do not enable us to observe a quantitative distinction 
between the bulk and edges of the spectrum.

\subsubsection{Universality}
\label{sec:universality}

Throughout this paper we have been primarily concerned with Gaussian 
random matrices, and our proofs make heavy use of Gaussian analysis. In 
contrast, classical matrix concentration inequalities \cite{Tro15} apply 
to much more general non-Gaussian models $X=\sum_{i=1}^n Z_i$, where $Z_i$ 
are arbitrary independent centered random matrices. It has long been 
known, however, that such non-Gaussian inequalities can be deduced from 
the corresponding Gaussian inequalities \cite{Rud99,Tro16}. The 
idea behind this approach is that a routine symmetrization argument yields
$$
	\EE\Bigg\|\sum_{i=1}^n Z_i\Bigg\|
	\le \sqrt{2\pi}\,\EE\Bigg\|\sum_{i=1}^n g_i Z_i\Bigg\|
$$
where $g_1,\ldots,g_n$ are i.i.d.\ standard real Gaussian variables that 
are independent of $Z_1,\ldots,Z_n$. If one conditions on the 
matrices $Z_i$ on the right-hand side, one is left with a Gaussian random 
matrix. This approach makes it possible to derive non-Gaussian 
inequalities, such as the widely used matrix Bernstein inequality, from 
the Gaussian noncommutative Khintchine inequality.

In a preprint version of this paper, we used the symmetrization approach 
to derive a non-Gaussian inequality from our main results, which
superficially resembles the bound of Theorem \ref{thm:mconcbvh}. 
Unfortunately, however, this approach proves to be unsatisfactory in the 
present setting for several reasons.
\begin{enumerate}[$\bullet$]
\itemsep\abovedisplayskip
\item The symmetrization method necessarily results in the loss of a 
universal constant. It is therefore unable to capture the sharp nature of 
our main results.
\item The symmetrization method can only be applied to 
convex functionals such as the spectral norm. It therefore does not 
provide access to other spectral statistics, such as the support of the 
spectrum or Stieltjes transforms.
\item In our context, symmetrization gives rise to
a term of the form $\max_i \tr[Z_i^2]^{\frac{1}{2}}$ that captures the 
deviation from Gaussianity. In contrast, the analogous quantity that 
arises in classical matrix concentration inequalities is $\max_i \|Z_i\|$, 
which can be much smaller. (This inefficiency arises from the quantity 
$v(X)$ in our bounds, whose definition involves Hilbert-Schmidt norms; cf.\ 
section \ref{sec:introg}.)
\end{enumerate}
Even if one were only interested in the norms of random matrices up to a 
universal constant, the last issue can be a severe limitation in 
applications.

The follow-up work \cite{BV22} resolves these issues by establishing a 
\emph{universality principle}, which yields sharp nonasymptotic bounds on 
the deviation of the spectrum of the non-Gaussian model $X=\sum_{i=1}^n 
Z_i$ from that of the Gaussian random matrix $G$ whose entries have the 
same mean and covariance as those of $X$. This makes it possible to obtain 
direct analogues of the main results of this paper for the independent sum 
model by applying the Gaussian bounds to $G$.

From a broader viewpoint, the results of the present paper and of 
\cite{BV22} suggest that the study of a broad class of random matrices can 
be separated into two largely independent problems: a universality 
principle, which shows that a non-Gaussian and Gaussian model behave 
alike; and the ``intrinsic freeness'' principle of the present paper, 
which relates the spectral properties of the Gaussian model to explicitly 
computable deterministic quantities in free probability theory. It is an 
interesting question whether there are general non-Gaussian models of 
random matrices, beyond the independent sum model, that admit analogous 
universality principles. When combined with the results of this paper, 
such principles would immediately give rise to new kinds of sharp matrix 
concentration inequalities.

\subsubsection{Reverse bounds on the spectrum}

The results of Section \ref{sec:mainstat} yield two-sided bounds on the 
spectral statistics of $X$ in terms of $X_{\rm free}$. In contrast, 
Section \ref{sec:mainspec} only yields one-sided bounds on the support of 
the spectrum: we show that $\spc(X)\subseteq\spc(X_{\rm 
free})+[-\varepsilon,\varepsilon]$ with high probability. When one is 
interested in asymptotics, the latter is usually the difficult 
direction, while the reverse inclusion follows rather easily from weak 
bounds on the spectral statistics. It is not clear, however, how to obtain 
nonasymptotic bounds of the form $\spc(X_{\rm free})\subseteq\spc(X)+ 
[-\varepsilon,\varepsilon]$.

To illustrate where the difficulty lies, let us derive a two-sided bound 
on the spectral norm $\|X\|$ from Theorem \ref{thm:moments}. As $X$ is a 
$d\times d$ matrix, we have 
$$
	d^{-\frac{1}{2p}}\|X\| \le
	(\ntr |X|^{2p})^{\frac{1}{2p}} \le \|X\|
$$
pointwise. Thus Theorem \ref{thm:moments} and Corollary \ref{cor:gconcmtx} 
yield
$$
	\EE\|X\| = (1+o(1)) \,
	({\ntr}\otimes\tau)(|X_{\rm free}|^{2p})^{\frac{1}{2p}}
	\quad\mbox{when}\quad
	\frac{v(X)}{\sigma(X)}\ll p^{-\frac{3}{2}}
	\ll (\log d)^{-\frac{3}{2}}.
$$
However, while $({\ntr}\otimes\tau)(|X_{\rm free}|^{2p})^{\frac{1}{2p}}\le
\|X_{\rm free}\|$ holds trivially, it is not clear how one can 
reverse this bound for $X_{\rm free}$.
Precisely the same issue arises in the proof of Theorem \ref{thm:mainsp}:
obtaining a reverse bound would require a lower bound on the moments of 
the resolvent of $X_{\rm free}$ (cf.\ Lemma \ref{lem:ressingle}).

Resolving this issue would require a quantitative understanding of the 
concentration of the mass of the spectral distribution of $X_{\rm free}$. 
Under restrictive model assumptions (``flatness''), the results of 
\cite{AEK20} provide a detailed study of the regularity of the spectral 
distribution. However, sufficiently precise quantitative bounds that are 
applicable to general random matrix models do not appear to be known to 
date.

\SkipTocEntry\subsection*{Acknowledgments}

M.T.B.\ was supported in part by NSF grant DMS-1856221, and the 
NSF-Simons Collaboration on Theoretical Foundations of Deep Learning. 
R.v.H.\ was supported in part by NSF grants DMS-1811735 and DMS-2054565, 
and the Simons Collaboration on Algorithms \& Geometry. The authors 
thank Benson Au, Tatiana Brailovskaya, Ioana Dumitriu, Antti Knowles, 
Gilles Pisier, Mark Rudelson, Dominik Schr\"oder, Joel Tropp, Nikita 
Zhivotovskiy, and Yizhe Zhu for interesting discussions, and the 
referees for helpful suggestions and feedback.

\bibliographystyle{abbrv}
\SkipTocEntry\bibliography{ref}

\end{document}